\documentclass[11pt]{amsart}

\usepackage{tikz}
\usepackage{latexsym}
\usepackage{graphicx}
\usepackage{enumerate}
\usepackage{amsmath, amssymb, amscd, mathrsfs}
\usepackage{xypic}
\usepackage[all]{xy}
\usepackage[pdftex,bookmarks=true,bookmarksnumbered=true]{hyperref}

\newtheorem*{thm*}{Theorem}

\DeclareMathOperator{\Kconv}{\bbK_{conv}}

\DeclareMathOperator{\CRD}{C_{\bbR}(\Delta)}

\DeclareMathOperator{\AffR}{Aff}

\newcommand{\LLLambda}{\Lambda_3}
\newcommand{\LLambda}{\Lambda_2}

\DeclareMathOperator{\Kchoq}{\bbK_{Choq}}

\DeclareMathOperator{\PPU}{PPU}

\newtheorem{theorem}{Theorem}[section]
\newtheorem{prop}[theorem]{Proposition}
\newtheorem{lemma}[theorem]{Lemma}
\newtheorem{thm}{Theorem}[section]
\newtheorem{conj}[thm]{Conjecture}

\newtheorem{problem}[theorem]{Problem}

\newtheorem{quest}[theorem]{Question}

\newtheorem{corollary}[theorem]{Corollary}

\theoremstyle{definition}
\newtheorem{definition}[theorem]{Definition}

\theoremstyle{remark}
\newtheorem{remark}[theorem]{Remark}
\theoremstyle{remark}

\numberwithin{equation}{section}

\def\R{{\mathbb R}}
\def\C{{\mathbb C}}

\def\M{{\mathbb M}}
\def\N{{\mathbb N}}

\def\K{{\mathbb K}}

\def\Q{{\mathbb Q}}

\DeclareMathOperator{\Span}{span} 
 \DeclareMathOperator{\id}{id}
 \DeclareMathOperator{\LIM}{LIM}
\DeclareMathOperator{\diag}{diag} \DeclareMathOperator{\Aff}{Aff}

\DeclareMathOperator{\length}{length}

\DeclareMathOperator{\sa}{sa}

\newcommand{\bbK}{{\mathbb K}}
\newcommand{\bbS}{{\mathbb S}}
\newcommand{\bbP}{{\mathbb P}}
\newcommand{\bbT}{{\mathbb T}}

\newcommand{\p}{\mathfrak p}
\newcommand{\q}{\mathfrak q}

\DeclareMathOperator{\rE}{E}
\DeclareMathOperator{\rR}{R}

\renewcommand{\land}{\wedge}

\newcommand{\cP}{{\mathcal P}}

\newcommand{\bbN}{{\mathbb N}}
\newcommand{\bbC}{\mathbb C}
\newcommand{\bbQ}{\mathbb Q}
\newcommand{\bbR}{\mathbb R}

\newcommand{\cY}{{\mathcal Y}}
\newcommand{\cZ}{{\mathcal Z}}

\newcommand{\cO}{{\mathcal O}}
\newcommand{\cU}{{\mathcal U}}

\newcommand{\rs}{\restriction}

\newcommand{\restrict}{\upharpoonright}

\DeclareMathOperator{\proj}{proj}

\DeclareMathOperator{\Proj}{Proj} 
\DeclareMathOperator{\Sa}{Sa}
\DeclareMathOperator{\Un}{Un}
\DeclareMathOperator{\Pos}{Pos}
\DeclareMathOperator{\Sym}{Sym}
\DeclareMathOperator{\im}{im}
\DeclareMathOperator{\Tensor}{Tensor}
\DeclareMathOperator{\Tenx}{Tensor_0}
\DeclareMathOperator{\Unit}{Unit}

\DeclareMathOperator{\State}{State}

\DeclareMathOperator{\Pure}{Pure} \DeclareMathOperator{\Trace}{Trace}

\newcommand{\cC}{\mathcal C}
\newcommand{\cF}{\mathcal F}

\newcommand{\cB}{\mathcal B}
\newcommand{\cK}{\mathcal K}

\newcommand{\e}{\varepsilon}

\DeclareMathOperator{\Mod}{Mod}
\DeclareMathOperator{\Aut}{Aut}
\DeclareMathOperator{\ran}{ran}
\DeclareMathOperator{\Ad}{Ad}
\DeclareMathOperator{\homo}{hom}
\DeclareMathOperator{\mono}{mono}

\DeclareMathOperator{\Rhom}{R_{hom}}
\DeclareMathOperator{\Rmono}{R_{mono}}
\DeclareMathOperator{\Riso}{R_{iso}}
\DeclareMathOperator{\Rdir}{R_{dir}}
\newcommand{\Rhomx}[1]{\rR_{\homo}^{#1}}
\newcommand{\Rmonox}[1]{\rR_{\mono}^{#1}}

\newcommand{\DeltaN}{\Delta^{\bbN}}

\DeclareMathOperator{\SA}{SA}
\DeclareMathOperator{\SAu}{SA_u}
\DeclareMathOperator{\SAuns}{SA_{uns}}
\DeclareMathOperator{\Au}{\mathfrak A_u}
\DeclareMathOperator{\Gammau}{\Gamma_u}
\DeclareMathOperator{\Gammaex}{\Gamma_{Exact}}
\DeclareMathOperator{\Gammaexu}{\Gamma_{Exact,u}}
\DeclareMathOperator{\hatGammau}{\hat\Gamma_u}

\DeclareMathOperator{\Xiu}{\Xi_{u}}

\DeclareMathOperator{\hatXiu}{\hat\Xi_{u}}

\DeclareMathOperator{\XiAuex}{\Xi_{\Au,Exact}}
\DeclareMathOperator{\Hom}{Hom}
\DeclareMathOperator{\End}{End}

\title[Turbulence, orbit equivalence, and the classification of nuclear C$^*$-algebras]
{Turbulence, orbit equivalence, and the \\ classification of nuclear C$^*$-algebras}
\author{Ilijas Farah, Andrew S. Toms and Asger T\"ornquist}

\address{Department of Mathematics and Statistics\\
York University\\
4700 Keele Street\\
North York, Ontario\\ Canada, M3J 1P3\\
and Matematicki Institut, Kneza Mihaila 34, Belgrade, Serbia}

\email{ifarah@yorku.ca}
\urladdr{http://www.math.yorku.ca/$\sim$ifarah}

\address{Department of Mathematics, Copenhagen University, Universitetsparken 5, 2100 K\o benhavn \O, Denmark}
\email{atornqui@gmail.com}

\address{Department of Mathematics, Purdue University, 150 N University St., West Lafayette IN 47906, USA}
\email{atoms@purdue.edu}

\date{\today}
\begin{document}

\begin{abstract}
We bound the Borel cardinality of the isomorphism relation for nuclear simple separable C$^*$-algebras: It is turbulent, yet Borel reducible to the action of the automorphism group of the Cuntz algebra $\mathcal{O}_2$ on its closed subsets.  The same bounds are obtained for affine homeomorphism of metrizable Choquet simplexes.  As a by-product we recover a result of Kechris and Solecki, namely, that homeomorphism of compacta in the Hilbert cube is Borel reducible to a Polish group action.  These results depend intimately on the classification theory of nuclear simple C$^*$-algebras by $\mathrm{K}$-theory and traces. Both of necessity and in order to lay the groundwork for further study on the Borel complexity of C$^*$-algebras, we prove that many standard C$^*$-algebra constructions and relations are Borel, and we prove Borel versions of Kirchberg's $\mathcal{O}_2$-stability and embedding theorems. We also find a C$^*$-algebraic witness for a $K_\sigma$ hard equivalence relation.
\end{abstract}

\maketitle


\begin{center}{\it The authors dedicate this article to the memory of Greg Hjorth. } \end{center}

\section{Introduction}\label{intro}

The problem of classifying a category of objects by assigning objects of another category as complete invariants is fundamental to many disciplines of mathematics. This is particularly true in C$^*$-algebra theory, where the problem of classifying the nuclear simple separable C$^*$-algebras up to isomorphism is a major theme of the modern theory.  Recent contact between descriptive set theorists and operator algebraists has highlighted two quite different views of what it means to have such a classification. Operator algebraists have concentrated on finding complete invariants which are assigned in a functorial manner, and for which there are good computational tools ($\mathrm{K}$-theory, for instance.) Descriptive set theorists, on the other hand, have developed an abstract \emph{degree theory} of classification problems, and have found tools that allow us to compare the complexity of different classification problems, and, importantly, allow us to rule out the use of certain types of invariants in a complete classification of highly complex concrete classification problems.

The aim of this paper is to investigate the complexity of the classification problem for nuclear simple separable C$^*$-algebras from the descriptive set theoretic point of view. 
A minimal requirement of any reasonable classification is that the invariants are somehow definable or calculable from the objects being classified themselves. For example, it is easily seen that there are at most continuum many non-isomorphic separable C$^*$-algebras, and so it is possible, in principle, to assign to each isomorphism class of separable C$^*$-algebras a unique real number, thereby classifying the separable C$^*$-algebras completely up to isomorphism. Few mathematicians working in C$^*$-algebras would find this a satisfactory solution to the classification problem for separable C$^*$-algebras, let alone nuclear simple separable C$^*$-algebras, since we do not obtain a way of computing the invariant, and therefore do not have a way of effectively distinguishing the isomorphism classes.

Since descriptive set theory is the theory of definable sets and functions in Polish spaces, it provides a natural framework for a theory of classification problems. In the past 30 years, such an abstract theory has been developed. This theory builds on the fundamental observation that in most cases where the objects to be classified are themselves  either countable or separable, there is a natural standard Borel space which parameterizes (up to isomorphism) all the objects in the class. From a descriptive set theoretic point of view, a classification problem is therefore a pair $(X,E)$ consisting of a standard Borel space $X$, the (parameters for) objects to be classified, and an equivalence relation $E$, the relation of isomorphism among the objects in $X$. In most interesting cases, the equivalence relation $E$ is easily definable from the elements of $X$, and is seen to be Borel or, at worst, analytic.

\begin{definition}

Let $(X,E)$ and $(Y,F)$ be classification problems, in the above sense. A \emph{Borel reduction} of $E$ to $F$ is a Borel function $f:X\to Y$ such that
$$
xEy\iff f(x) F f(y).
$$
If such a function $f$ exists then we say that $E$ is \emph{Borel reducible} to $F$, and we write $E\leq_B F$.
\end{definition}

If $f$ is a Borel reduction of $E$ to $F$, then evidently $f$ provides a complete classification of the points of $X$ up to $E$ equivalence by an assignment of $F$ equivalence classes. The ``effective'' descriptive set theory developed in 1960s and 1970s (see e.g. \cite{moschovakis80}) established in a precise way that the class of Borel functions may be thought of as a very general class of calculable functions. Therefore the notion of Borel reducibility provides a natural starting point for a systematic theory of classification which is both generally applicable, and manages to ban the trivialities provided by the Axiom of Choice.  Borel reductions in operator algebras have been studied in the recent work of Sasyk-T\"ornquist \cite{sato09a,sato09b,sato09c}, who consider the complexity of isomorphism for various classes of von Neumann factors, and in that of Kerr-Li-Pichot \cite{kelipi} and Farah \cite{Fa:Dichotomy}, who concentrate on certain representation spaces and, in \cite{kelipi},  group actions on the hyperfinite $\mathrm{II}_1$ factor.  This article initiates the study of Borel reducibility in separable C$^*$-algebras.

In \cite{Kec:C*}, Kechris introduced a standard Borel structure on the space of separable C$^*$-algebras, providing a natural setting for the study of the isomorphism relation on such algebras.  This relation is of particular interest for the subset of (unital) nuclear simple separable C$^*$-algebras, as these are the focus of G. A. Elliott's long running program to classify such algebras via $\mathrm{K}$-theoretic invariants.  To situate our main result for functional analysts, let us mention that an attractive class of invariants to use in a complete classification are the countable structures type invariants, which include the countable groups and countable ordered groups, as well as countable graphs, fields, boolean algebras, etc. If $(X,E)$ is a classification problem, we will say that $E$ is \emph{classifiable by countable structures} if there is a Borel reduction of $E$ to the isomorphism relation for some countable structures type invariant.  If $(X,E)$ is not classifiable by countable structures, then it may still allow \emph{some} reasonable classification, in the sense that it is Borel reducible to the orbit equivalence relation of a Polish group action on a standard Borel space.  Our main result is the following theorem (which is proved in \S\ref{borelreduction} and in \S\ref{s.bga}): 

\begin{theorem}\label{mainintro}
The isomorphism relation $E$ for unital simple separable nuclear C$^*$-algebras is turbulent, hence not classifiable by countable structures.  
Moreover, if $\mathcal L$ is any countable language and $\simeq^{\Mod(\mathcal L)}$ denotes the isomorphism relation for countable models of $\mathcal L$, then $\simeq^{\Mod(\mathcal L)}$ is Borel reducible to $E$. 
On the other hand, $E$ is Borel reducible to the orbit equivalence relation of a Polish group action, namely, the action of $\mathrm{Aut}(\mathcal{O}_2)$ on the closed subsets of $\mathcal{O}_2$.
\end{theorem}

\noindent
This establishes that the isomorphism problem for nuclear simple separable unital C$^*$-algebras does not have the maximal complexity among analytic classification problems, and rules out the usefulness of some additional types of invariants for a complete classification of nuclear simple separable unital C$^*$-algebras. It also establishes that this relation has higher complexity than the isomorphism relation of any class of countable structures.  
Remarkably, establishing both the lower and upper $\leq_B$ bounds of Theorem \ref{mainintro} requires that we prove Borel versions of two well-known results from Elliott's $\mathrm{K}$-theoretic classification program for nuclear simple separable C$^*$-algebras.  The lower bound uses the classification of the unital simple approximately interval (AI) algebras via their $\mathrm{K}_0$-group and simplex of tracial states, while the upper bound requires that we prove a Borel version of Kirchberg's Theorem that a simple unital nuclear separable C$^*$-algebra satisfies $A \otimes \mathcal{O}_2 \cong \mathcal{O}_2$.

By contrast with Theorem \ref{mainintro}, we shall establish in \cite{FaToTo} that Elliott's classification of unital AF algebras via the ordered $\mathrm{K}_0$-group amounts to a classification by countable structures.  This will follow from a more general result regarding the Borel computability of the Elliott invariant.  We note that there are non-classification results in the study of simple nuclear C$^*$-algebras which rule out the possibility of classifying all simple nuclear separable C$^*$-algebras via the Elliott invariant in a functorial manner (see \cite{Ror} and \cite{Toms}).  At heart, these examples exploit the structure of the Cuntz semigroup, an invariant whose descriptive set theory will be examined in \cite{FaToTo}.

The proof of Theorem \ref{mainintro} allows us to draw conclusions about the complexity of metrizable Choquet simplexes, too.
\begin{theorem}\label{mainchoquet}
The relation of affine homeomorphism on metrizable Choquet simplexes is turbulent, yet Borel reducible to the orbit equivalence relation of a Polish group action.
\end{theorem}
\noindent
Furthermore, and again as a by-product of Theorem \ref{mainintro}, we recover an unpublished result of Kechris and Solecki:
\begin{theorem}[Kechris-Solecki, 2006]\label{T.KS}
The relation of homeomorphism on compact subsets of the Hilbert cube is Borel reducible to the orbit equivalence relation of a Polish group action.
\end{theorem}


Finally, we show that a Borel equivalence relation which is {\it not} Borel reducible to any orbit equivalence relation of a Polish group action
has a C*-algebraic witness. 
Recall that $E_{K_\sigma}$ is the complete $K_\sigma$ equivalence relation. 

\begin{theorem}\label{biembeda}
$E_{K_\sigma}$ is Borel reducible to 
bi-embeddability of unital AF algebras.
\end{theorem}

The early sections of this paper are dedicated to establishing that a variety of standard constructions in C$^*$-algebra theory are Borel computable, and that a number of important theorems in C$^*$-algebra theory have Borel computable counterparts.  This is done both of necessity---Theorems \ref{mainintro}--\ref{biembeda} depend on these facts---and to provide the foundations for a general theory of calculability for constructions in C$^*$-algebra theory.  Constructions that are shown to be Borel computable include passage to direct limits, minimal tensor products, unitization, and the calculation of states, pure states, and traces.  Theorems for which we establish Borel counterparts include Kirchberg's Exact Embedding Theorem, as well as Kirchberg's $A\otimes\mathcal O_2\simeq\mathcal O_2$ Theorem for unital simple separable nuclear $A$. 

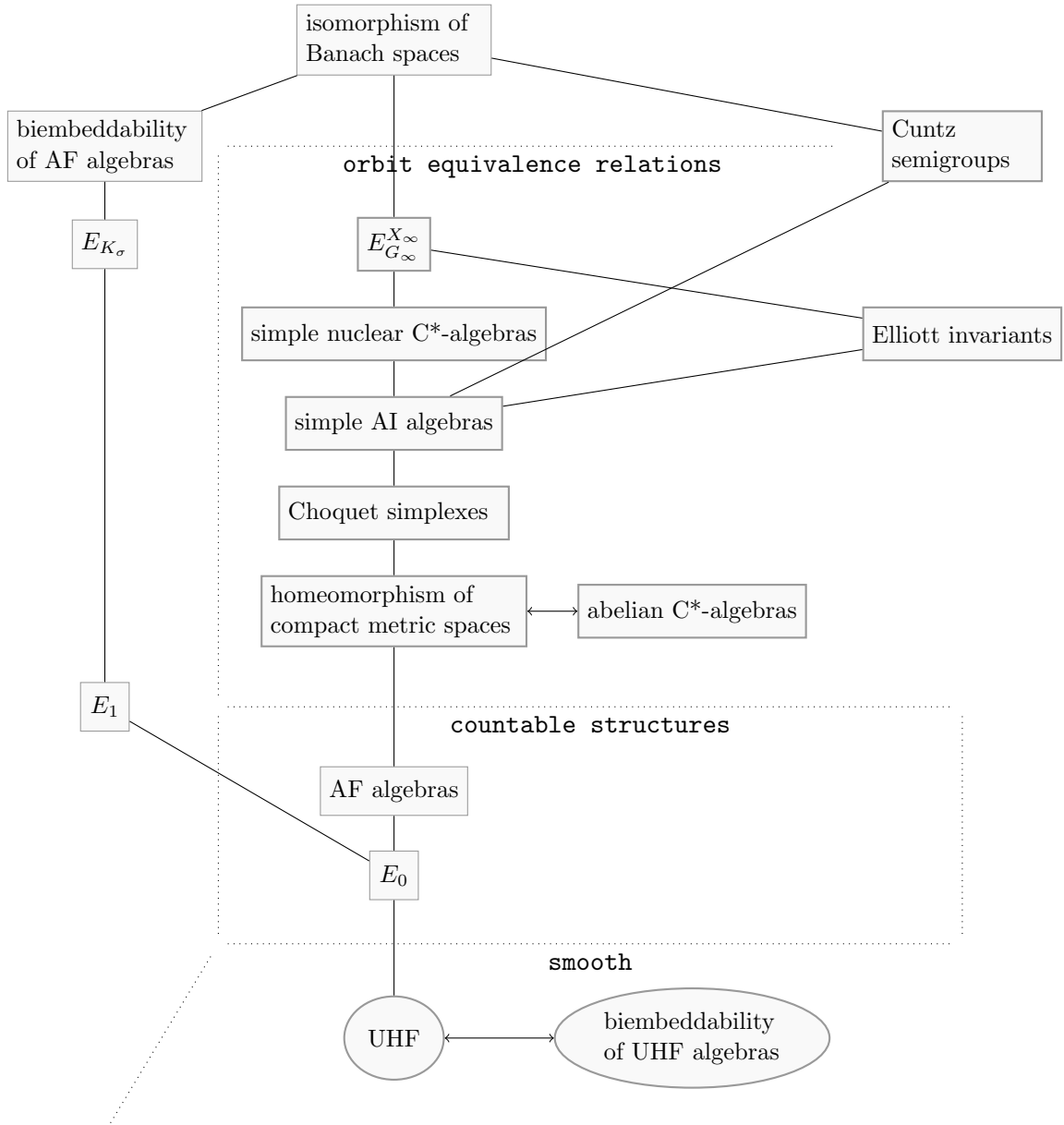
\begin{figure}
\label{Figure1}
\usetikzlibrary{shapes}

\tikzstyle{smooth}=[ellipse,
										font={\small},
                                                thick,
                                                text centered, 
                                                minimum size=1.2cm,
                                                draw=gray!80,
                                                fill=gray!05]

\tikzstyle{dark-side}=[rectangle,
										font={\small},
                                                thin,
                                                minimum size=.7cm,
                                                draw=gray!80,
                                                fill=gray!05]

\tikzstyle{unknown}=[rectangle,
										font={\small},
                                                thick,
                                                minimum size=.3in,
                                                draw=gray!80,
                                                fill=gray!05]

\tikzstyle{orbit}=[rectangle,
										font={\small},
                                                thick,
                                                minimum size=.3in,
                                                draw=gray!80,
                                                fill=gray!05]

\tikzstyle{pioneone}=[rectangle,
										font={\small},
                                                thick,
                                                minimum size=1.2cm,
                                                draw=gray!80,
                                                fill=gray!30]

\tikzstyle{ctble}=[rectangle,
										font={\small},
                                                thick,
                                                minimum size=1.2cm,
                                                draw=gray!80,
                                                fill=gray!05]

\noindent\begin{tikzpicture}
  \matrix[row sep=0.5cm,column sep=0.1cm] {
&& & \node  (Banach)[dark-side] {\parbox{1in}{isomorphism of\\ Banach spaces}}; & & 
\\
 \node (biAF)[dark-side] {\parbox{1in}{biembeddability of AF algebras}}; 
& \node (orbit-l) {}; 
& &  & 
 &  \node (orbit-r) {}; & 
\node  (Cuntz)[unknown] {\parbox{.8in}{Cuntz\\ semigroups}}; 
\\
\node (EKsigma)[dark-side] {$E_{K_\sigma}$}; 
& & & 
\node (max-orbit)[orbit] {$E^{X_\infty}_{G_\infty}$};
\\
&& & 
\node(nuclear)[orbit] {simple nuclear C*-algebras};
& & &
 \node(Ell)[orbit]{Elliott invariants}; 
\\
& 
& 
&
\node (AI)[orbit] {simple AI algebras}; & 
\\
& &  & \node (Choquet)[orbit] {\parbox{1.2in}{Choquet simplexes}}; 
& & 
\\
&&  
& \node  (cpct)[orbit]{\parbox{1.4in}{ {homeomorphism of\\ compact metric spaces}}}; 
&\node (abelian) [orbit]{abelian C*-algebras};
& 
\\ 
\node (Eone)[dark-side]{$E_1$};
& \node (ctble-l) {}; 
& & &  & & \node  (ctble-r) {};
\\
& &  &  \node (AF)[dark-side]{AF algebras} ;
\\
& & & 
\node (Ezero) [dark-side] {$E_0$};
& 
& 
\\
&  \node  (smooth-l) {}; 
& & &  & & 
\node (smooth-r) {};
\\
& & &  \node (UHF) [smooth]{UHF};
&  
\node (biUHF)[smooth] {\parbox{1in}{biembeddability of UHF algebras}};
\\
\node (bottom-l) {};
\\
};
\path (Banach)edge  (biAF);
\path(AI) edge (Choquet); 
\path[double](abelian) edge[<->] (cpct); 
\path (Ezero) edge (AF);
\path (AF) edge (cpct);
\path (Eone) edge  (Ezero); 
\path (Ezero) edge (UHF); 
\path (Eone) edge (EKsigma);
\path (EKsigma) edge (biAF);
\path (AI) edge (nuclear);
\path (AI) edge (Cuntz);
\path (AI) edge (Ell);
\path (nuclear) edge (max-orbit);
\path(max-orbit) edge (Banach); 
\path (Cuntz) edge (Banach);
\path (Ell) edge (max-orbit);
 \path[double] (UHF) edge[<->](biUHF); 
\path  (cpct) edge (Choquet); 
\path[dotted] (orbit-l) edge node[below]{\tt orbit equivalence relations} (orbit-r); 
\path[dotted](orbit-l) edge (ctble-l);
\path[dotted](ctble-l) edge (smooth-l);
\path[dotted](ctble-r) edge (smooth-r);
\path[dotted] (ctble-l) edge node[below]{\tt countable structures} (ctble-r); 
\path[dotted] (smooth-l) edge node[below]{\tt smooth} (smooth-r); 
\path[dotted] (bottom-l) edge (smooth-l); 
\end{tikzpicture}
\caption{Borel reducibility diagram}
\end{figure}

Figure~\ref{Figure1} summarizes the Borel reductions we obtain in this article, in addition to some known reductions. All classes of C*-algebras occurring in the diagram are unital and separable. 
Unless otherwise specified, the equivalence relation on a given class is the isomorphism relation. 
The bi-reducibility between the isomorphism for UHF algebras and bi-embeddability of UHF algebras
is an immediate consequence of Glimm's characterization of UHF algebras, or rather of its (straightforward) Borel version. 
$E_0$ denotes the eventual equality relation in the space $2^{\N}$. 
The fact that $E_0$ is the minimal non-smooth Borel equivalence relation is 
the Glimm--Effros dichotomy, proved by Harrington, Kechris and Louveau (see  \cite{Hj:Borel}). 
$E_1$ denotes the eventual equality relation in $[0,1]^\N$. 
By \cite{KeLou:Structure} $E_1$ is not Borel-reducible to any Polish group action. 
$E_{K_\sigma}$ is the complete $K_\sigma$ equivalence relation, and 
$E^{X_\infty}_{G_\infty}$ is the maximal orbit equivalence relation of a Polish group action (see \cite{Hj:Borel}). 
The nontrivial direction of Borel bi-reducibility between abelian C*-algebras and compact metric spaces follows from 
 Lemma~\ref{L.SPT}.
A Borel reduction from compact metric spaces to Choquet simplexes is given in   Lemma~\ref{L.Choq.4}.
A Borel reduction from Choquet simplexes to simple AI algebras is given in 
Corollary~\ref{c.reduction}. 
The Borel version of Elliott's reduction of simple AI algebras to Elliott invariant follows from Elliott's classification result and
the fact that the computation 
of the Elliott invariant is Borel, proved in  \cite{FaToTo}.
The reduction of the Elliott invariant to $E^{X_\infty}_{G_\infty}$, as well as the facts about the Cuntz semigroup, is proved in 
 \cite{FaToTo}. 
Bi-embeddability of AF algebras is proved to be above $E_{K_\sigma}$ in Section~\ref{S.biembed}.
The isomorphism of separable Banach spaces is the complete analytic equivalence relation by \cite{FeLouRo}. Some of the reductions in Figure~\ref{Figure1} are not known to be sharp. 
For example, it is not known whether the homeomorphism of compact metric spaces is
equireducible with $E^{X_\infty}_{G_\infty}$ (cf. remark at the end of \S\ref{S.problems}).

The standard reference for descriptive set theory is \cite{Ke:Classical} and specific facts about Borel-reducibility can 
be found in \cite{hjorth00, Hj:Borel} and \cite{gao09}. 
The general theory of C*-algebras can be found in \cite{blackadar} and references for Elliott's classification 
program are \cite{Ror:Classification} and \cite{et}.

The paper is organized as follows:  In section \ref{S.parameterize} we introduce a notion of standard Borel parameterization of a category of objects, and define several equivalent parameterizations of the class of separable C$^*$-algebras.  In section \ref{S.Basicdef} we prove that most standard constructions in C$^*$-algebra theory correspond to Borel functions and relations. 
In Section \ref{S.choquet} we give a parameterization of the set of metrizable separable Choquet simplexes. In section \ref{borelreduction} we establish the lower bound of Theorem \ref{mainintro}.  A Borel version of Kirchberg's Exact Embedding Theorem is obtained in section \ref{s.exact}.  The upper bound in Therorem \ref{mainintro} is proved in Section \ref{s.bga} using the Borel version of Kirchberg's $A\otimes\mathcal O_2\simeq\mathcal O_2$ Theorem;  Theorem \ref{mainchoquet} is also established. 
  Section \ref{S.biembed} establishes Theorem \ref{biembeda}, and Section \ref{S.problems} discusses several questions that remain open and warrant further investigation.

\subsection*{Acknowledgements} Ilijas Farah was partially supported by NSERC. A. Toms was supported by NSF grant DMS-0969246 and the 2011 AMS Centennial Fellowship.
Asger T\"ornquist wishes to acknowledge generous support from the following grants: The Austrian Science Fund FWF grant no. P 19375-N18, The Danish Council for Independent Research (Natural Sciences) grant no. 10-082689/FNU, and Marie Curie re-integration grant no. IRG- 249167, from the European Union.
We would like to thank the referee for a very useful report.


\section{Parameterizing separable C$^*$-algebras}\label{S.parameterize}
\label{S.Space}

In this section we describe several standard Borel spaces that in a natural way parameterize the set of all separable C$^*$-algebras. To make this precise, we adopt the following definition.


\begin{definition}\label{d.parameterization}
Let $\mathscr C$ be a category of objects.

\begin{enumerate}[\indent (1)]

\item A \emph{standard Borel parameterization} of $\mathscr C$ is a pair $(X,f)$ consisting of standard Borel space $X$ and a function $f:X\to\mathscr C$ such that $f(X)$ meets each isomorphism class in $\mathscr C$. (For brevity, we often simply call $(X,f)$ a \emph{parameterization} of $\mathscr C$.  We will also usually abuse notation by suppressing $f$ and writing $X$.  Finally, note that despite the terminology, it is $X$ rather than the parameterization that is standard Borel.)

\item The equivalence relation $\simeq^{(X,f)}$ on $X$ is defined by
$$
x\simeq^{(X,f)} y\iff f(x)\text{ is isomorphic to } f(y).
$$

\item A parameterization $(X,f)$ is called \emph{good} if $\simeq^{(X,f)}$ is analytic as a subset of $X\times X$.

\item Let $(X,f)$ and $(Y,g)$ be two parameterizations of the same category $\mathscr C$. A \emph{homomorphism} of $(X,f)$ to $(Y,g)$ is a function $\psi:X\to Y$ such that for $\psi(x)$ is isomorphic to $g(\psi(x))$ for all $x\in X$. An isomorphism of $(X,f)$ and $(Y,g)$ is a bijective homomorphism; a monomorphism is an injective homomorphism.

\item We say that $(X,f)$ and $(Y,g)$ are \emph{equivalent} if there is a Borel isomorphism from $(X,f)$ to $(Y,g)$.\label{it.equiv}

\item We say that $(X,f)$ and $(Y,g)$ are \emph{weakly equivalent} if there are Borel homomorphisms $\psi:X\to Y$ of $(X,f)$ to $(Y,g)$ and $\phi:Y\to X$ of $(Y,g)$ to $(X,f)$.
\end{enumerate}
\end{definition}

When $f$ is clear from the context, we will allow a slight \emph{abus de langage} and say that $X$ is a parameterization of $\mathscr C$ when $(X,f)$ is. Further, we will usually write $\simeq^X$ for $\simeq^{(X,f)}$. Note that by the Borel Schr\"oder-Bernstein Theorem (\cite[Theorem~15.7]{Ke:Classical}), (\ref*{it.equiv}) is equivalent to

\begin{enumerate}[\indent({\ref*{it.equiv}}')]
\item There are Borel monomorphisms $\psi:X\to Y$ of $(X,f)$ to $(Y,g)$ and $\phi: Y\to X$ of $(Y,g)$ to $(X,f)$.
\end{enumerate}

We now introduce four different parameterizations of the class of separable C$^*$-algebras, which we will later see are all (essentially) equivalent and good.


\subsection{The space $\Gamma(H)$.} \label{ss.gamma} 
Let $H$ be a separable infinite dimensional Hilbert space and let as usual $\cB(H)$ denote the space of bounded operators on $H$. The space $\cB(H)$ becomes a standard Borel space when equipped with the Borel structure generated by the weakly open subsets. Following \cite{Kec:C*} we let
$$
\Gamma(H)=\cB(H)^{\bbN},
$$
and equip this with the product Borel structure. Every $\gamma\in \Gamma(H)$ is identified with a sequence $\gamma_n$, for $n\in \bbN$, of elements of $\cB(H)$. 
For each $\gamma\in\Gamma(H)$ we let $C^*(\gamma)$ be the C$^*$-algebra generated by the sequence $\gamma$. If we identify each $\gamma\in\Gamma(H)$ with the $C^*(\gamma)$, then naturally $\Gamma(H)$ parameterizes all separable C$^*$-algebras acting on $H$. Since every separable C$^*$-algebra is isomorphic to a C$^*$-subalgebra of $\cB(H)$ this gives us a standard Borel parameterization of the category of all separable C$^*$-algebras. If the Hilbert space $H$ is clear from the context we will write $\Gamma$ instead of $\Gamma(H)$. Following Definition \ref{d.parameterization},  we define 
$$
\gamma\simeq^{\Gamma}\gamma'\iff C^*(\gamma)\text{ is isomorphic to } C^*(\gamma').
$$


\subsection{The space $\hat\Gamma(H)$.} \label{ss.hatgamma}
Let $\Q(i)=\Q+i\Q$ denote the complex rationals. Following \cite{Kec:C*}, let $(\p_j: j\in \N)$ enumerate the non-commutative $*$-polynomials without constant term in the formal variables $X_k$, $k\in\N$, with coefficients in $\Q(i)$, and for $\gamma\in \Gamma$ write $\p_j(\gamma)$ for the evaluation of $\p_j$ with $X_k=\gamma(k)$. Then $C^*(\gamma)$ is the norm-closure of $\{\p_j(\gamma): j\in \bbN\}$. The map $\Gamma\to\Gamma:\gamma\mapsto\hat\gamma$ where $\hat\gamma(j)=\p_j(\gamma)$ is clearly a Borel map from $\Gamma$ to $\Gamma$. If we let
$$
\hat\Gamma(H)=\{\hat\gamma:\gamma\in\Gamma(H)\},
$$
then $\hat\Gamma(H)$ is a standard Borel space and provides another parameterization of the C$^*$-algebras acting on $H$; we suppress $H$ and write $\hat\Gamma$ whenever possible. For $\gamma\in\hat\Gamma$, let $\check\gamma\in\Gamma$ be defined by 
$$
\check\gamma(n)=\gamma(i)\iff \p_i=X_n,
$$
and note that $\hat\Gamma\to\Gamma:\gamma\mapsto\check\gamma$ is the inverse of $\Gamma\to\hat\Gamma:\gamma\mapsto\hat\gamma$. We let $\simeq^{\hat\Gamma}$ be
$$
\gamma\simeq^{\hat\Gamma}\gamma'\iff C^*(\gamma)\text{ is isomorphic to } C^*(\gamma').
$$
It is clear from the above that $\Gamma$ and $\hat\Gamma$ are equivalent parameterizations.

An alternative picture of $\hat\Gamma(H)$ is obtained by considering the free (i.e., surjectively universal) countable unnormed $\Q(i)$-$*$-algebra $\mathfrak A$. We can identify $\mathfrak A$ with the set $\{\p_n:n\in\N\}$. Then
$$
\hat \Gamma_{\mathfrak A}(H)=\{f:\mathfrak A\to\mathcal B(H):f\text{ is a $*$-homomorphism}\}
$$
is easily seen to be a Borel subset of $\cB(H)^{\mathfrak A}$. For $f\in\hat\Gamma_{\mathfrak A}$ let $C^*(f)$ be the norm closure of $\im(f)$, and define
$$
f\simeq^{\hat\Gamma_{\mathfrak A}} f'\iff   C^*(f)\text{ is isomorphic to } C^*(f').
$$
Clearly the map $\hat\Gamma\to\hat\Gamma_{\mathfrak A}:\gamma\mapsto f_\gamma$ defined by $f_\gamma(\p_j)=\gamma(j)$ provides a Borel bijection witnessing that $\hat\Gamma$ and $\hat\Gamma_{\mathfrak A}$ are equivalent (and therefore they are also equivalent to $\Gamma$.)

We note for future reference that if we instead consider the free countable \emph{unital} unnormed $\Q(i)$-$*$-algebra $\Au$ and let
$$
\hat \Gamma_{\Au}(H)=\{f:\Au\to\mathcal B(H):f\text{ is a \emph{unital} $*$-homomorphism}\},
$$
then this gives a parameterization of all unital C$^*$-subalgebras of $\mathcal B(H)$. Note that $\Au$ may be identified with the set of all formal $*$-polynomials in the variables $X_k$ with coefficients in $\Q(i)$ (allowing a constant term.)


\subsection{The space $\Xi$.}\label{ss.xi}

Consider the Polish space $\R^\N$. We let $\Xi$ be the space of all $\delta\in \R^\N$ such that for some separable C$^*$-algebra $A$ and a sequence $y=(y_n)$ in $A$ generating it we have that
$$
\delta(j)=\|\p_j(y)\|_A.
$$
Each $\delta\in\Xi$ defines a seminorm $\|\p_j\|_\delta=\delta(j)$ on $\mathfrak A$ which satisfies the C$^*$-axiom. Letting $I=\{\p_j:\delta(j)=0\}$ we obtain a norm on $\mathfrak A/I$. The completion of this algebra is then a C$^*$-algebra, which we denote by $B(\delta)$. It is clearly isomorphic to any C$^*$-algebra $A$ with $y=(y_n)$ as above satisfying $\|\p_j(y)\|=\delta(j)$.


\begin{lemma} 
The set $\Xi$ is closed in $\bbR^{\bbN}$. 
\end{lemma}

\begin{proof} Assume $\delta^n\in \Xi$ converges to $\delta\in \R^\N$ pointwise. Fix C$^*$-algebras $A_n$ and sequences $y^n=(y^n_i\in A_n: i\in \bbN)$ such that $\delta^n(j)=\|\p_j(y^n)\|_{A_n}$ for all $n$ and $j$. For a nonprincipal ultrafilter $\cU$ on $\bbN$, let $A_\infty$ be the subalgebra of the ultraproduct $\prod_{\cU} A_n$ generated by the elements $\mathbf y_i=(y^0_i,y^1_i,\dots)$, for $i\in \bbN$. Then clearly $\delta(j)=\lim_{n\to \cU} \|\p_j(y^n)\|_{A_n}=\|\p_j(\mathbf y_i)\|_{A_\infty}$, hence $A$ witnesses $\delta\in \Xi$. 
\end{proof}

Thus $\Xi$ provides yet another parameterization of the category of separable C$^*$-algebras, and we define in $\Xi$ the equivalence relation
$$
\delta\simeq^{\Xi}\delta'\iff B(\delta)\text{ is isomorphic to } B(\delta').
$$
Below we will prove that this parameterization is equivalent to $\Gamma$ and $\hat\Gamma$. Note that an alternative description of $\Xi$ is obtained by considering the set of $f\in\R^{\mathfrak A}$ which define a C$^*$-seminorm on $\mathfrak A$; this set is easily seen to be Borel since the requirements of being C$^*$-seminorm are Borel conditions.


\subsection{The space $\hat\Xi$.}\label{ss.hatxi}

Our last parameterization is obtained by considering the set
$$
\hat\Xi \subseteq \N^{\N\times\N}\times\N^{\Q(i)\times\N}\times\N^{\N\times\N}\times\N^\N\times\R^\N
$$
of all tuples $(f,g,h,k,r)$ such that the operations (with $m,n$ in $\N$ and $q\in \bbQ(i)$) defined by
\begin{align*}
& m+_f n=f(m,n) \\
& q\cdot_g n = g(q,n)\\
& m\cdot_h n=h(m,n)\\
& m^{*_k}=k(m)\\
& \|n\|_r=r(n)
\end{align*}
give $\N$ the structure of a normed $*$-algebra over $\Q(i)$ which further satisfies the ``C$^*$-axiom'',
$$
\|n\cdot_{h} n^{*_k}\|_r=\|n\|_r^2
$$
for all $n\in\N$. The set $\hat\Xi$ is Borel since the axioms of being a normed $*$-algebra over $\Q(i)$ are Borel conditions. For $A\in\hat\Xi$, let $\hat B(A)$ denote the completion of $A$ with respect to the norm and equipped with the extension of the operations on $A$ to $\hat B(A)$. Note in particular that the operation of scalar multiplication may be uniquely extended from $\Q(i)$ to $\C$. We define for $A_0,A_1\in\hat\Xi$ the equivalence relation
$$
A_0\simeq^{\hat\Xi} A_1\iff \hat B(A_0)\text{ is isomorphic to } \hat B(A_1).
$$
For future reference, we note that the infinite symmetric group $\Sym(\N)$ acts naturally on $\hat\Xi$: If $\sigma\in\Sym(\N)$ and $(f,g,h,k,r)\in\hat\Xi$, we let $\sigma\cdot f\in\N^{\N\times\N}$ be defined by
$$
(\sigma\cdot f)(m,n)=k\iff f(\sigma^{-1}(m),\sigma^{-1}(n))=\sigma^{-1}(k),
$$
and defined $\sigma\cdot g$, $\sigma\cdot h$, $\sigma\cdot k$ and $\sigma\cdot r$ similarly. Then we let $\sigma\cdot (f,g,h,k,r)=(\sigma\cdot f, \sigma\cdot g,\sigma\cdot h, \sigma\cdot k, \sigma\cdot r)$. It is clear that $\sigma$ induces an isomorphism of the structures $(f,g,h,k,r)$ and $\sigma\cdot (f,g,h,k,r)$. However, it clearly \emph{does not} induce the equivalence relation $\simeq^{\hat\Xi}$, which is strictly coarser.

\begin{remark} 
(1) It is useful to think of $\Gamma$ and $\hat\Gamma$ as parameterizations of \emph{concrete} C$^*$-algebras, while $\Xi$ and $\hat\Xi$ can be thought of as parameterizing \emph{abstract} C$^*$ algebras.

(2) The parameterizations $\Gamma$, $\hat\Gamma$ and $\Xi$ all contain a unique element corresponding to the trivial C$^*$-algebra, which we denote by $0$ in all cases. Note that $\hat\Xi$ does not parameterize the trivial C$^*$-algebra.
\end{remark}


\subsection{Equivalence of $\Gamma$, $\hat\Gamma$, $\Xi$ and $\hat\Xi$.}
\label{S.Equivalence}

We now establish that the four parameterizations described above give us equivalent parameterizations of the non-trivial separable C$^*$-algebras. First we need the following lemma.


\begin{lemma}\label{l.injection}
Let $X$ be a Polish space and let $Y$ be any of the spaces $\Gamma,\hat\Gamma,\Xi$ or $\hat\Xi$. Let $f:X\to Y$ be a Borel function such that $f(x)\neq 0$ for all $x\in X$. Then there is a Borel injection $\tilde f:X\to Y$ such that for all $x\in X$, $f(x)\simeq^Y \tilde f(x)$.
\end{lemma}

\begin{proof}
$Y=\Gamma$: We may assume that $X=[\N]^\infty$, the space of infinite subsets of $\N$. (Under the natural identification this is a $G_\delta$ subset of $2^\N$ and therefore Polish. It is then homeomorphic to the set of irrationals.) Given $\gamma\in\Gamma\setminus \{0\}$ and $x\in X$, let $n_0(\gamma)\in\N$ be the least such that $\gamma(n_0)\neq 0$ and define
$$
\gamma_x(k)=\left\{\begin{array}{ll}
i\gamma(n_0(\gamma)) & \text{ if } k=2^i \text{ for some } i\in x;\\
\gamma(j) & \text{ if } k=3^j \text{ for some } j\in\N;\\
0         & \text{otherwise}.
\end{array}
\right.
$$
Clearly $C^*(\gamma)=C^*(\gamma_x)$, and $\tilde f:X\to \Gamma\setminus\{0\}$ defined by $\tilde f(x)= (f(x))_x$ is a Borel injection.

$Y=\hat\Gamma$: Clear, since $\Gamma$ and $\hat\Gamma$ are equivalent.

$Y=\Xi$. We may assume that $X=\R_+$. Fix $x\in X$, let $\delta=f(x)$, and let $n_0(\delta)\in\N$ be least such that $\p_{n_0}=X_{i_0}$ for some $i_0\in\N$ and $\delta(n_0)\neq 0$. Let $A$ be a C$^*$-algebra and $y=(y_n)$ be a dense sequence in $A$ such that $\delta(n)=\|\p_n(y)\|_A$, and let $\tilde y=(\tilde y_n)$ be
$$
\tilde y_n=\left\{\begin{array}{ll}
\frac x {\|y_n\|_A}y_n & \text{if } n=i_0;\\
y_n & \text{otherwise.}
\end{array}\right.
$$
Then define $\tilde f(x)(n)=\|\p_n(\tilde y)\|_A$. Clearly $\tilde f(x)\simeq^{\Xi} f(x)$ for all $x\in X$, and since $\|\p_{n_0}(\tilde y)\|_A=x$, the function $\tilde f$ is injective. (Note that $\tilde f(x)$ does not depend on the choice of $A$ and $y$, so it is in fact a function.) Finally, $\tilde f$ is Borel by \cite[14.12]{Ke:Classical}, since
\begin{align*}
\tilde f(x)=\delta'\iff &(\exists\gamma,\gamma'\in\Gamma)(\exists\delta\in\Xi\setminus\{0\}) f(x)=\delta\wedge (\forall n) \delta(n)=\|\p_n(\gamma)\|\wedge\\
&(\forall i) ((i\neq n_0(\delta)\wedge \gamma'(i)=\gamma(i))\vee (i=n_0(\delta)\wedge\gamma'(i)=\frac x {\|\gamma(i)\|}\gamma(i))),
\end{align*}
gives an analytic definition of the graph of $\tilde f$ (with $n_0(\delta)$ defined as above.)

$Y=\hat\Xi$: Assume $X=[2\N]^\infty$, the infinite subsets of the \emph{even} natural numbers. Given $(f,g,h,k,r)\in\hat\Xi$ and $x\in X$, we can find, in a Borel way, a permutation $\sigma=\sigma_{(f,g,h,j,r),x}$ of $\N$ so that 
$$
x=\{2^n\cdot_{\sigma\cdot g} 1: n\in\N\}.
$$ 
Then $\tilde f(x)=\sigma_{(f,g,h,j,r),x}\cdot f(x)$ works.
\end{proof}

\begin{remark} The classical principle \cite[14.12]{Ke:Classical} that a function whose graph is analytic is Borel will be used frequently in what follows, usually without comment.

\end{remark}


\begin{prop}\label{p.xihatxiequiv}
$\Xi\setminus\{0\}$ and $\hat\Xi$ are equivalent.
\end{prop}

\begin{proof}
By the previous Lemma, it suffices to show that $\Xi\setminus\{0\}$ and $\hat\Xi$ are weakly equivalent.

Identify $\Xi\setminus\{0\}$ with a subset of $\R^{\mathfrak A}$ in the natural way, and define 
$$
E=\{(\delta,m,n)\in\Xi\setminus\{0\}\times\N\times\N: \delta(\p_m-\p_n)=0\}.
$$
Then the section $E_\delta=\{(m,n)\in\N^2: (\delta,m,n)\in E\}$ defines an equivalence relation on $\N$. Let $f_n:\Xi\to\N$ be Borel functions such that each $f_n(\delta)$ is the least element in $\N$ not $E_\delta$-equivalent to $f_m(\delta)$ for $m<n$. If we let $I_\delta=\{\p_n: \delta(\p_n)=0\}$,  then $n\mapsto f(n) I_\delta$ provides a bijection between $\mathfrak A/I_\delta$ and $\N$, and from this we can define (in a Borel way) algebra operations and the norm on $\N$ corresponding to $A\in\hat\Xi$ such that $A\simeq \mathfrak A/I_\delta$.

Conversely, given a normed $\Q(i)$-$*$-algebra $A\in\hat\Xi$ (with underlying set $\N$), an element $\delta_A\in\Xi$ is defined by letting $\delta_A(n)=\|\p_n(X_i=i:i\in\N)\|_A$, where $\p_n(X_i=i:i\in\N)$ denotes the evaluation of $\p_n$ in $A$ when letting $X_i=i$.
\end{proof}


\begin{prop}\label{p.gammaxiequiv}
$\Gamma$ and $\Xi$ are equivalent. Thus $\Gamma$, $\hat\Gamma$ and $\Xi$ are equivalent parameterizations of the separable C$^*$-algebras, and $\Gamma\setminus\{0\}$, $\hat\Gamma\setminus\{0\}$, $\Xi\setminus\{0\}$ and $\hat\Xi$ are equivalent parameterizations of the non-trivial separable C$^*$-algebras.
\end{prop}

For the proof of this we need the following easy (but useful) Lemma:


\begin{lemma}\label{l.gammamaps}

Let $H$ be a separable infinite dimensional Hilbert space. Then:

(1) A function $f:X\to\Gamma(H)$ on a Polish space $X$ is Borel if and only if for some (any) sequence $(e_i)$ with dense span in $H$ we have that the functions
$$
x\mapsto (f(x)(n)e_i|e_j)
$$
are Borel, for all $n,i,j\in\N$.

(2) Suppose $g:X\to\bigcup_{x\in X}\Gamma(H_x)$ is a function such that for each $x\in X$ we have $g(x)\in\Gamma(H_x)$, where $H_x$ is a separable infinite dimensional Hilbert space, and there is a system $(e^x_i)_{j\in\N}$ with $\Span\{e^x_i:i\in\N\}$ dense in $H_x$. If for all $n,i,j\in\N$ we have that the functions
$$
X\to\C: x\mapsto (e_i^x|e_j^x)
$$ 
and
$$
X\to\C: x\mapsto (g(n)(x)e_i^x|e_j^x)
$$
are Borel, then there is a Borel $\hat g: X\to\cB(H)$ and a family $T_x:H\to H_x$ of linear isometries such that for all $n\in\N$,
$$
g(x)(n)=T_x\hat g(x)(n) T_x^{-1}.
$$
\end{lemma}

We postpone the proof of Lemma \ref{l.gammamaps} until after the proof of Proposition \ref{p.gammaxiequiv}.


\begin{proof}[Proof of Proposition \ref{p.gammaxiequiv}]
By Lemma \ref{l.injection}, it is again enough to show that $\Gamma$ and $\Xi$ are weakly equivalent. For the first direction, the map $\psi:\Gamma\to\Xi$ given by
$$
\psi(\gamma)(n)=\|\p_n(\gamma)\|
$$
clearly works.

For the other direction we rely on the GNS construction (e.g. \cite[II.6.4]{blackadar}). For each $\delta\in \Xi$ let $S(\delta)$ be the space of all $\phi\in \bbC^{\bbN}$ such that
\begin{enumerate}
\item $|\phi(k)|\leq \delta(k)$ for all $k$,
\item $\phi(k)=\phi(m)+\phi(n)$, whenever $\p_k=\p_m+\p_n$,
\item  $\phi(k)\geq 0$ whenever $\p_k=\p_m^*\p_m$ for some $m$.
\end{enumerate}
Then $S(\delta)$ is a compact subset of $\bbC^{\bbN}$ for each $\delta\in\Xi$, and so since the relation
$$
\{(\delta,\phi)\in\Xi\times\C^\N:\phi\in S(\delta)\}
$$
is Borel, it follows by \cite[28.8]{Ke:Classical} that $\Xi\to K(\bbC^{\bbN}): \delta\mapsto S(\delta) $ is a Borel function into the Polish space $K(\bbC^{\bbN})$ of compact subsets of $\bbC^{\bbN}$. Consider the set
$$
N=\{(\delta,\phi,n,m)\in\Xi\times\C^\N\times\N\times\N: \phi\in S(\delta)\wedge (\exists k) \p_k=(\p_n-\p_m)^*(\p_n-\p_m)\wedge \phi(k)=0\}.
$$
Then for each $\delta$ and $\phi$ the relation $N_{\delta,\phi}=\{(n,m)\in\N^2:(\delta,\phi,n,m)\in N\}$ is an equivalence relation on $\N$. Without any real loss of generality we can assume that $N_{\delta,\phi}$ always has infinitely many classes. Let $\sigma_n:\Xi\times\C^\N\to\N$ be a sequence of Borel maps such that for all $\delta$ and $\phi$ fixed the set
$$
\{\sigma_n(\delta,\phi):n\in\N\}
$$
meets every $N_{\delta,\phi}$ class once. For $\delta$ and $\phi$ fixed we can then define an inner product on $\N$ by
$$
(n|m)_{\delta,\phi}=\phi(k)\iff \p_k=\p_{\sigma_n(\delta,\phi)}^* \p_{\sigma_m(\delta,\phi)}.
$$
Let $H(\delta,\phi)$ denote the completion of this pre-Hilbert space. Then there is a unique operator $\gamma(n)\in \cB(H(\delta,\phi))$ extending the operator acting on $(\N,(\cdot|\cdot)_{\delta,\phi})$ defined by letting $\gamma_{\delta,\phi}(n)(m)=k$ iff there is some $k'\in\N$ such that 
$$
\p_{\sigma_n(\delta,\phi)}\p_{\sigma_m(\delta,\phi)}=\p_{k'}
$$ 
and $(\delta,\phi,k',\sigma_k(\delta,\phi))\in N$. Note that $n\mapsto \gamma_{\delta,\psi}$ corresponds to the GNS representation of the normed $*$-algebra over $\Q(i)$ that corresponds to $\delta$. Since the elements of $\N$ generate $H(\delta,\psi)$ and the map
$$
(\delta,\psi)\mapsto (\gamma_{\delta,\psi}(n)(i)|j)_{\delta,\psi}
$$
is Borel, it follows from Lemma \ref{l.gammamaps} that there is a Borel function 
$$
\Xi\times\C^\N\to\Gamma(H): (\delta,\psi)\mapsto \tilde\gamma(\delta,\psi)
$$
such that $\tilde\gamma(\delta,\psi)\in\Gamma(H)$ is conjugate to $(\gamma_{\delta,\phi}(n))_{n\in\N}\in\Gamma(H(\delta,\psi))$ for all $\delta,\psi$.

Since the map $\Xi\to K(\bbC^{\bbN}): \delta\mapsto S(\delta)$ is Borel,
by  the Kuratowski--Ryll-Nardzewski theorem (\cite[Theorem~12.13]{Ke:Classical})
there are Borel maps $\phi_n\colon \Xi\to \bbC^{\bbN}$ such that for every $\delta\in \Xi$
the set $\phi_n(\delta)$, for $n\in \bbN$, is dense in $S(\delta)$. Writing $H=\bigoplus_{n=1}^\infty H_n$ where $H_n$ are infinite dimensional Hilbert spaces, we may then by the above find a Borel map $\Xi\to\Gamma(H):\delta\to\gamma(\delta)$ such that the restriction $\gamma(\delta)\restrict H_n$ is conjugate to $\tilde\gamma(\delta,\phi_n)$. Since the sequence $(\phi_n(\delta))$ is dense in $S(\delta)$, it follows that $\gamma(\delta)$ is a faithful representation of the algebra corresponding to $\delta$.
\end{proof}


\begin{proof}[Proof of Lemma \ref{l.gammamaps}]
(1) is clear from the definition of $\Gamma(H)$. To see (2), first note that the Gram-Schmidt process provides orthonormal bases $(f_i^x)_{i\in\N}$ for the $H_x$ such that
$$
f^x_i=\sum_{j=1}^i r_{i,j}^xe_j^x
$$
and the coefficient maps $x\mapsto r_{i,j}^x$ are Borel. Therefore the maps
$$
X\mapsto\C: (g(n)(x)f_i^x|f_j^x)
$$
are Borel for all $n,i,j\in\N$, and so we may in fact assume that $(e_i^x)_{i\in\N}$ forms an orthonormal basis to begin with. But then if $(e_i)_{i\in\N}$ is an orthonormal basis a function $\hat g:X\to\Gamma(H)$ is defined by
$$
\hat g(x)(n)=S\iff (\forall i,j) (Se_i, e_j)=(g(n)(x)e_i^x, e_j^x),
$$
and since this also provides a Borel description of the graph of $\hat g$, $\hat g$ is Borel by \cite[14.12]{Ke:Classical}. Finally, defining $T_x:H\to H_x$ to be the isometry mapping $e_i$ to $e_i^x$ for each $x$ provides the desired conjugating map.
\end{proof}


\subsection{Parameterizing unital C$^*$-algebras}\label{ss.paramunital}
We briefly discuss the parameterization of unital C$^*$-algebras. Define
$$
\Gammau=\{\gamma\in\Gamma: C^*(\gamma)\text{ is unital}\}.
$$
We will see (Lemma \ref{l.unital}) that this set is Borel. We can similarly define $\hatGammau\subseteq\hat\Gamma$, $\Xiu\subseteq\Xi$ and $\hatXiu\subseteq\hat\Xi$. However, as noted in \ref{ss.hatgamma}, the set
$$
\hat \Gamma_{\Au}(H)=\{f:\Au\to\mathcal B(H):f\text{ is a \emph{unital} $*$-homomorphism}\}
$$
is Borel and naturally parameterizes the unital C$^*$-subalgebras of $\cB(H)$. In analogy, we define
$$
\Xi_{\Au}=\{f\in\R^{\Au}:f\text{ defines a C$^*$-seminorm on } \Au\text{ with } f(1)=1\},
$$
which is also Borel. Then a similar proof to that of Proposition \ref{p.gammaxiequiv} shows:

\begin{prop}\label{p.gammaxiunitalequiv}
The Borel sets $\hat\Gamma_{\Au}$ and $\Xi_{\Au}$ provide equivalent parameterizations of the unital $C^*$-algebras.
\end{prop}
\noindent In \S 3 we will see that $\hat\Gamma_{\Au}$ and $\Xi_{\Au}$ are also equivalent to $\Gammau$ (and therefore also $\hatGammau$, $\Xiu$ and $\hatXiu$.)  For future use, we fix once and for all an enumeration $(\q_n)_{n\in\N}$ of all the formal $\Q(i)$-$*$-polynomials (allowing constant terms), so that naturally $\Au=\{\q_n:n\in\N\}$. Also for future reference, we note that Lemma \ref{l.injection} holds for $Y=\hat\Gamma_{\Au}$ and $Y=\Xi_{\Au}$ (the easy proof is left to the reader.)


\subsection{Basic maps and relations} We close this section by making two simple, but useful, observations pertaining to the parameterization $\Gamma$. While Borel structures of the weak operator topology, strong operator topology, $\sigma$-weak operator topology and $\sigma$-strong operator topology all coincide, the Borel structure of the norm topology is strictly finer. However, we have:

\begin{lemma} \label{L.B.1} Every norm open ball in $\cB(H)$ is a Borel subset of $\Gamma$, and for every $\e>0$ the set $\{(a,b): \|a-b\|<\e\}$ is Borel.
It follows that the maps 
$$
\cB(H)\to\bbR: a\mapsto \|a\|, \ \ \cB(H)^2\to\R: (a,b)\mapsto \|a-b\| 
$$
are Borel.
\end{lemma}

\begin{proof} Clearly $\{a: \|a\|>\e\}$ is weakly open for all $\e\geq 0$.
Hence norm open balls are $F_\sigma$.
\end{proof}

\begin{lemma} \label{L.equality.borel}
The relations 
$$
\{(\gamma,\gamma')\in \Gamma\times\Gamma: C^*(\gamma)\subseteq C^*(\gamma')\}
$$
and
$$
\{(\gamma,\gamma')\in\Gamma\times\Gamma: C^*(\gamma)= C^*(\gamma')\}
$$ 
are Borel. 
\end{lemma} 

\begin{proof} 
We have
$$ 
C^*(\gamma)\subseteq C^*(\gamma')\iff (\forall n)(\forall \e>0)(\exists m)\|\gamma_n'-\p_m(\gamma)\|<\e,
$$ 
which is Borel by Lemma~\ref{L.B.1}. 
\end{proof}


\section{Basic definability results}\label{S.Basicdef}

In this section we will show that a wide variety of standard C$^*$-algebra constructions correspond to Borel relations and functions in the spaces $\Gamma$ and $\Xi$.


\begin{prop} \label{P.isomorphism} 
\ 

\begin{enumerate}[{\rm \indent (1)}]

\item The relation $\precsim$ on $\Gamma$, defined by $\gamma\precsim\delta$ if
and only if $C^*(\gamma)$ is isomorphic to a subalgebra of
$C^*(\delta)$, is analytic.

\item The relation $\simeq^{\Gamma}$ is analytic. In particular, $\Gamma$, $\hat\Gamma$, $\Xi$ and $\hat\Xi$ are good standard Borel parameterizations of the class of separable $C^*$-algebras. 
\end{enumerate}
\end{prop}

Before the proof of Proposition~\ref{P.isomorphism} we introduce some terminology
and prove a lemma. 
The following terminology will be useful both here and later: We
call $\Phi:\N\to\N^\N$ a \emph{code} for a $*$-homomorphism
$C^*(\gamma)\to C^*(\gamma')$ if for all $m,n,k$ we have:
\begin{enumerate}
\item For each fixed $m$ the sequence  $a_{m,k}=\p_{\Phi(m)(k)}(\gamma')$, $k\in \bbN$, is Cauchy. Write $a_m=\lim_k a_{m,k}$.
\item If $\p_m(\gamma)+\p_n(\gamma)=\p_k(\gamma)$ then $a_m+a_n=a_k$.
\item If $\p_m(\gamma)\p_n(\gamma)=\p_k(\gamma)$ then $a_m a_n=a_k$.
\item If $\p_m(\gamma)^*=\p_k(\gamma)$ then $a_m^*=a_k$.
\item \label{L.subalgebra.4} $\|\p_m(\gamma)\|\leq\|a_m\|$.
\end{enumerate}

We call $\Phi$ a code for a {\it monomorphism} if equality holds in
(\ref{L.subalgebra.4}).

\subsection{Definitions of $\Rhom$, $\Rmono$ and $\Riso$}\label{S.Rhom}
Let $H_0$ and $H_1$ be separable complex Hilbert spaces. Then it is easy to see that the relations $\Rhomx{H_0,H_1}, \Rmonox{H_0,H_1}\subseteq
\Gamma(H_0)\times\Gamma(H_1)\times(\N^\N)^\N$ defined by
\begin{align*}
& \Rhomx{H_0,H_1}(\gamma,\gamma',\Phi)\iff \ \Phi \text{ is a code for a *-homomorphism } C^*(\gamma)\to C^*(\gamma')\\
& \Rmonox{H_0,H_1}(\gamma,\gamma',\Phi)\iff \ \Phi \text{ is a code for a
*-monomorphism } C^*(\gamma)\to C^*(\gamma')
\end{align*}
are Borel. We let $\Rhomx{H}=\Rhomx{H,H}$ and $\Rmonox{H}=\Rmonox{H,H}$ for any Hilbert space $H$. If $H_0, H_1$ or $H$ are clear from the context or can be taken to be any (separable) Hilbert spaces then we will suppress the superscript and write $\Rhom$ and $\Rmono$. 
The following is immediate from the definitions:


\begin{lemma}
If $(\gamma,\gamma',\Phi)\in \Rhom$ then there is a unique
homomorphism $\hat\Phi: C^*(\gamma)\to C^*(\gamma')$ which satisfies
$$
\hat\Phi(\gamma(j))=a_j
$$
for all $j\in\N$. If $(\gamma,\gamma',\Phi)\in \Rmono{}$ then
$\hat\Phi$ is a monomorphism.
\end{lemma}


\begin{proof}
If $\Rhom{}(\gamma,\gamma',\Phi)$ then
$$
\p_m(\gamma)\mapsto a_m
$$
is a *-homomorphism from a dense subalgebra of $C^*(\gamma)$ into a
subalgebra of $C^*(\delta)$. Since it is a contraction it extends to
a *-homomorphism of $\hat\Phi:C^*(\gamma)\to C^*(\gamma')$ onto a
subalgebra of $C^*(\gamma)$. If $\Rmono{}(\gamma,\gamma',\Phi)$
holds then $\hat\Phi$ is clearly a monomorphism.
\end{proof}

We also define a relation $\Riso$ (we are suppressing $H_0$ and $H_1$) by 
\begin{align*}
\Riso(\gamma,\gamma',\Phi) \iff & \Rmono{}(\gamma,\gamma',\Phi)\wedge\\
&(\forall m)(\forall\e>0)(\exists k\in\N)(\forall n>k)
\|\p_{\Phi(m)(n)}(\gamma)-\p_m(\gamma')\|<\e.
\end{align*}

This relation states that $\Phi$ is a monomorphism and an epimorphism, and therefore an isomorphism. 
It is Borel because $\Rmono$ is Borel.


\begin{proof}[Proof of Proposition \ref{P.isomorphism}] (1) Clear, since
$$
\gamma\precsim \gamma'\iff (\exists\Phi:\N\to\N^\N)
\Rmono{}(\gamma,\gamma',\Phi).
$$
(2) We have
$$
C^*(\gamma)\simeq C^*(\gamma')\iff (\exists \Phi:\N\to\N^{\N})\Riso(\gamma,\gamma',\Phi),
$$  
giving an analytic definition of $\simeq^{\Gamma}$, and so $\Gamma$ is a good parameterization. The last assertion follows from the equivalence of the four parameterizations.
\end{proof}


\begin{remark} 
Note that the equivalence relation $\rE$ on $\Gamma$ defined by $\gamma\rE \delta$ if and only if there is a unitary $u\in\cB(H)$ such that $u C^*(\gamma)u^*=C^*(\delta)$ is a proper subset of $\simeq^{\Gamma}$ and that $\rE$ is induced by a continuous action of the unitary group. We don't know whether the relation $\simeq^\Gamma$ is an orbit equivalence relation induced be the action of a Polish group action on $\Gamma$, see discussion at the end of \S9.
\end{remark}

For future use, let us note the following.

\begin{lemma} \label{L.B.2} The set $Y$ of all $\gamma\in \Gamma$ such that $\gamma_n$, $n\in \bbN$,  is a
Cauchy sequence (in norm) is Borel.
The function $\Psi\colon Y\to \cB(H)$ that assigns the limit to a Cauchy
sequence is Borel.
\end{lemma}

\begin{proof} We have $\gamma\in Y$ if and only if $(\forall \e>0)(\exists m)
(\forall n\geq m) \|\gamma_m-\gamma_n\|<\e$. By Lemma~\ref{L.B.1}, the
conclusion follows.

It suffices to show that the graph $G$ of  $\Psi$ is a Borel subset of
$\cB(H)^{\bbN}\times \cB(H)$. But $(\gamma,a)\in G$ if and only if for all
$\e>0$ there is $m$ such that for all $n\geq m$ we have $\|\gamma_m-a\|\leq
\e$, which is by Lemma~\ref{L.B.1} a Borel set.
\end{proof}


\subsection{Directed systems, inductive limits, and $\Rdir$.}

A directed system of C*-algebras can be coded by a sequence $(\gamma_i)_{i\in \bbN}$ in $\Gamma$
and a sequence $\Phi_i: \bbN\to \bbN^{\bbN}$, for $i\in \bbN$,  such that
$$
(\forall i\in\N) \Rhom{}(\gamma_i,\gamma_{i+1},\Phi_i).
$$
The set $\Rdir{}\subseteq \Gamma^\N\times((\N^\N)^\N)^\N$ of codes for inductive systems is defined by
$$
((\gamma_i)_{i\in \bbN},(\Phi_i)_{i\in \bbN})\in \Rdir{}
\iff (\forall i\in\N) \Rhom{}(\gamma_i,\gamma_{i+1},\Phi_i)
$$
and is clearly Borel.


\begin{prop} \label{P.Directed}
There are Borel maps $\LIM: \Rdir{}\to \Gamma$ and $\Psi_i:\Rdir{}\to (\N^\N)^\N$ such that
$$
C^*(\LIM((\gamma_i)_{i\in \bbN},(\Phi_i)_{i\in \bbN}))\simeq\lim_{i\to\infty} (C^*(\gamma_i),\hat\Phi_i)
$$
and it holds that
$$
(\forall n\in\N) \Rhom{}(\gamma_n,\LIM((\gamma_i)_{i\in \bbN},(\Phi_i)_{i\in \bbN}),\Psi_n((\gamma_i)_{i\in \bbN},(\Phi_i)_{i\in \bbN}))
$$
and $\hat\Psi_n((\gamma_i)_{i\in \bbN},(\Phi_i)_{i\in \bbN}):
C^*(\gamma)\to C^*(\LIM((\gamma_i)_{i\in \bbN},(\Phi_i)_{i\in \bbN}))$ satisfies
$$
\hat\Psi_{n+1}\circ\hat\Phi_{n}=\hat\Psi_{n},
$$
i.e. the diagram
\[
\diagram
C^*(\gamma_{n+1}) \rrto^{\hat\Psi_{n+1}}
&& \LIM((\gamma_i)_{i\in \bbN},(\Phi_i)_{i\in \bbN})\\
C^*(\gamma_n)
\uto^{\hat\Phi_n}
\urrto^{\hat\Psi_n}
\enddiagram
\]
commutes.
\end{prop}

We start by noting the simpler Lemma~\ref{L.1step} below.
The constant $i$ sequence is denoted $\overline i$.
For $(\gamma,\gamma',\Phi)\in \Rhom$ define
the function $f:\Rhom{}\to\Gamma$  by
$$
f(\gamma,\gamma',\Phi)(m)=\left\{\begin{array}{ll}
\gamma'_k & \text{ if } m=3^k \text{ for } k\geq 1\\
a & \text{ if } m=2^k\text{ and } \lim_{i\to\infty}\gamma'_{\Phi(k)(i)}=a\\
0 & \text{ otherwise.}
\end{array}\right.
$$
Then the following is obvious:


\begin{lemma}\label{L.1step}
The function $f$
introduced above is Borel and
 for all $(\gamma,\gamma',\Phi)\in \Rhom{}$ we have
$$
C^*(\gamma')\simeq C^*(f(\gamma,\gamma',\Phi)).
$$
Moreover, for $\Psi,\Phi':\N\to\N^\N$ defined by $\Phi'(m)=\overline{2^m}$ and $\Psi(m)=\overline{3^m}$ for $m\geq 1$, we have that
$\Riso(\gamma',f(\gamma,\gamma',\Phi),\Psi)$, $\Rhom{}(\gamma,f(\gamma,\gamma',\Phi),\Phi')$ and
$$
\hat\Psi\circ\hat\Phi=\hat\Phi'.
$$
\end{lemma}


\begin{proof}[Proof of Proposition~\ref{P.Directed}] By Proposition~\ref{p.gammaxiequiv}
it will suffice to define $\LIM$ with the range in $\Xi$. 
Fix $((\gamma_i)_{i\in \bbN},(\Phi_i)_{i\in \bbN})\in \Rdir{}$, let
$$
A=\lim_{i\to\infty} (C^*(\gamma_i),\hat\Phi_i),
$$
and let $f_i:C^*(\gamma_i)\to A$ be the connecting maps satisfying
$f_{i+1}\circ\hat\Phi_i=f_i$. By  Lemma~\ref {L.1step} we may assume
that that for all $m\in\N$ the sequence $\Phi_i(m)$, for $i\in\bbN$,  is constant.
Let $\varphi_i(m)=\Phi_i(m)(1)$, define
$\varphi_{i,j}=\varphi_i\circ\cdots\circ\varphi_j$ for $j<i$, and
let $\beta:\N\to\N\times\N$ be a fixed bijection. Let
$\tilde\gamma\in\Gamma$ be defined by
$$
\tilde\gamma(i)=f_{\beta(i)_0}(\gamma_{\beta(i)_0}(\beta(i)_1)).
$$
Then a code $\delta\in\Xi$ for $\tilde\gamma$ is given by
$$
\delta(i)=\lim_{k\to\infty}\|\p_{\varphi_{k,\beta(i)_0}(\beta(i)_1)}(\gamma_k)\|
$$
and if we define
$$
\Psi_j((\gamma_i)_{i\in \bbN},(\Phi_i)_{i\in \bbN})(m)(n)=k\iff \beta(k)_0=j\wedge
\beta(k)_1=m
$$
then $\Psi_j$ is a code for $f_j$.
\end{proof}

Next we prove that most standard constructions and relations that occur in C$^*$-algebra theory correspond to Borel maps and relations in the parameterizations we have introduced. The first lemma follows easily from the definitions, and we leave the proof to the reader.


\begin{lemma}\label{L.B.3.0} The following maps are Borel.
\begin{enumerate}
\item $\cB(H)\times \cB(H)\to\cB(H):(a,b)\mapsto ab$,
\item $\cB(H)\times \cB(H)\to\cB(H): (a,b)\mapsto a+b$,
\item $\cB(H)\times \bbC\to\cB(H): (a,\lambda)\mapsto \lambda a$,
\item  $\cB(H)\to\cB(H): a\mapsto a^*$,
\item $\cB(H)\times \cB(H)\to\cB(H)\otimes_{\min}\cB(H): (a,b)\mapsto a\otimes b
$ (where $\cB(H)\otimes_{\min} \cB(H)$ is identified
with $\cB(H)$ by fixing a $*$-isomorphism),
\item $\cB(H)\times \cB(H)\to M_2(\cB(H)): (a,b)\mapsto \begin{pmatrix} a & 0 \\ 0 & b
\end{pmatrix}$ (where $M_2(\cB(H))$ is identified with $\cB(H)$
by fixing a $*$-isomorphism).
\end{enumerate}
\end{lemma}


\begin{lemma}\label{L.B.3} The following subsets of $\cB(H),\cB(H)^2$, and $\Gamma$ are Borel.
\begin{enumerate}
\item $\{(a,b): ab=ba\}$.
\item \label{L.B.3.sa} $\cB(H)_{\sa}=\{a: a=a^*\}$.
\item $\cB(H)_+=\{a\in \cB(H)_{\sa}: a\geq 0\}$.
\item \label{L.B.3.projection} $\cP(\cB(H))=\{a\in \cB(H): a$ is a projection$\}$.
\item \label{L.B.3.isometry} $\{a: a$ is a partial isometry$\}$.
\item \label{L.B.3.invertible} $\{a: a$ is invertible$\}$.
\item \label{L.B.3.normal} $\{a: a$ is normal$\}$.

\end{enumerate}
\end{lemma}

\begin{proof}
\eqref{L.B.3.projection} Immediate since the maps $a\mapsto a-a^2$ and
$a\mapsto a-a^*$ are Borel measurable.

\eqref{L.B.3.isometry} Since $a$ is a partial isometry if and only if $a^*a$
and $aa^*$ are both projections, this follows from the Borel-measurability of
these maps and~\eqref{L.B.3.projection}.

\eqref{L.B.3.invertible} Let $\xi_n$ be a countable dense subset of the unit
ball of $\cB(H)$. Then $a$ is invertible if and only if there is $\e>0$ such
that $\|a\xi_n\|\geq \e$ and $\|a^*\xi_n\|\geq \e$ for all $n$ (\cite[3.2.6]{Pede:Analysis}).

\eqref{L.B.3.normal} Immediate since the map $a\mapsto [a,a^*]$ is Borel.
\end{proof}

Next we consider formation of the matrix algebra over a C$^*$-algebra. For this purpose, fix bijections $\beta_n:\N\to \N^{n\times n}$ for each $n$. While the next Lemma is in some sense a special case of the Lemma that follows it (which deals with tensor products) the formulation given below will be used later for the proof of Theorem \ref{mainintro}.


\begin{lemma}\label{L.matrix}
For each $n\in\N$ there are Borel functions $M_n:\Gamma(H)\to \Gamma(H^n)$ and $\theta_n:\Gamma(H)\times(\N^\N)^n\to\N^\N$ such that

(1)
$$
M_n(\gamma)=(\left(\begin{array}{ccc}
    \gamma_{\beta_n(l)(1,1)} &\cdots & \gamma_{\beta_n(l)(1,n)}\\
    \vdots & & \vdots\\
    \gamma_{\beta_n(l)(n,1)} &\cdots & \gamma_{\beta_n(l)(n,n)}\\
\end{array}\right):l\in\N)
$$

(2) If $(\gamma,\Psi_i)\in \Rhomx{H}$ for all $i=1,\ldots, n$ then
$$
(\gamma,M_n(\gamma),\theta_n(\gamma,\Psi_1,\ldots,\Psi_n))\in \Rhomx{H,H^n}
$$
and
$$
\theta_n(\gamma,\Psi_1,\ldots,\Psi_n)(k)(i)=m\implies \p_m(M_n(\gamma))=\diag(\gamma(\Psi_1(k)(i)),\ldots,\gamma(\Psi_n(k)(i))).
$$
That is, $\theta_n(\gamma,\Psi_1,\ldots,\Psi_n)$ codes the diagonal embedding twisted by the homomorphisms $\hat\Psi_i$.
\end{lemma}

\begin{proof}
(1) is clear. (2) follows by letting $\theta_n(\gamma,\psi_1,\ldots,\psi_n)(k)(i)=m$ if and only if $m$ is the least such that
$$
\p_m(M_n(\gamma))=\diag(\gamma(\Psi_1(k)(i)),\ldots,\gamma(\psi_n(k)(i))).
$$
\end{proof}


\begin{lemma}\label{L.tensor} There is a Borel-measurable map $\Tensor\colon
\Gamma\times \Gamma\to \Gamma$ such that
\[
C^*(\Tensor(\gamma,\delta))\cong C^*(\gamma)\otimes_{\min}
C^*(\delta)
\]
for all $\gamma$ and $\delta$ in $\Gamma$.

Moreover, there is a Borel-measurable map 
$\Tenx\colon \Gamma\times\Gamma\to \Gamma$ such that if $1\in C^*(\delta)$
then $C^*(\Tenx(\gamma,\delta))$ is the canonical copy of $C^*(\gamma)$ inside
$C^*(\gamma)\otimes_{\min} C^*(\delta)$. 
\end{lemma}


\begin{proof}
Fix a *-isomorphism $\Psi\colon \cB(H)\otimes_{\min} \cB(H)\to \cB(H) $. Define
$$
\Tensor(\gamma,\delta)_{2^m(2n+1)}=\Psi(\gamma_m \otimes \delta_n).
$$
Then $\Tensor$ is clearly Borel and the algebra generated by
$\Tensor(\gamma,\delta)$ is $C^*(\gamma)\otimes_{\min} C^*(\delta)$. For the moreover part,  $\Tenx(\gamma)_m=\Psi(\gamma_m\otimes 1)$ clearly works.
\end{proof}

It is not difficult to see that the set $\{\gamma\in \Gamma: 1\in C^*(\gamma)\}$ is Borel (cf. Lemma~\ref{l.unital}) but we shall not need this fact.


\begin{lemma}\label{L.C(X,A)} If $X$ is a locally compact Hausdorff
space then there is a Borel measurable map $\Phi\colon \Gamma\to
\Gamma$ such that
\[
\Phi(\gamma)\cong C_0(X,C^*(\gamma))
\]
for all $\gamma\in\Gamma$. In particular, letting $X=(0,1)$, we conclude that
there is a Borel map $\Phi$ such that $\Phi(\gamma)$ is isomorphic to
the suspension of $C^*(\gamma)$. 
\end{lemma}

\begin{proof}
This is immediate from Lemma~\ref{L.tensor} since $C(X,A)\cong C(X)\otimes_{\min} A$.
\end{proof}


\begin{lemma}\label{L.unitization}
There is a Borel function $\Unit\colon \Gamma\to \Gamma$ such that
$C^*(\Unit(\gamma))$ is isomorphic to the unitization of
$C^*(\gamma)$.
\end{lemma}

\begin{proof} Fix a partial isometry $v$ such that $vv^*=1$ and $v^*v$ is a projection onto a space of codimension 1. Let $\Unit(\gamma)_0=1$ and
$\Unit(\gamma)_{n+1}=v^* \gamma_n v$. Then $C^*(\gamma)$ is as
required.
 \end{proof}


\subsection{Effective enumerations}


\begin{lemma}\label{L.E.1}
\ 
\begin{enumerate}[\indent $(1)$]
\item \label{L.E.1.2} There is a Borel map $\Sa\colon \Gamma\to \Gamma$ such that for
every $\gamma\in \Gamma$ the set $\{\Sa(\gamma)(n): n\in \bbN\}$ is a
norm-dense subset of the set of self-adjoint elements
of~$C^*(\gamma)$.
\item \label{L.E.1.4} There is a Borel map $\Un:\Gamma\to\Gamma$ such that the set $\{\Un(\gamma):n\in\N\}$ is norm-dense in the set of unitaries in $C^*(\gamma)$ whenever $C^*(\gamma)$ is unital.
\item \label{L.E.1.1} There is a Borel map $\Pos\colon \Gamma\to \Gamma$
such that for every $\gamma\in \Gamma$ the set $\{\Pos(\gamma)(n):n\in \bbN\}$ is a norm-dense subset of the set of positive elements
 of~$C^*(\gamma)$.
\item \label{L.E.1.3} There is a Borel map $\Proj\colon \Gamma\to \Gamma$ such that for every
$\gamma\in \Gamma$ the set $\{\Proj(\gamma)(n): n\in \bbN\}$ is a
norm-dense subset of the set of projections of~$C^*(\gamma)$.
\end{enumerate}
\end{lemma}

\begin{proof}
\eqref{L.E.1.2} Let $\Sa(\gamma)(n)=\frac 12(\p_n(\gamma)+\p_n(\gamma)^*)$ for all $n$.
Clearly each $\Sa(\gamma)(n)$ is self-adjoint. If  $a\in C^*(\gamma)$ is self-adjoint
then $\|a-\Sa(\gamma)(n)\|\leq \|a-\p_n(\gamma)\|$. Therefore the range of $\Sa$ is norm-dense
subset of the set of self-adjoint elements of $C^*(\gamma)$.

\eqref{L.E.1.4} Let $\Un(\gamma)(n)=\exp(i\Sa(\gamma)(n))$.

 \eqref{L.E.1.1} Let
$\Pos(\gamma)(n)=\p_n(\gamma)^* \p_n(\gamma)$ for all $n$.
Pick a positive $a\in C^*(\gamma)$ and fix $\e>0$.
Pick $b\in C^*(\gamma)$ such that $a=b^*b$. Let $n$ be such that
$\|\p_n(\gamma)-b\|<\e/(2\|b\|)$ and $\|\p_n(\gamma)\|\leq \|b\|$.
Then
\[
\p_n(\gamma)^*\p_n(\gamma)-a=(\p_n(\gamma)^*-b^*)\p_n(\gamma)+
b^*(\p_n(\gamma)-b)
\]
and the right hand side clearly has norm $<\e$.

\eqref{L.E.1.3} Fix a function $f\colon \bbR\to [0,1]$ such that the
iterates $f^n$, $n\in \bbN$, of $f$ converge uniformly to the
function defined by $g(x)=0$, $x\leq 1/4$ and $g(x)=1$ for $x\geq
3/4$ on $(-\infty,1/4]\cup [3/4,\infty)$. For example, we can take
\[
f(x)=\begin{cases} 0, & x\leq 0\\
\frac x2, & 0<x\leq \frac 14\\
\frac 32 x -\frac 14,& \frac 14<x\leq \frac 34\\
  1-(1-x)/2, & \frac 34 <x\leq 1\\
 1, & x>1.
 \end{cases}
\]
The set $X=\cB(H)_{\sa}$  is a Borel subset of $\cB(H)$ by
Lemma~\ref{L.B.3}. Note that $X\cap \{\Pos(\gamma)(n): n\in \bbN\}$ is
dense in $X\cap C^*(\gamma)$. Let $\Psi\colon X\to \cB(H)^\bbN$ be
defined~by
\[
\Psi(a)(n)=f^n(a).
\]
By Lemma~\ref{L.B.2} the set $Y=\{b\in X: \Psi(b)$ is Cauchy$\}$ is Borel. For $n$ such that $\Pos(\gamma)(n)\in Y$ let $\Proj(\gamma)(n)$ be the limit of this
sequence, and let $\Proj(\gamma)(n)=0$ otherwise. By Lemma~\ref{L.B.2} again, $\Proj$ is Borel.

Fix  $\gamma$ and $n$. Clearly, the operator $\Proj(\gamma)(n)$ is
positive and its spectrum is a subset of $\{0,1\}$. Therefore it is a
projection in $C^*(\gamma)$. We need to check that for every
projection $p\in C^*(\gamma)$ and $\e>0$ there is $n$ such that
$\|\Proj(\gamma)(n)-p\|<\e$.

We may assume $\e<1/4$.  Pick $n$ so that $\|\Pos(\gamma)(n)-p\|<\e$.
Since $\e<1/4$, the spectrum of $\Pos(\gamma)(n)$
is included in $(-\e,\e)\cup (1-\e,1+\e)\subseteq (-1/4,1/4)\cup (3/4,5/4)$ and therefore the sequence
$f^j(\Pos(\gamma)(n))$, $j\in \bbN$, converges to a projection, $q$.
Clearly $\|p-q\|<2\e$.
\end{proof}


Recall from \ref{ss.paramunital} that $\Gammau$ denotes the set of $\gamma\in\Gamma$ parameterizing unital C$^*$-algebras. From the previous Lemma we now obtain:

\begin{lemma} \label{l.unital}
The set $\Gammau$ is Borel, and there is a Borel map $u:\Gammau\to \N$ such that $\Proj(\gamma)(u(\gamma))$ is the unit in $C^*(\gamma)$.
\end{lemma}

\begin{proof} For projections $p$ and $q$ we have that $p\leq q$ and
$p\neq q$ implies $\|p-q\|=1$. Therefore   $C^*(\gamma)$ is unital if
and only $\Proj(\gamma)(n)$ is its unit for some $n$. Also,  $p$ is a
unit in $A$ if and only if $pa=a=ap$ when $a$ ranges over a dense
subset of $A$. Therefore $C^*(\gamma)$ is unital if and only if there
is $m$ such that for all $n$ we have
\begin{equation}\label{eq.unit}
\Proj(\gamma)(m)\p_n(\gamma)=\p_n(\gamma)\Proj(\gamma)(m)=\p_n(\gamma).
\end{equation} To define $u:\Gammau\to\N$, simply let $u(\gamma)=m$ if and only if $m\in\N$ is least such that \ref{eq.unit} holds for all $n\in\N$.
\end{proof}

\begin{corollary}\label{c.paramunital}
The parameterization $\hat\Gamma_{\Au}$, $\Xi_{\Au}$, $\Gammau$, $\hatGammau$, $\Xiu$ and $\hatXiu$ of unital separable C$^*$-algebras all equivalent. 
\end{corollary}
\begin{proof}
It is clear from the previous Lemma and Propositions \ref{p.xihatxiequiv} and \ref{p.gammaxiequiv} that $\Gammau$, $\hatGammau$, $\Xiu$ and $\hatXiu$ are equivalent standard Borel parameterizations. On the other hand, it is easy to see that Lemma \ref{l.injection} hold for $Y=\hat\Gamma_{\Au}$, and so it is enough to show weak equivalence of $\Gammau$ and $\hat\Gamma_{\Au}$. In one direction, the natural map $\hat\Gamma_{\Au}\to\Gamma:f\to\gamma(f)$ given by $\gamma(f)(n)=f(\q_n)$ clearly works. The other direction can be proven by a GNS argument analogous to the proof of proposition \ref{p.gammaxiequiv}. 
\end{proof}

\subsection{Effros Borel structure}\label{S.Eff}
 If $X$ is a Polish space then $F(X)$ denotes the space of all 
closed subsets of $X$ equipped with the $\sigma$-algebra  generated by the sets
\[
\{K\in F(X): K\cap U\neq \emptyset\}
\]
for $U\subseteq X$ open. 
This is a standard Borel space (\cite[\S 12.C]{Ke:Classical}) 
and its subspaces are typically used as Borel spaces of 
separable Banach spaces, von Neumann algebras, etc. Since by a result of 
Junge and Pisier there is no universal separable C*-algebra (\cite{JunPis}), the space of subalgebras of a given separable C*-algebra cannot be a Borel space of all C*-algebras. 
However, the subspace of $F(\cO_2)$ consisting of subalgebras 
of $\cO_2$ (where $\cO_2$ is the Cuntz algebra with two generators) 
is, by  a result of Kirchberg,  a Borel space of all exact C*-algebras (see \S\ref{s.bga} and cf. 
Lemma~\ref{l.borelassign}). 

For $A\subseteq X\times Y$ and $x\in X$ let $A_x$ denote the (vertical) 
\emph{section} of $A$ at $x$, that is, $A_x=\{y: (x,y)\in A\}$. Below (and later) we will need the following well-known fact (see \cite[28.8]{Ke:Classical})

\begin{lemma}\label{L.Effros} Let $X$ and $Y$ be Polish spaces and assume that 
$A\subseteq X\times Y$ is Borel and all sections $A_x$ are compact. Then the set $A^+=\{(x,A_x): x\in X\}$ is a Borel subset of 
$X\times F(Y)$, and the map $x\mapsto A_x$ is Borel. 
\end{lemma} 


\subsection{Coding states}\label{ss.codingstates}

Roughly following \cite[\S 2]{Kec:C*}, we shall  describe a coding of states
on $C^*(\gamma)$. If $\phi$ is a functional on $C^*(\gamma)$ then,
being norm-continuous, it is uniquely determined by its restriction
to $\{\p_n(\gamma): n\in \bbN\}$. Also, writing $\Delta(r)=\{z\in
\bbC: |z|\leq r\}$ we have $\|\phi\|\leq 1$ if and only if for every
$n$ we have $\phi(\p_n(\gamma))\in \Delta(\|\p_n(\gamma)\|)$.
Therefore we can identify $\phi$ with $\hat\phi\in \prod_n \Delta(\|\p_n(\gamma)\|)$.
Clearly, the set of $\hat\phi$ such that $\phi$ is additive is compact in the product
metric topology. Since $\phi$ is positive if and only if $\phi(\p_n^*(\gamma)\p_n(\gamma))\geq 0$
for all $n$, the set of all states is also compact.
Similarly, the set of all traces is compact.
 By the obvious rescaling of the coordinates, we can identify $\prod_n \Delta(\|\p_n(\gamma)\|)$
 with $\DeltaN$ (writing $\Delta$ for $\Delta(1)$).
Consider the space $\bbK=K_c(\DeltaN)$ of compact 
subsets of $\DeltaN$ and its subspace  $\Kconv$ of of compact convex subsets of $\DeltaN$.

\begin{lemma} \label{L.SPT}
With the above identifications, there are Borel maps
$\bbS\colon \Gamma\to \bbK$, $\bbP\colon \Gamma\to \bbK$ and
$\bbT\colon \Gamma\to \bbK$ such that $\bbS(\gamma)$ is the set of
all states on $C^*(\gamma)$, $\bbP(\gamma)$ is the closure of the set of all pure
states on $C^*(\gamma)$ and $\bbT(\gamma)$ is the set of all tracial
states on $C^*(\gamma)$.
\end{lemma}

\begin{proof}
For $\bbS$ and $\bbT$ this is obvious from the above discussion and Lemma~\ref{L.Effros}. 
The existence of $\bbP$ can be proved by a proof similar to that of
\cite[Lemma~2.2]{Kec:C*}.
\end{proof}


\begin{lemma} \label{L.States} There are Borel maps
$\State\colon \Gamma\to (\DeltaN)^{\bbN}$, $\Pure\colon
\Gamma\to (\DeltaN)^{\bbN}$ and $\Trace\colon \Gamma\to
(\DeltaN)^{\bbN}$ such that $\State(\gamma)(m)$, for $m\in
\bbN$, is a dense subset of $\bbS(\gamma)$,  $\Pure(\gamma)(m)$, for
$m\in \bbN$  is a dense subset of $\bbP(\gamma)$ and
$\Trace(\gamma)(m)$, for $m\in \bbN$, is a dense subset of
$\bbT(\gamma)$. 
\end{lemma}

\begin{proof} For $\State$ and $\Trace$ this is a consequence of the
previous lemma and the Kuratowski--Ryll-Nardzewski Theorem 
 (\cite[Theorem~12.13]{Ke:Classical}). 
The  construction of the
map $\Pure$, was given in \cite[Corollary~2.3]{Kec:C*}.
\end{proof}


\section{Choquet and Bauer simplexes}
\label{S.choquet}

Let us first recall the pertinent definitions. 
All compact convex sets considered here will be metrizable, and 
therefore without a loss of generality subsets of the Hilbert cube.  
For such $S$ its \emph{extreme boundary}, denoted $\partial S$, 
is the set of its extremal points. By the Krein--Milman theorem 
$S$ is contained in the closure of the convex hull of $\partial S$. 
A {\it metrizable Choquet simplex} is a simplex $S$ as above with the following property:  for every point $x$ in 
$S$ there exists a unique probability boundary measure $\mu$
 (i.e., a measure concentrated on $\partial S$) 
such that $x$ is the barycentre of $\mu$. 
This notion has a number of equivalent definitions, see \cite[\S II.3]{alfsen71}. 
The isomorphism relation in the category of Choquet simplexes is affine homeomorphism.

The extreme boundary of a Choquet simplex $S$ 
is always $G_\delta$, and in the case that it is compact $S$ is said to be a \emph{Bauer simplex}. 
It is not difficult to see that in this case $S$ is isomorphic to the space $P(\partial S)$ of Borel probability measures on $\partial S$. In particular Bauer simplexes $S$ and $L$ are isomorphic if and only if 
their extreme boundaries $\partial S$ and $\partial L$ are homeomorphic. 

Let  $\Delta_{n}$ denote the $n$-simplex ($n\in\N$). 
Every metrizable Choquet simplex $S$ can be represented as 
an inverse limit of finite-dimensional Choquet simplexes 
\begin{equation}\label{invlim}
S\simeq\lim_{\leftarrow} (\Delta_{n_i},\psi_i),
\end{equation}
where and $\psi_i:\Delta_{n_i} \to \Delta_{n_{i-1}}$ is an affine surjection
 for each $i \in \mathbb{N}$.  
This was proved in \cite[Corollary to Theorem~5.2]{LaLi} and we shall prove a Borel 
version of  this result in Lemma~\ref{L.Choq.2}.

\subsubsection{Order unit spaces} \label{OrderUnit} 
Let $(A,A^+)$ be an ordered real Banach space.
Here $A^+$ is a cone in $A$ and the order is defined by $a\leq b$ if and only if $b-a\in A^+$. 
Such a space is \emph{Archimedean} if for every $a\in A$ the set $\{ra: r\in \bbR^+\}$ has 
an upper bound precisely when $a$ is \emph{negative}, i.e., $a\leq 0$. 
An element $1_A\in A$ is an \emph{order unit} if  for every $a\in A$ there is $r\in \bbR^+$ such that 
$-r1_A\leq a\leq r1_A$.  
We say that an Archimedean ordered vector space with a distinguished unit $(A,A^+,1_A)$ is 
an \emph{order unit space}, and define a norm on $A$ by 
\[
\|a\|=\inf\{r>0: -r1_A\leq a\leq r1_A\}. 
\]
Our interest in order unit spaces stems from the fact that the 
 category of separable complete order unit spaces is the dual category to the category of metrizable Choquet simplexes. For a Choquet simplex $S$, the associated dual object is 
 $\Aff(S)$, the real-valued affine functions on $S$, with the natural 
 ordering and order unit set to be the constant function with value 1.
 Conversely, given an order unit space $(A,A^+,1_A)$, the associated dual object is the space of positive real functionals $\phi$ on $A$ of norm one, with respect to the weak*-topology. 
 In the case of Bauer simplexes $S$  there is also a natural identification of the 
complete separable order unit spaces $\Aff(S)$ and $C_\mathbb{R}(\partial S)$ 
obtained by restriction. 
In particular,  
for the simplex $\Delta_n$  we have 
\[
\Aff(\Delta_n) \cong (\mathbb{R}^{n+1}, (\mathbb{R}^+)^{n+1}, (1,1,\ldots,1)).  
\]
Setting $e_0$ to be the origin in $\mathbb{R}^n$, the co-ordinate functions $f_k:\Delta_n \to \mathbb{R}$, $0 \leq k \leq n$, given by the formula
\[
f_k(e_i) = \left \{ \begin{array}{ll} 1 & i=k \\ 0 & i \neq k \end{array} \right.
\]
on vertices and extended affinely, form a canonical basis for $\Aff(\Delta_n)$.

Let $X$ and $Y$ be separable order unit spaces 
with order units $1_X$ and $1_Y$. 
Let as usual $L(X,Y)$ denote the set of linear, continuous maps, and let $L_1(X,Y)=\{T\in L(X,Y): \|T\|\leq 1\}$. The space $L_1(X,Y)$ is a Polish space when given the strong topology.  
The set of order unit preserving maps in $L_1$, 
$$
L_{\rm ou}(X,Y)=\{T\in L_1(X,Y): (\forall x,x'\in X) x\leq x'\implies T(x) \leq T(x')\wedge T(1_X)=1_Y\}
$$
is a closed subset of $L_1(X,Y)$, and is therefore Polish in its own right.  
(Our definition of $L_{\rm ou}$ involves some redundancy since it is 
 a standard fact that $T\in L_1(X,Y)$ such that $T(1_X)=1_Y$ is automatically order preserving.)

\subsection{Parameterizing metrizable Choquet simplexes and their duals.}\label{simplexbasic} 

\subsubsection{The space $\Lambda$} \label{S.Lambda}
If $X=\Aff(K)$ and $Y=\Aff(L)$ for metrizable Choquet simplexes $K$ and $L$, then $L_{\rm ou}(X,Y)$ is the set of morphisms dual to the affine continuous maps from $L$ to $K$.  It follows from (\ref{invlim}) that the separable complete order unit spaces all arise as direct limits of sequences 
\[
\mathbb{R}^{m_1} \stackrel{\phi_1}{\longrightarrow} \mathbb{R}^{m_2} \stackrel{\phi_2}{\longrightarrow} \mathbb{R}^{m_3} \stackrel{\phi_3}{\longrightarrow} \cdots
\]
with $\phi_n \in L_{\rm ou}(\mathbb{R}^{m_n},\mathbb{R}^{m_{n+1}})$.  
Since we can identify  an operator in 
$L_{\rm ou}(\mathbb{R}^{m_n},\mathbb{R}^{m_{n+1}})$ 
with its matrix, $L_{\rm ou}(\mathbb{R}^{m_n},\mathbb{R}^{m_{n+1}})$ 
is affinely homeomorphic with a closed 
subspace of  $m_n\times m_{n+1}$ matrices.
We can therefore parameterize the separable complete order unit spaces (and therefore their duals) using 
$$
\Lambda=\N^\N\times \prod_{(m,n)\in \N^2} L_{\rm ou}(\R^m,\R^n)
$$
in the following way:   each $(f,\psi)\in \Lambda$ corresponds to the limit $X(f,\psi)$ of the system
$$
\R^{f(1)}\underset{\psi(f(1),f(2))}{\longrightarrow} \R^{f(2)}\underset{\psi(f(2),f(3))}{\longrightarrow} \R^{f(3)}\underset{\psi(f(3),f(4))}{\longrightarrow}\cdots.
$$
Since $\Lambda$ is a Polish space with respect to  the product topology, we have what we will refer to as the \emph{standard Borel space of metrizable Choquet simplexes}.  We note that our parameterization is similar in spirit to that of $\Gamma$, as we identify our objects with something akin to a dense sequence.
This is a good Borel parameterization (see Definition~\ref{d.parameterization}). 



\subsubsection{The space $\LLambda$} \label{S.LLambda}
The following Borel space  
of metrizable Choquet simplexes was essentially defined by Lazar and 
Lindenstrauss in \cite{LaLi} where the emphasis was put on Banach spaces $X$ instead of 
the simplexes $B(X^*)$. Another difference is that in \cite{LaLi} the authors studied a wider class
of spaces whose dual is $L^1$. 

A simple analysis of an $n\times (n+1)$ matrix  
shows that, modulo permuting the basis of $\bbR^{n+1}$, 
every $\phi\in L_{\rm ou}(\bbR^n,\bbR^{n+1})$ is of the form  
\begin{equation}\label{representing}
\textstyle\phi (x_1,x_2,\dots, x_n)=(x_1,x_2,\dots, x_n, \sum_{i=1}^n a_{i} x_i)
\end{equation}
where $0\leq a_i \leq 1$ and $\sum_i a_i=1$ is in $L_{\rm ou}(\bbR^n,\bbR^{n+1})$.

A \emph{representing matrix} of a Choquet simplex 
is a matrix $(a_{ij})_{(i,j) \in \mathbb{N}^2}$ in which all entries are non-negative, 
 $\sum_{i=1}^n a_{in}=1$, and $a_{in}=0$ for $i>n$. 
By the above, such a matrix codes a directed  system 
\[
\mathbb{R}^{1} \stackrel{\phi_1}{\longrightarrow} \mathbb{R}^{2} \stackrel{\phi_2}{\longrightarrow} \mathbb{R}^{3} \stackrel{\phi_3}{\longrightarrow} \cdots
\]
where 
$\phi_n (x_1,x_2,\dots, x_n)=(x_1,x_2,\dots, x_n, \sum_{i=1}^n a_{in} x_i)$. 
A limit of this directed system is a Banach space $X$ and the unit ball of its dual is
a Choquet simplex with respect to its weak*-topology.  This is because an inverse limit of Choquet simplexes is again a Choquet simplex. We let $\LLambda$ denote the set of all representing matrices, which is a closed set when viewed as a subset of $[0,\infty)^{\N\times\N}$.

On p. 184 of \cite{LaLi} the authors refer to the Borel space of representing matrices 
 when they point out that 
``It seems to be a very difficult problem to determine the set of all representing matrices of a given separable infinite-dimensional predual of $L_1(\mu)$. 
We know the answer to this question only for one such space, namely the space of Gurarii   and even here the situation  is not entirely clear.''
Gurarii space is dual to the Poulsen simplex and the Lazar--Lindenstrauss characterization 
alluded to above implies that a dense $G_\delta$ set of representing matrices corresponds to 
the Poulsen simplex. (By removing zeros, 
here we identify the matrix $a_{in}$, $i\leq n\in \bbN$ with an element of $\prod_n [0,1]^n$.)
This can be taken as a remark about the Borel complexity of certain set, 
close to the point of view of the present paper or of~\cite{Kec:C*}.

\subsubsection{The space $\LLLambda$} \label{S.LLLambda}
Let $\delta_n=2^{-2n}$ and for 
 each $n$ 
consider the set of all $\phi\in L_{\rm ou}(\bbR^n,\bbR^{n+1})$ 
of the form \eqref{representing} such that all $a_{i}$  are 
of the form $k2^{-2n}$ for $k\in \bbN$. 
 Let  $\cF_n$ be the set of all  $n\times (n+1)$ matrices  
 representing such $\phi$. 
Modulo permuting basis of $\bbR^{n+1}$, the set $\cF_n$ is $\delta_n$-dense in 
$L_{\rm ou}(\bbR^n,\bbR^{n+1})$.

\begin{lemma} \label{L.Choq.-1}
 For all $m\leq n$ in $\bbN$ and every $\Phi\in L_{\rm ou}(\bbR^m,\bbR^n)$
 there are $F_i\in \cF_i$ for $m\leq i\leq n$  such that 
$F_{n-1}\circ \dots \circ F_{m+1}\circ F_m$  
is within $2^{-m}$
from $\Phi$ composed with a permutation of the canonical basis of $\bbR^n$ in the operator norm. 
\end{lemma} 

\begin{proof} 
The linear operator $\Phi$ 
is coded by an $n\times m$-matrix $(a_{ij})$ that has at least one entry equal to 1 in each column. 
After possibly re-ordering the basis, we may then assume $a_{ii}=1$ for all $i\leq m$. 
Furthermore, we can canonically write $\Phi$ as a composition of
$m-n$ operators 
\[
\Phi_{n-1}\circ \Phi_{n-2}\circ \dots \circ \Phi_{m}
\]
so that $\Phi_{k}\in L_{\rm ou}(\bbR^k,\bbR^{k+1})$, and the last row of the matrix of 
$\Phi_{k}$ is the $k$-th row of the matrix of $\Phi$ padded with zeros. 
 Now choose 
 $F_{n-1}, \dots, F_{m}$ in $\prod_{k=m}^{n-1} \cF_k$
 such that $\|F_k-\Phi_{k}\|<2^{-2k}$. 
Then  
  $F=F_{n-1}\circ  \dots\circ F_{m}$ 
is within $2^{-m}$ of $\Phi$ in the operator norm, as required. 
\end{proof}

Let $\LLLambda$ be the compact metric space  $\prod_n \cF_n$. 
By identifying $\psi\in \LLLambda$ with $(\id,\psi)\in \Lambda$, one sees that each element of $\LLLambda$ represents a Choquet simplex. 
We fix a well-ordering $\prec_{\cF}$ of finite sequences of elements of $\bigcup_n \cF_n$, to be 
used in the proof of Lemma~\ref{L.Choq.2}. 

\subsubsection{The space $\Kchoq$} Recall that $\Kconv$ is the space of all compact convex subsets of the Hilbert cube.  Let $\Kchoq$ denote the space of all Choquet 
simplexes in $K\in \Kconv$. In Lemma~\ref{L.Choq.2} we shall show that $\Kchoq$ is a 
Borel subspace of $\bbK$, and therefore $\Kchoq$ is the `natural' parameterization 
of Choquet simplexes.

\subsubsection{Our Borel parameterizations of Choquet simplexes are weakly equivalent}
Weak equivalence of Borel parameterizations was defined in 
(4') of Definition~\ref{d.parameterization}.

\begin{prop} \label{P.Choquet} 
The  four Borel parameterizations of  Choquet simplexes introduced above, 
$\Lambda$, 
$\LLambda$, 
  $\LLLambda$,
  and    $\Kchoq$, 
are all  weakly 
equivalent. 
  \end{prop}

A proof of Proposition~\ref{P.Choquet} will take up the rest of this section. 
Clearly the space  $\Kconv$  is a closed subset
of $K_c(\Delta^{\bbN})$. 
In the following consider the Effros Borel space $F(\CRD)$ of all closed subsets
of $\CRD$ (see \S\ref{S.Eff}).

Recall that a \emph{peaked partition of unity} in an order-unit space $(A.A^+,1_A)$
is a finite set $f_1,\dots, f_n$ of positive elements of $A$ such that $\sum_i f_i=1_A$
and $\|f_i\|=1$ for all $i$. A peaked partition of unity $P'$ \emph{refines} a peaked 
partition of unity $P$ if every element of $P$ is a convex combination of 
the elements of $P'$.  

We shall need two facts about real Banach spaces. 
For a separable Banach space $X$ let $S(X)$ denote the space of closed subspaces of $X$, 
with respect to the Effros Borel structure (\S\ref{S.Eff}). 
It was proved by Banach that $S(\CRD)$ is universal for separable Banach spaces, and therefore this space with respect to its Effros Borel structure can be considered as the standard Borel space of separable Banach spaces (see Lemma~\ref{L.dual.2}).  
Consider the space 
$\Kconv \subseteq K_c(\DeltaN)$ of compact convex
subsets of the Hilbert cube, $\DeltaN$. With respect to the Borel structure induced by the Hausdorff metric, this is
the standard Borel space of all compact convex metrizable spaces.
For a Banach space $X$ let $B(X^*)$ denote the unit ball of the dual of $X$, with respect to the weak*-topology. Then $B(X^*)$ is a compact convex space, and it is metrizable if $X$ is 
separable. 
The idea in the following is taken from of the proof of 
\cite[Lemma~2.2]{Kec:C*}.

\begin{lemma} \label{L.dual.1} 
If $X$ is a separable Banach space
then there is a Borel map $\Phi\colon S(X)\to \Kconv$ such that 
$\Phi(Y)$ is affinely homeomorphic to the unit ball $B(Y^*)$ of $Y^*$, with respect to 
its weak*-topology. 
\end{lemma} 

\begin{proof} By Kuratowski--Ryll-Nardzewski's theorem (\cite[Theorem~12.13]{Ke:Classical})
there are Borel  $f_n\colon S(X)\to X$ for $n\in \bbN$
 such that $\{f_n(Y): n \in \bbN\}$ 
is a dense subset of $Y$ for every $Y$. 

Fix an enumeration $F_n=(r_{n,i}: i\leq k_n)$, for $n\in \bbN$, of finite sequences of rationals. 
Define $h_n\colon S(X) \to X$ by 
\[
h_n(Y)=\sum_{i=1}^{k_n} r_{n,i} f_i(Y).
\]
Then $\{h_n(Y): n\in \bbN\}$ is a dense linear 
subspace of $Y$ for each $Y\in S(X)$. 

Let $\Delta(Y)=\prod_n [-\|h_n(Y)\|,\|h_n(Y)\|]$. 
Let $K(Y)$ be the set of all $\phi \in \Delta(Y)$ such that 
\begin{enumerate}
\item [(*)]  $F_i + F_j = F_l$ (where the sum is taken pointwise) implies $\phi(i)+\phi(j)=\phi(l)$, 
for all $i,j$ and $l$. 
\end{enumerate}
Such a $\phi$ defines a functional of norm $\leq 1$ on a dense subspace of $Y$, 
and therefore extends to an element of $B(Y^*)$.  
Moreover, every functional in $B(Y^*)$ is obtained in this way. 
Therefore the set of  $\phi$ satisfying (*) 
is affinely homeomorphic to $B(Y^*)$. 

It remains to rescale $K(Y)$.  
Let $\Phi(Y)=\{\phi \in \DeltaN: (\phi(n)\|h_n(Y)\|)_{n\in \bbN}\in K(Y)\}$. 
Then $\phi\in \Phi(Y)$ if and only if 
\[
\phi(i)\|h_i(Y)\|+ \phi(j)\|h_j(Y)\|=\phi(l)\|h_l(Y)\|
\]
for all triples $i,j,l$ satisfying $F_i + F_j = F_l$ (a condition not depending on $Y$). 

Since the map $y\mapsto \|y\|$ is continuous, the map $Y\mapsto \Phi(Y)$ 
is Borel, and clearly $\Phi(Y)$ is affinely homeomorphic to $K(Y)$ and therefore to $B(Y^*)$. 
\end{proof}


\begin{lemma} \label{L.dual.2} 
There is a Borel map  $\Psi\colon \Kconv\to S(\CRD)$ such that the 
Banach spaces $\AffR(K)$ and  $ \Psi(K)$ are isometrically isomorphic for all $K$.
\end{lemma} 

\begin{proof} 
 Identify $\DeltaN$ with $\prod_n [-1/n,1/n]$. Consider the compatible 
$\ell_2$ metric $d_2$ on $\DeltaN$ and  the set
\[
\cZ=\{(K,x,y): K\in \Kconv, x\in \DeltaN, y\in K, \text{ and  } d_2(x,y)=\inf_{z\in K} d_2(x,z)\}.
\]
Since the map $(K,x)\mapsto \inf_{z\in K} d_2(x,z)$ is continuous 
on $\{K\in \Kconv: K\neq\emptyset\}$, 
this set is closed. Also, for every pair $K,x$ there is the unique point $y$ such that 
$(K,x,y)\in \cZ$ (e.g., \cite[Lemma~3.1.6]{Pede:Analysis}). 
By compactness, the function $\chi$ that sends $(K,x)$ to the unique $y$ 
such that $(K,x,y)\in \cZ$ is 
continuous. Fix a continuous surjective map $\eta\colon \Delta\to \DeltaN$.  
Then  $\chi_K(x)=\chi(K,\eta(x))$ defines a continuous surjection 
from $\Delta$ onto $K$  and $K\mapsto \chi_K$ is a continuous
map from $\Kconv$ into $C(\Delta, \DeltaN)$ with respect to the uniform metric. 
The set 
\[
\cY=\{(K,f)\in \Kconv\times C_{\bbR}(\Delta) :  
\text{ for some }
g\in \AffR(K)\}
\]
 is closed. 
To see this, note that $(K,f)\notin \cY$ iff one of the following two conditions
happens: 
\begin{enumerate}
\item There are $x$ and $y$ such that $f(x)\neq f(y)$ but $\chi(K,f)(x)=\chi(K,f)(y)$, or
\item There are $x,y,z$ and $0<t<1$ such that 
$f(tx+(1-t)y)\neq z$ but 
$$
t\chi(K,f)(x)+(1-t)\chi(K,f)(y)=\chi(K,f)(z).
$$ 
\end{enumerate}
We need to prove that the map that sends $K$ to $\cY_K=\{f: (K,f)\in \cY\}$ is Borel. 
Since $\cY_K$ is clearly isometric to $\Aff(K)$, this will conclude the proof. 

Let $g_n$, for $n\in \bbN$, be a countable dense subset of $\CRD$. 
By compactness  
\[
h_n(K)= (g_n\rs K)\circ \chi_K\circ \eta
\]
 is a continuous map from $\Kconv$ to $\CRD$ such that $h_n(K)\in \cY_K$. 
 Moreover, the set $\{h_n(K): n\in \bbN\}$  is dense in $\cY_K$ for every $K$. 
 Since $\cY_K\neq \emptyset$, 
 we conclude that the  map $\Psi(K)=\cY_K$ is Borel. 
This follows by \cite[12.14]{Ke:Classical} or directly by noticing that 
 $\Psi^{-1}(\{X\in S(\CRD): X\cap U\neq \emptyset\}=\bigcap_n h^{-1}(U)$
is Borel for every  open 
$U\subseteq \CRD$. 
 \end{proof} 

Let $\Psi\colon \Kconv\to S(\CRD)$ be the Borel-measurable 
map that sends $K$ to $\Aff(K)\subseteq \CRD$ from Lemma~\ref{L.dual.2}. 
 For every $K$ and  $n$ the set 
 \[
 \PPU_n(K)\subseteq (\CRD)^n
 \] 
of all $n$-tuples in $\Psi(K)$  forming a peaked partition of unity 
is closed, by compactness of $K$.  

The following lemma is a reformulation of \eqref{invlim}. 

\begin{lemma} \label{L.Choq.1} For a metrizable compact convex 
set $K$ the following are equivalent.  
\begin{enumerate}
\item\label{L.Choq.1.1}  $K$ is a Choquet simplex, 
\item\label{L.Choq.1.2}  for every finite $F\subseteq \Aff(K)$, every $\epsilon>0$
  and every peaked partition of unity $P$ in $\Aff(K)$  there is a peaked partition of unity $P'$
 that refines $P$ and is such that every element of $F$ is within $\e$ of the span of $P'$. \qed
\end{enumerate}
 \end{lemma}

Another equivalent condition, in which (2) is weakened to approximate refinement, 
follows from~\cite{Vill:Range} and it will be reproved
  during the course of the proof of Lemma~\ref{L.Choq.2} below.
 
\begin{lemma} \label{L.Choq.0}
The map from $\Kconv$ to $F(\CRD^n)$
that sends    $K$ to $ \PPU_n(K)$
is Borel for every fixed~$n$. 
\end{lemma}

\begin{proof} The set of all $(K,f_1,f_2,\dots, f_n)\in \Kconv\times (\CRD)^n$
such that $f_i\in \Psi(K)$ for $1\leq i\leq n$ and $\sum_{i\leq n} f_i\equiv 1$
is a relatively closed subset of the set of all $(K,f_1,\dots, f_n)$ such that $f_i\in \Psi(K)$ for 
all $i\leq n$, and the conclusion follows. 
\end{proof}

By  \cite[Theorem~12.13]{Ke:Classical} or the proof of Lemma~\ref{L.dual.2}  and the above 
we have Borel maps $h_n\colon \Kconv\to \CRD$
such that $\{h_n(K): n\in \bbN\}$ is a dense subset of $\Psi(K)$, and Borel maps $P_{i,n}\colon \Kconv\to (\CRD)^n$, for $i\in \bbN$, such that $\{P_{i,n}(K): i\in \bbN\}$ is a dense subset of $\PPU_n(K)$, 
 for every $K\in \Kconv$. Also fix  $h_i\colon \bbK\to \DeltaN$ 
such that $\{h_i(K): i\in \bbN\}$ is a dense subset of $K$ for all $K$.

\begin{lemma}\label{L.Choq.2} 
The set $\Kchoq$ is a Borel subset of $\bbK$. 
Moreover, there is a Borel map $\Upsilon\colon \Kchoq\to \LLLambda$ 
such that $\Upsilon(K)$ is a parameter for~$K$. 
\end{lemma}

\begin{proof} 
We shall prove both assertions simultaneously. 
Let $\e_i=i^{-2}2^{-i-4}$. 

Fix $K\in \Kchoq$ for a moment. Let us say that a partition of unity $P$ \emph{$\e$-refines} a partition of unity 
$P'$ if every element of $P'$ is within $\e$ of the span of $P$. 
By Lemma~\ref{L.Choq.1}, there are sequences $d(j)=d(j,K)$, 
$i(j)=i(j,K)$ and $n(j)=n(j,K)$, for $j\in \bbN$, 
such that for each $j$ we have
\begin{enumerate}
\item $P_{i(j+1), n(j+1)}(K)$ is in $\PPU_{d(j)}(K)$, 
\item  $\{h_i(K)\rs K: i\leq j\}$ and the restriction of all elements of $P_{i(j),n(j)}(K)$ to $K$  
are within $\e_j$ of the rational linear span of the restrictions of elements of 
$P_{i(j+1),n(j+1)}(K)$ to $K$, 
\item $i(j+1), n(j+1)$ is the lexicographically minimal pair for which (1) and (2) hold. 
\end{enumerate}
The set of all triples $(K,(i(j): j\in \bbN), (n(j): j\in \bbN))$ such that (1) and (2) hold
is Borel. Since a function is Borel if and only if its graph is Borel (\cite{Ke:Classical}), 
the function sending $K$ to $((i(j,K), n(j,K)): j\in \bbN)$ is Borel. 

Still having $K$ fixed, 
let us write $P_j$ for $P_{i(j,K), n(j,K)}(K)$. 
Since each $f\in P_j$ is within $\e_j$ of 
the span of $P_{j+1}$, 
 \cite[Lemma~2.7]{Vill:Range} implies 
 there is an isometry $\Phi_j\colon \Span(P_j)\to \Span(P_{j+1})$ such that 
$\|\Phi_j(f)-f\|<2^{-j}$ for all $f\in \Span(P_j)$. 
Using Lemma~\ref{L.Choq.-1}, we can fix  the $\prec_{\cF}$-least
 composition of operators in $\bigcup_n \cF_n$ (see \S\ref{S.LLambda}), $\psi(j)$, 
 that $2^{-j}$-approximates $\Phi_j$ in the operator norm. 
This  defines an element $\psi$ of $\LLLambda$. 
Again, the function that associates $\psi$ to $K$ is Borel since its graph is a Borel set.

It remains to prove that the  direct limit of $\bbR^{d(j)}$, for $j\in \bbN$,
determined by $\psi$  is isometric to $\Aff(K)$. 
For every fixed $k$ the sequence 
of linear operators 
$\psi(k+j)\circ \psi(k+j-1)\circ \dots \psi(k)$ for $j\in \bbN$ forms 
a Cauchy sequence in the supremum norm.  
Therefore the image of $P_k$ under this sequence converges to a peaked partition of unity, 
denoted by $Q_k$,  
of $\Aff(K)$. Then $Q_k$, for $k\in \bbN$, form a refining sequence of peaked partitions
of unity of $\Aff(K)$ such that the span of $\bigcup_ k Q_k$ is dense in $\Aff(K)$. 
Thus with the dependence of $\psi$ on $K$ understood, we have that $\Upsilon(K):=\psi$ is the required parameter for $K$ in $\LLLambda$.
\end{proof} 

The following Lemma will only be used later in Section \ref{borelreduction}, but we include it here as it fits thematically in this section.

\begin{lemma} \label{L.Choq.4} There is a Borel map $\Psi\colon K_c(\DeltaN)
\to \LLLambda$ such that $\Psi(K)$ represents a Choquet simplex 
affinely homeomorphic to the Bauer simplex $P(K)$. 
\end{lemma} 

\begin{proof} By Lemma~\ref{L.Choq.2} it suffices to define a Borel 
map $\Psi_0\colon K_c(\DeltaN)\to \Kchoq$ so that $\Psi_0(K)$ is affinely homeomorphic to $P(K)$ for all $K$. 
For each $K\in K_c(\DeltaN)$ the set $P(K)$ is affinely homeomorphic to  
 a closed convex subset $Y_K$ of $P(\Delta^{\bbN})$, by identifying each measure $\nu$  on $K$ 
with its canonical extension $\nu'$ to $\Delta^{\bbN}$, $\nu'(A)=\nu(A\cap K)$. 
Moreover, the map $K\mapsto Y_K$ is continuous with respect to the Hausdorff metric. 
 Fix an   affine homeomorphism of $P(\Delta^{\bbN})$ into 
$\Delta^{\bbN}$. For example, if $f_n$, for $n\in \bbN$, is a sequence uniformly dense 
in $\{f\colon \Delta^{\bbN}\to \Delta\}$ then take $\nu\mapsto (\int f_n d\nu: n\in \bbN)$. 
By composing the map $K\mapsto Y_K$ with this map we conclude the proof.  
\end{proof}

\begin{proof}[Proof of Proposition~\ref{P.Choquet}]
A Borel homomorphism from the parametrization $\Kchoq$ to $\LLLambda$ was given in Lemma~\ref{L.Choq.2}.
If $\psi\in \LLLambda$ then (possibly after permuting the basis of $\bbR^{n+1}$) 
each $\psi(n)$ defines $a_{1n}, \dots, a_{nn}$  as in \S\ref{S.LLambda}. 
Therefore we have a canonical Borel homomorphism from the parametrization $\LLLambda$ into $\LLambda$. 
This map is continuous, and even Lipschitz in the sense that $\psi(n)$ determines 
all $a_{in}$ for $i\leq n$. Similarly, every representing matrix in $\LLambda$ canonically 
defines a directed system in $\Lambda$.

We therefore only need to check that there is a Borel homomorphism from $\Lambda$ to $\Kchoq$.

Given $(f,\psi)\in \Lambda$, we define $K=K(f,\psi)$ as follows. With $f(0)=0$ let $k_n=\sum_{i=0}^n f(i)$. For $a\in \bbR^{\bbN}$ and $n\geq 0$ let $a_n=a\rs [k_n,k_{n+1})$  and identify $a\in \bbR^{\bbN}$ with $(a_n)_{n\in \bbN}$. Let $B=\{ a\in \bbR^{\bbN}: (\forall n) \psi_n (a_n)=a_{n+1}\}$. 
Then $B=B(f,\psi)$  is a  separable subspace of  $\ell_\infty$ closed in the product topology. 
Also, $(f,\psi)\mapsto B(f,\psi)\cap \Delta^{\bbN}$ is a continuous map from $\Lambda$ into the hyperspace of $\Delta^{\bbN}$, and therefore the map $(f,\psi)\mapsto B(f,\psi)$ is a Borel map from $\Lambda$ into $F(\bbR^{\bbN})$. 

Let $K(f,\psi)$ denote the unit ball $B^*(f,\psi)$ of the dual of the Banach space $B(f,\psi)$.  When equipped with the weak*-topology, $K(f,\psi)$ is affinely homeomorphic to the Choquet simplex represented by $(f,\psi)$. 
We complete the proof by applying Lemma~\ref{L.dual.1}.
\end{proof}







\section{The isomorphism relation for AI algebras}\label{borelreduction}

Recall that an \emph{approximately interval} (or \emph{AI}) C$^*$-algebra is a direct limit
$$
A = \lim_{\longrightarrow} (A_i,\phi_i),
$$
where, for each $i \in \mathbb{N}$, $A_i \cong F_i \otimes \mathrm{C}([0,1])$ for some finite-dimensional C$^*$-algebra $F_i$ and $\phi_i:A_i \to A_{i+1}$ is a $*$-homomorphism. In this section we will prove the following ($\Lambda$ is the space defined in~\S \ref{S.choquet} and 
 notation $X(f,\psi)$ was introduced in \S\ref{S.Lambda}):

\begin{theorem}\label{t.reduction}
There is a Borel function $\zeta: \Lambda\to \Gamma$ such
that for all $(f,\psi)\in \Lambda$,
\begin{enumerate}
\item $C^*(\zeta(f,\psi))$ is a unital simple AI algebra.
\item $(K_0(C^*(\zeta(f,\psi)), K_0^+(C^*(\zeta(f,\psi)), 1)\simeq (\Q,\Q^+, 1)$ and $K_1(C^*(\zeta(f,\psi)))\simeq \{1\}$.
\item If $T$ is the tracial state simplex of $C^*(\zeta(f,\psi))$ then $\Aff(T)\simeq X(f,\psi)$.
\end{enumerate}
\end{theorem}

We note that Theorem \ref{t.reduction} immediately implies Theorem \ref{mainintro}:

\begin{corollary}\label{c.reduction}
The following relations are Borel reducible to isomorphism of simple unital $AI$ algebras:
\begin{enumerate}
\item Affine homeomorphism of Choquet simplexes.
\item Homeomorphisms of compact Polish spaces.
\item For any countable language $\mathcal L$, the isomorphism relation $\simeq^{\Mod(\mathcal L)}$ on countable models of $\mathcal L$.
\end{enumerate}
Moreover, isomorphism of simple unital AI algebras is not classifiable by countable structures, and is not a Borel equivalence relation.
\end{corollary}
\begin{proof}
For (1), let $\zeta$ be as in Theorem \ref{t.reduction}. Since simple unital AI algebras are classified by their Elliott invariant and since $(\bbQ,\bbQ^+,1)$ has a unique state, it follows that $(f,\psi)\simeq^\Lambda (f',\psi')$ if and only if $C^*(\zeta(f,\psi))\simeq C^*(\zeta(f',\psi'))$.

For (2), note that by Lemma \ref{L.Choq.4}, homeomorphism of compact subsets of $[0,1]^\N$ is Borel reducible to affine homeomorphism in $\Lambda$.

(3) follows from (2) and \cite[4.21]{hjorth00}, where it was shown $\simeq^{\mathcal L}$ is Borel reducible to homeomorphism.

It was shown in \cite[4.22]{hjorth00} that homeomorphism of compact subsets of $\K$ is not classifiable by countable structures, and so by (2) neither is isomorphism of AI algebras. Finally, it was shown in \cite{frst89} that $\simeq^{\Mod(\mathcal L)}$ is not Borel when $\mathcal L$ consists of just a single binary relation symbol, and so it follows from (3) that isomorphism of simple unital AI algebras is not Borel.
\end{proof}

The strategy underlying the proof of Theorem \ref{t.reduction} is parallel to the main argument in \cite{thomsen}. As a first step, we prove the following:


\begin{lemma}\label{l.conversion}
There is a Borel map $\varsigma: \Lambda\to L_{\rm ou}(C_\R[0,1])^\N$ such that for all $(f,\psi)\in \Lambda$ we have 
\begin{equation}\label{eq.conversion}
X(f,\psi)\simeq \lim (C_\R[0,1],\varsigma(f,\psi)_n).
\end{equation}
\end{lemma}


\begin{proof}
Let $f_{1,1}\in C_\R[0,1]$ be the constant 1 function, and for each $n>1$ and $0\leq i\leq n-1$, let $f_{n,i}:[0,1]\to\R$ be the function such that
$$
f_{n,i}\left(\frac j {n-1}\right)=\left\{ \begin{array}{ll}
1 & \text{if } j=i\\
0 & \text{if } j\neq i
\end{array}\right.
$$
and which is piecewise linear elsewhere. Then $\mathcal P_n=\{f_{n,i}: 0\leq i\leq n\}$ is a peaked partition of unity. For each $n$, let $\eta_n:\R^n\to C_\R[0,1]$ be the linear map given on the standard basis $(e_i)$ of $\R^n$ by $\eta_n(e_i)=f_{n,i}$, and let $\beta_n: C_\R[0,1]\to \R^n$ be given by $\beta_n(f)_i=f(\frac i {n-1})$. Then $\eta_n$ and $\beta_n$ are order unit space homomorphisms and $\beta_n\circ\eta_n=\id_{\R^n}$. Define $\varsigma(f,\psi)_n=\eta_{f(n+1)}\circ \psi(f(n),f(n+1))\circ\beta_{f(n)}$ and note that $\varsigma$ is continuous, and so it is Borel. Since the diagram
$$
\xymatrixcolsep{4pc}\xymatrix{
C_\R([0,1]) \ar[r]^{\varsigma(f,\psi)_1}\ar[d]_{\beta_{f(1)}}  &   C_\R[0,1]  \ar[r]^{\varsigma(f,\psi)_2}\ar[d]_{\beta_{f(2)}} & C_\R[0,1]  \ar[r]^{\varsigma(f,\psi)_3}\ar[d]_{\beta_{f(3)}} & \ \cdots\\
\R^{f(1)} \ar[r]_{\psi(f(1),f(2))} & \R^{f(2)}    \ar[r]_{\psi(f(2),f(3))}  & \R^{f(3)}    \ar[r]_{\psi(f(3),f(4))} & {\ }\cdots
}
$$
commutes, \eqref{eq.conversion} holds.
\end{proof}

Before proceeding, we fix our notation and collect the key results from \cite{thomsen} that we need. We identify $C[0,1]\otimes \M_n(\C)$ and $\M_n(C[0,1])$ in the natural way. We call a *-homomorphism $\phi: \M_n(C[0,1])\to \M_m(C[0,1])$ a \emph{standard homomorphism} when there are continuous functions 
$$
f_1,\ldots, f_{\frac m n}: [0,1]\to [0,1]
$$ 
such that $\phi(g)=\diag(g\circ f_1,\ldots, g\circ f_{\frac m n})$. Following \cite{thomsen}, we will call the sequence $f_1,\ldots, f_{\frac m n}$ the \emph{characteristic functions} of the standard homomorphism $\phi$. The tracial state space of $\M_n(C[0,1])$ is canonically identifed with the Borel probability measures on $[0,1]$ (see \cite[p. 606]{thomsen}), and so we canonically identify $\Aff(T(\M_n(C[0,1])))$ and $C_\R[0,1]$.

The following Lemma collects the results from \cite{thomsen} that we need.


\begin{lemma}[Thomsen]\label{l.thomsen} \ 

\begin{enumerate}
\item Any AI algebra 
can be represented as an inductive limit $\lim_n (\M_n(C[0,1],\phi_n)$, where each $\phi_n$ is a standard homomorphism.
\item If $\phi: \M_n(C[0,1])\to \M_m(C[0,1])$ is a standard homomorphism with characteristic functions $f_1,\ldots, f_{\frac m n}$, then the induced order unit space homomorphism $\hat\phi: C_\R[0,1]\to C_\R[0,1]$ (under the natural identification with the tracial state spaces) is given by
$$
\hat\phi(g)=\frac n m \sum_{i=1}^{\frac m n} g\circ f_i.
$$
\item Let $\phi_i,\psi_i\in L_{\rm ou} (C_\R[0,1])$ be order unit morphisms $(i\in\N)$ and let $\delta_i\in \R_+$ be a sequence such that $\sum_{i=1}^\infty \delta_n<\infty$. Suppose there are finite sets $F_k\subseteq C_\R[0,1]$ such that
\begin{enumerate}
\item $F_k\subseteq F_{k+1}$ for all $k\in\N$;
\item $\bigcup_{k} F_k$ has dense span in $C_\R[0,1]$;
\item for all $f\in F_k$ there are $g,h\in F_{k+1}$ such that $\|\phi_i(f)-g\|,\|\psi_i(f)-h\|\leq \delta_{k+1}$ for all $i\leq k$;
\item for all $f\in F_k$ we have $\|\phi_k(f)-\psi_k(f)\|\leq \delta_k$.
\end{enumerate}
Then $\lim_{\to} (C_\R[0,1],\phi_i)$ and $\lim_{\to} (C_\R[0,1],\psi_i)$ are isomorphic as order unit spaces.

\item For any order unit homomorphism $\psi: C_\R[0,1]\to C_\R[0,1]$, $f_0\in C_\R[0,1]$, finite $F\subseteq C_\R[0,1]$, $n,k\in\N$ and $\varepsilon>0$ there is $m=m_0 n k\in\N$ and continuous $f_1,\ldots f_{\frac m n}:[0,1]\to [0,1]$ such that for all $g\in F$ we have
\begin{equation}\label{eq.standardapprox}
\|\psi(g)-\frac n m\sum_{i=1}^{\frac m n} g\circ f_i\|_\infty<\varepsilon.
\end{equation}
\end{enumerate}
\end{lemma}


\begin{proof}
(1) and (2) are simply restatements of Lemma 1.1 and Lemma 3.5 in \cite{thomsen}, while (3) follows immediately from \cite[Lemma 3.4]{thomsen}. For (4), note that by the Krein-Milman type theorem \cite[Theorem 2.1]{thomsen}, we can find a multiple $d$ of $k$ and continuous $\tilde f_1,\ldots,\tilde f_N: [0,1]\to [0,1]$ such that $\|\psi(g)- \sum_{i=1}^{N} \frac {n_i} {d} (g\circ \tilde f_i)\|_\infty<\varepsilon$ where $1\leq n_i\leq d$ satisfy $\sum_{i=1}^N n_i=d$. Let $m=dn$ and let $f_1,\ldots, f_{\frac m n}$ be the list of functions obtained by repeating $n_1$ times $\tilde f_1$, then $n_2$ times $\tilde f_2$, etc. Then \eqref{eq.standardapprox} is clearly satisfied.
\end{proof}

For the next lemma we refer back to \S \ref{S.Rhom} and Lemma~\ref{L.matrix}
 for the definition of the relation $\Rhomx{H_0,H_1}$ and the functions $M_n:\Gamma(H)\to \Gamma(H^n)$ and $\theta_n:\Gamma(H)\times(\N^{\N})^n\to\N^{\N}$.


\begin{lemma}\label{biglimit}
View $C[0,1]$ as multiplication operators on $H=L^2([0,1])$. Then there is an element $\gamma\in\Gamma(H)$ such that $C^*(\gamma)$ is equal to $C[0,1]$ and such that there are Borel maps
$$
d_N:L_{\rm ou}(C_\R[0,1])^\N\to \N \ \ \text{and} \ \ \Phi_N: L_{\rm ou}(C_\R[0,1])^\N\to \N^\N
$$
for all $N\in\N$, so that for all $\vec\varsigma\in L_{\rm ou}(C_\R[0,1])^\N$ we have:
\begin{enumerate}[$(I)$]
\item For all $N\in\N$ we have $(M_{d_N(\vec\varsigma)}(\gamma),M_{d_{N+1}(\vec\varsigma)}(\gamma),\Phi_N(\vec\varsigma))\in \Rhomx{H^{d_N(\vec\varsigma)},H^{d_{N+1}(\vec\varsigma)}}$.
\item The limit
$$
A_{\vec\varsigma}=\lim_N (C^*(M_{d_N(\vec\varsigma)_i}(\gamma)),\hat\Phi_N(\vec\varsigma))
$$
is a unital simple AI algebra, which satisfies 
$$
(K_0(A_{\vec\varsigma}),K_0^+(A_{\vec\varsigma}), [1_{A_{\vec\varsigma}}])\simeq (\Q,\Q^+,1),\qquad K_1(A_{\vec\varsigma})=\{1\}
$$
and
$$
\Aff(T(A_{\vec\varsigma}))\simeq \lim_N (C_\R[0,1],\vec\varsigma_N).
$$
\end{enumerate}
\end{lemma}


\begin{proof}
Fix a sequence of continuous functions $\lambda_n:[0,1]\to [0,1]$ which is dense in $C([0,1],[0,1])$ and such that $\lambda_1(x)=x$ and $\lambda_{2n}$ enumerates all rational valued constant functions with infinite repetition. Also fix a dense sequence $g_n\in C_\R[0,1]$, $n\in\N$, closed under composition with the $\lambda_n$ (i.e., for all $i,j\in\N$ there is $k\in\N$ such that $g_i\circ \lambda_j=g_k$.)

Pick $\gamma\in\Gamma(H)$ to consist of the operators on $H$ that correspond to multiplication by the $g_n$. Each $\lambda_n$ induces an endomorphism $\psi_{n,m}$ of $C^*(M_m(\gamma))$ by entry-wise composition. Let $\Psi_{n,m}:\N\to\N^\N$ enumerate a sequence of codes corresponding to the $\psi_{n,m}$. These may even be chosen so that $\Psi_{n,m}(l)$ is always a constant sequence since we assumed that the sequence $(g_n)$ is closed under composition with the $\lambda_k$.

Define for each $N\in\N$ a relation $R_N\subseteq L_{\rm ou}(C_\R[0,1])\times\N\times\N\times\Q_+\times\N\times \N^{<\N}$ by
\begin{align*}
R_N(\psi,n,k,\varepsilon,m,t)\iff 
&\frac m {nk}\in\N\wedge \length(t)=\frac m n\wedge t(1)=1\wedge t(2)=2N\wedge \\
&(\forall j\leq N)\|\psi(g_j)-\frac n m \sum_{i=1}^{\frac m n} g_j\circ \lambda_{t(i)}\|_\infty<\varepsilon
\end{align*}
Note that this is an open relation in the product space when $L_{\rm ou}(C_\R[0,1])$ has the strong topology (and $\N$, $\Q_+$ and $\N^{<\N}$ have the discrete topology.) By Lemma \ref{l.thomsen}.(4) it holds that for all $\psi$, $n$, $k$ and $\varepsilon$ there is $m$ and $t$ such that $R_N(\psi,n,k,\varepsilon,m,t)$ holds. (Note that this still holds although we have fixed the first two elements of the sequence $t$, since $m$ can be picked arbitrarily large.) Let $t_N(\psi,n,k,\varepsilon)$ be the lexicographically least $t\in \N^{<\N}$ such that $R_N(\psi,n,k,\varepsilon,n\length(t),t)$ holds. We let $m_N(\psi,n,k,\varepsilon)=n\length(t)$, and note that $t_N$ and $m_N$ define Borel functions.

Fix a sequence $(\delta_i)_{i\in\N}$ in $\Q_+$ such that $\sum_{i=1}^\infty \delta_i<\infty$. Let $q_i\in\N$ enumerate the primes with each prime repeated infinitely often. We can then define Borel functions $G_N :L_{\rm ou}(C_\R[0,1])^\N\to\N$, $d_N: L_{\rm ou}(C_\R[0,1])^\N\to \N$, $\mathbf{d}_N: L_{\rm ou}(C_\R[0,1])^\N\to\N$ and $s_N: L_{\rm ou}(C_\R[0,1])\to \N^{<\N}$ recursively such that the following is satisfied:
\begin{enumerate}[\indent $(A)$]
\item $G_1$,$d_1$ and $\mathbf {d}_1$ are the constant 1 functions, $s_1$ is constantly the empty sequence.
\item $G_{N+1}(\vec\varsigma)$ is the least natural number $k$ such that for all $i\leq N$ and $j\leq G_{N}(\vec\varsigma)$ there are $j_0,j_1\leq k$ such that
$$
\|\sum_{l=1}^{\mathbf{d}_N} g_j\circ \lambda_{s_i(\vec\varsigma)_l}-g_{j_0}\|\leq\delta_N
$$
and
$$
\|\vec\varsigma_i(g_j)\circ \lambda_{s_i(\vec\varsigma)_l}-g_{j_1}\|\leq\delta_N.
$$

\item $d_{N+1}(\vec\varsigma)=m_{G_{N+1}(\vec\varsigma)}(\vec\varsigma_{N+1}, d_N(\vec\varsigma), q_1\cdots, q_N,\delta_{N+1})$.

\item $s_{N+1}(\vec\varsigma)=t_{G_{N+1}(\vec\varsigma)}(\vec\varsigma_{N+1},d_N(\vec\varsigma),q_1\cdots q_N,\delta_{N+1})$.

\item $\mathbf{d}_{N+1}(\vec\varsigma)=\frac {d_{N+1}(\vec\varsigma)}{d_{N}(\vec\varsigma)}=\length(s_{N+1}(\vec\varsigma))$.
\end{enumerate}
Note that $\mathbf{d}_N$ takes integer values by the definition of $d_N$. Define
$$
\Phi_N(\vec\varsigma)=\theta_{\mathbf{d}_{N+1}(\vec\varsigma)}(M_{d_N(\vec\varsigma)}(\gamma),\Psi_{s_{N+1}(\vec\varsigma)_1,d_N(\vec\varsigma)},\ldots, \Psi_{s_{N+1}(\vec\varsigma)_{\mathbf{d}_{N+1}(\vec\varsigma)},d_N(\vec\varsigma)}).
$$
Then $\Phi_N$ and $d_N$ are Borel functions for all $N\in\N$, and (I) of the Lemma holds by definition of $\theta_n$.

We proceed to prove that (II) also holds. Fix $\vec\varsigma\in L_{\rm ou}(C_\R[0,1])^\N$. Note that the inductive system $(C^*(M_{d_N(\vec\varsigma)}(\gamma), \hat\Phi_N(\vec\varsigma))$ is isomorphic to the system $(\M_{d_N(\vec\varsigma)}(C[0,1]),\phi_N)$ where
$$
\phi_N(f)=\diag(f\circ\lambda_{s_{N+1}(\vec\varsigma)_1},\ldots,f\circ \lambda_{s_{N+1}(\vec\varsigma)_{\mathbf{d}_{N+1}(\varsigma)}}).
$$
Since each natural number divides some $d_N(\vec\varsigma)$ we have
$$
(K_0(A_{\vec\varsigma}),K_0^+(A_{\vec\varsigma}), [1_{A_{\vec\varsigma}}])\simeq (\Q,\Q^+,1)
$$
while $K_1(A_{\vec\varsigma})=\{1\}$ since $[0,1]$ is contractible.

To establish that $\Aff(T(A_{\vec\varsigma}))\simeq \lim_i (C_\R[0,1],\vec\varsigma_i)$ we apply Lemma \ref{l.thomsen}. By Lemma \ref{l.thomsen}.(2) the order unit space morphism induced by $\phi_N$ is given by
$$
\hat\phi_N(f)=\frac 1 {\mathbf{d}_{N+1}(\vec\varsigma)} \sum_{i=1}^{\mathbf{d}_{N+1}(\vec\varsigma)} f\circ \lambda_{s_{N+1}(\vec\varsigma)_i}.
$$
Letting $F_N=\{g_i:i\leq G_N(\vec\varsigma)\}$, it is clear that (a) and (b) of Lemma \ref{l.thomsen}.(3) are satisfied. That (c) of \ref{l.thomsen}.(3) then also is satisfied for the sequences $\hat\phi_N,\vec\varsigma_N\in C_\R[0,1]$ follows from property (B) above. Finally, \ref{l.thomsen}.(3).(d) holds by (D) and the definition of $t_N$ and $R_N$. Thus 
$$
\lim_i(C_\R[0,1],\vec\varsigma_i)\simeq \lim_i(C_\R[0,1],\hat\phi_i)\simeq\Aff(T(A_{\vec\varsigma})).
$$
It remains only to verify that $A_{\vec\varsigma}$ is simple. For this we need only prove that if $0 \neq f \in \M_{d_N(\vec\varsigma)}(C[0,1])$, then for all $t\in [0,1]$ we have
$$
\phi_{N,j}(f) := \left( \phi_{j-1} \circ \phi_{j-2} \circ \cdots \circ \phi_N \right)(f)
$$
is nonzero at $t$ for some (and hence all larger) $j \geq N$.  By the definition of the sequence $(\lambda_n)$, there is some $j \geq l$ such that $f\circ\lambda_{2j}\neq 0$.  By the definition of the relations $R_n$, $f$ is a direct summand of $\phi_{l,j}(f)$, and so the constant function $f\circ\lambda_{2j} \neq 0$ is a direct summand of $\phi_{N,j+1}$.  This implies $\phi_{N,j+1}(f)(t) \neq 0$ for each $t \in [0,1]$, as required.
\end{proof}

\begin{proof}[Proof of Theorem \ref{t.reduction}]
Combine Lemma \ref{l.conversion} with Lemma \ref{biglimit}.
\end{proof}

\begin{corollary} There is a Borel measurable map $\Phi$ from $\{\gamma: C^*(\gamma)$ is 
unital and abelian$\}$ into $\{\gamma: C^*(\gamma)$ is simple and unital AI$\}$ such that $C^*(\gamma)\cong C^*(\gamma')$ 
if and only if $C^*(\Phi(\gamma))\cong C^*(\Phi(\gamma'))$. 

In other words,  unital abelian C*-algebras can be effectively classified by simple, unital AI algebras. 
\end{corollary}

\begin{proof} 
By Gelfand--Naimark duality a unital abelian C*-algebra $A$ is isomorphic to $C(\bbP(A))$, where $\bbP(A)$ denotes the pure states of $A$.    
We therefore only need to compose three Borel maps: 
The map taking the algebra $A$ to the space of its pure states (Lemma~\ref{L.SPT}), 
the map taking a compact Hausdorff space $X$ to the Bauer simplex $P(X)$ 
(Lemma~\ref{L.Choq.4}), and the map from the space of Choquet simplexes into the set of AI-algebras that was defined in Theorem~\ref{t.reduction}. 
\end{proof}


\section{A selection theorem for exact C$^*$-algebras}\label{s.exact}

For $2\leq n<\infty$, we will denote by $\mathcal O_n$ the Cuntz algebra generated by $n$ isometries $s_1,\ldots, s_n$ satisfying $\sum_{i=1}^n s_is_i^*=1$ (see \cite[4.2]{Ror:Classification}.)

Kirchberg's exact embedding Theorem states that the exact separable C$^*$-algebras are precisely those which can be embedded into $\mathcal O_2$. The purpose of this section is to prove a Borel version of this: There is a Borel function on $\Gamma$ selecting  an embedding of $C^*(\gamma)$ into $\mathcal O_2$ for each $\gamma\in\Gamma$ that codes an exact C$^*$-algebra. In the process we will also see that the set of $\gamma\in\Gamma$ such that $C^*(\gamma)$ is exact forms a Borel set.


\subsection{Parameterizing exact C$^*$-algebras.}\label{ss.exact} There is a multitude of ways of parameterizing exact separable C$^*$-algebras, which we now describe. Eventually, we will see that they are all equivalent good standard Borel parameterizations.

Define
$$
\Gammaex=\{\gamma\in\Gamma: C^*(\gamma)\text{ is exact}\},
$$
and let $\Gammaexu=\Gammaex\cap\Gammau$ denote the set of unital exact C$^*$-algebras\footnote{The sets $\Gammaex$ and $\Gammaexu$ are prima facie analytic, but since we will show they are Borel, the use of the language of Definition \ref{d.parameterization} is warranted.}. An alternative parameterization of the exact separable C$^*$-algebras is given by elements of $\Gamma(\mathcal O_2)=\mathcal O_2^\N$, equipped with the product Borel structure, where we identify $\gamma\in\mathcal O_2^\N$ with the C$^*$-subalgebra generated by this sequence. Let $\Gammau(\mathcal O_2)$ denote the set of $\gamma\in\Gamma(\mathcal O_2)$ which code unital C$^*$-subalgebras of $\mathcal O_2$.

Note that a parameterization weakly equivalent to $\Gamma(\mathcal O_2)$ is obtained by considering in the Effros Borel space $F(\mathcal O_2)$ of closed subsets of $\mathcal O_2$, the (Borel) set
$$
\SA(\mathcal O_2)=\{A\in F(\mathcal O_2): A\text{ is a sub-C}^*\text{-algebra  of } \mathcal O_2\}.
$$
Recall the parameterization $\Xi_{\Au}$ of unital separable C$^*$-algebras from \ref{ss.paramunital}. We define $\XiAuex$ to be the subset of $\Xi_{\Au}$ corresponding to exact unital C$^*$-algebras. Recall also that $\mathfrak A$ is the free countable unnormed $\Q(i)$-$*$-algebra, $\Au$ the unital counterpart. Define
$$
\hat\Gamma_{\mathfrak A}(\mathcal O_2)=\{\xi:\mathfrak A\to\mathcal O_2: \xi\text{ is a $\Q(i)$-$*$-algebra homomorphism } \mathfrak A\to\mathcal O_2\}
$$
and
$$
\hat\Gamma_{\Au}(\mathcal O_2)=\{\xi:\Au\to\mathcal O_2: \xi\text{ is a unital $\Q(i)$-$*$-algebra homomorphism } \Au\to\mathcal O_2\},
$$
and note that $\hat\Gamma_{\mathfrak A}(\mathcal O_2)$ and $\hat\Gamma_{\Au}(\mathcal O_2)$ are closed (and therefore Polish) in the subspace topology, when $\mathcal O_2^{\mathfrak A}$ and $\mathcal O_2^{\Au}$ are given the product topology. As previously noted, $\mathfrak A$ can be identified with the set of formal $\Q(i)$-$*$-polynomials $\mathfrak p_n$ in the formal variables $X_i$ without constant term, and $\Au$ with the formal $\Q(i)$-$*$-polynomials (allowing a constant term), which we enumerated as $\q_n$. We define $g:\hat\Gamma_{\Au}(\mathcal O_2)\to\Xi_{\Au}$ by $g(\xi)(\q_n)=\|\xi(\q_n)\|_{\mathcal O_2}$. Note that $g$ is continuous. By the exact embedding Theorem we have $g(\hat\Gamma_{\Au}(\mathcal O_2))=\XiAuex$.


Define an equivalence relation $E^g$ in $\hat\Gamma_{\Au}(\mathcal O_2)$ by
$$
\xi E^g\xi'\iff g(\xi)=g(\xi').
$$
For $\xi\in\hat\Gamma_{\Au}(\mathcal O_2)$, a norm is defined on $\Au / \ker(\xi)$ by letting $\|\mathfrak q_n\ker(\xi)\|_{\xi}=\|\xi(\mathfrak q_n)\|_{\mathcal O_2}$. We define $A_u(\xi)$ to be the unital C$^*$-algebra obtained from completing $(\Au,\|\cdot\|_\xi)$, and we note that $\xi$ extends to an injection $\bar\xi:A_u(\xi)\to\mathcal O_2$. It is clear that the definition of $A_u$ is $E^g$-invariant.


\begin{prop}\label{pr.selector}
With notation as above, there is a Borel set in $\hat\Gamma_{\Au}(\mathcal O_2)$ meeting every $E^g$ class exactly once (i.e., there is a Borel \emph{transversal} for $E^g$). 
\end{prop}

Before giving the proof, we first prove two general lemmas.


\begin{lemma}\label{l.borelassign}
Let $X,Y$ be Polish spaces. Suppose $B\subseteq X\times Y$ is a Borel relation such that for all $x\in X$ the section $B_x$ is closed (and possibly $\emptyset$.) Then the following are equivalent:
\begin{enumerate}
\item The map $X\to F(Y):x\mapsto B_x$ is Borel;
\item $\proj_X(B)$ is Borel and there are Borel functions $f_n:\proj_X(B)\to Y$ such that for all $x\in\proj_X(B)$ we have $f_n(x)\in B_x$ and $(f_n(x))_{n\in\N}$ enumerates a dense sequence in $B_x$;
\item the relation $R\subseteq X\times\N\times\Q_+$ defined by
$$
R(x,n,\varepsilon)\iff (\exists y\in Y) y\in B_x\wedge d(y,y_n)<\varepsilon
$$
is Borel for some (any) complete metric $d$ inducing the topology on $Y$ and $(y_n)_{n\in\N}$ dense in $Y$.
\end{enumerate}
In particular, if any of (1)--(3) above hold, there is a Borel function $F_0:X\to Y$ such that $F_0(x)\in B_x$ for all $x\in X$, and $F_0(x)$ depends only on $B_x$.
\end{lemma}

\begin{proof}
The equivalence of the first two is well-known, see \cite[12.13 and 12.14]{Ke:Classical}. Clearly (2) implies (3) since
$$
R(x,n,\varepsilon)\iff (\exists i) d(y_n,f_i(x))<\varepsilon.
$$
To see (3)$\implies$(1), simply notice that (3) immediately implies that for all $n\in\N$ and $\varepsilon\in\Q_+$ the set
$$
\{x\in X: B_x\cap \{y\in Y: d(y,y_n)<\varepsilon\}\neq\emptyset\}
$$
is Borel. This shows that the inverse images under the map $x\mapsto B_x$ of the sets $\{F\in F(Y):F\cap\{y\in Y:d(y,y_n)<\varepsilon\}\neq\emptyset\}$  for $n\in\N$, and $\varepsilon\in\Q_+\}$ are Borel, and since the latter sets generate the Effros Borel structure, it follows that the map $X\to B(Y):x\mapsto F_x$ is Borel, as required.

Finally, the last statement follows from (1) and the Kuratowski--Ryll-Nardzewski Theorem (\cite[Theorem~12.13]{Ke:Classical}).
\end{proof}


\begin{lemma}\label{l.select}
Let $X,Y$ and $B\subseteq X\times Y$ be as in Lemma \ref{l.borelassign}, and suppose moreover that $\proj_X(B)$ is Borel. Let $G$ be a Polish group, and suppose there is a continuous $G$-action on $Y$ such that the sets $B_x$ are $G$-invariant for all $x\in X$, and that for all $(x,y)\in B$ we have that the $G$-orbit of $y\in B_x$ is dense in $B_x$. Let $d$ be a complete metric on $Y$ and let $y_n$ be dense in $Y$. Then $R$ defined as in the previous Lemma is Borel, and so in particular (1) and (2) hold for $B$.
\end{lemma}

\begin{proof}
It is clear from the definition that
$$
R(x,n,\varepsilon)\iff (\exists y\in Y) y\in B_x\wedge d(y_n,y)<\varepsilon
$$
is an analytic set. To see that it is in fact Borel, fix a dense sequence $g_n\in G$. Then since all $G$-orbits are dense in $B_x$ we also have
$$
R(x,n,\varepsilon)\iff x\in\proj_X(B)\wedge (\forall y\in Y) y\notin B_x\vee (\exists i) d(g_i\cdot y,y_n)<\varepsilon,
$$
which gives a co-analytic definition of $R$, so that $R$ is Borel.
\end{proof}


We now turn to the proof of Proposition \ref{pr.selector}. Recall that if $A,B$ are C$^*$-algebras, $B$ is unital, and $\varphi_0,\varphi_1:A\to B$ are $*$-homormorphisms, we say that $\varphi_0$ and $\varphi_1$ are \emph{approximately unitarily equivalent} if for all finite $F\subseteq A$ and all $\varepsilon>0$ there is a unitary $u\in B$ such that $\|u^*\varphi_0(x)u-\varphi_1(x)\|<\varepsilon$ for all $x\in F$.

\begin{proof}[Proof of Proposition \ref{pr.selector}]
Let $U(\mathcal O_2)$ denote the unitary group of $\mathcal O_2$. The group $U(\mathcal O_2)$ acts continuously on $\hat\Gamma_{\Au}(\mathcal O_2)$ by
$$
u\cdot\xi(\q_n)=u^*\xi(\q_n)u=\Ad_u(\xi(\q_n)),
$$
and this action preserves the equivalence classes of $E^g$. Further, it is clear that $E^g$ is closed as a subset of $\hat\Gamma_{\Au}(\mathcal O_2)^2$.

We claim that for all $\xi\in\hat\Gamma_{\Au}(\mathcal O_2)$, the $U(\mathcal O_2)$-classes in $[\xi]_{E^g}$ are dense. To see this, let $\xi'E^g \xi$, and let $\bar\xi: A_u(\xi)\to\mathcal O_2$, $\bar\xi': A_u(\xi')\to\mathcal O_2$ be the injections defined before Proposition \ref{pr.selector}. Since $A_u(\xi)=A_u(\xi')$, it follows by \cite[Theorem 6.3.8]{Ror:Classification} that $\bar\xi$ and $\bar\xi'$ are approximately unitarily equivalent, and so we can find $u\in U(\mathcal O_2)$ such that $u\cdot\xi$ is as close to $\xi'$ as we like in $\mathcal O_2^{\mathfrak A}$.

Applying Lemma \ref{l.select} and \ref{l.borelassign}, we get a Borel function $F_0:\hat\Gamma_{\Au}(\mathcal O_2)\to \hat\Gamma_{\Au}(\mathcal O_2)$ selecting a unique point in each $E^g$-class. Then the set $F_0(\hat\Gamma_{\Au}(\mathcal O_2))=\{\gamma\in\hat\Gamma_{\Au}(\mathcal O_2): F_0(\gamma)=\gamma\}$ 
 is clearly a Borel transversal. 
\end{proof}

From Proposition \ref{pr.selector} we can obtain a Borel version of Kirchberg's exact embedding theorem. We first need a definition.


\begin{definition} Let $A$ be a separable C$^*$-algebra and $\gamma\in\Gamma$. Call $\Psi:\N\to A$ a \emph{code} for an embedding of $C^*(\gamma)$ into $A$ if for all $n,m,k\in\N$ we have:
\begin{enumerate}[\indent (1)]
\item If $\mathfrak p_m(\gamma)+\mathfrak p_n(\gamma)=\mathfrak \p_k(\gamma)$ then $\Psi(m)+\Psi(n)=\Psi(k)$;
\item if $\mathfrak p_m(\gamma)=\mathfrak p_n^*(\gamma)$ then $\Psi(m)=\Psi(n)^*$;
\item if $\mathfrak p_m(\gamma)\mathfrak p_n(\gamma)=\mathfrak \p_k(\gamma)$ then $\Psi(m)\Psi(n)=\Psi(k)$;
\item $\|\Psi(m)\|_{A}=\|\mathfrak p_m(\gamma)\|$.
\end{enumerate}
It is clear that if $\Psi:\N\to A$ is such a code then there is a unique $*$-monomorphism $\hat\Psi:C^*(\gamma)\to A$ satisfying $\hat\Psi(\p_n(\gamma))=\Psi(\p_n(\gamma))$. If $A$ is unital with unit $1_A$ and $C^*(\gamma)$ is unital, and $\Psi$ further satisfies
\begin{enumerate}[\indent (5)]
\item if $1_{C^*(\gamma)}$ is the unit in $C^*(\gamma)$ then $\hat\Psi(1_{C^*(\gamma)})=1_{A}$
\end{enumerate}
then we will call $\Psi$ a code for a \emph{unital} embedding into $A$.
Let $P^A\subseteq\Gamma\times A^\N$ be the relation
$$
P^A(\gamma,\Psi)\iff \Psi \text{ is a code for an embedding into } A
$$
and, assuming $A$ is unital, let $P^A_{\text{u}}\subseteq \Gammau\times A^\N$ be
$$
P^A_{\text{u}}(\gamma,\Psi)\iff \Psi \text{ is a code for a unital embedding into } A.
$$
\end{definition}

We note that the sections $P^A_\gamma$ and $(P^A_u)_\gamma$ are closed for all $\gamma\in\Gamma$.


\begin{theorem}[Borel Kirchberg exact embedding Theorem, unital case]

\ 

\begin{enumerate}[\indent\rm (1)]

\item The sets $\Gammaexu$, $\XiAuex$, $\Gammau(\mathcal O_2)$ and $\hat\Gamma_{\Au}(\mathcal O_2)$ are Borel and provide equivalent good parameterizations of the unital separable exact $C^*$-algebras.

\item There is a Borel function $f:\Gammaexu\to\mathcal O_2^\N$ such that $f(\gamma)$ is a code for a unital embedding of $C^*(\gamma)$ into $\mathcal O_2$ for all $\gamma\in\Gammaexu$. In other words, the relation $P_{\text{u}}^{\mathcal O_2}$ admits a Borel uniformization.
\end{enumerate}
\end{theorem}
\begin{proof}
(1) Let $T$ be a selector for $E^g$ as guaranteed by Proposition \ref{pr.selector}. Then $g:\hat\Gamma_{\Au}(\mathcal O_2)\to\Xi_{\Au}$ is injective on $T$, and so $\ran(g)=\ran(g\restrict T)=\XiAuex$ is Borel, and admits a Borel right inverse $h:\XiAuex\to\hat\Gamma_{\Au}(\mathcal O_2)$. Since Lemma \ref{l.injection} also holds with $Y=\Xi_{\Au}$, there is a Borel injection $\tilde g:\hat\Gamma_{\Au}(\mathcal O_2)\to\Xi_{\Au}$ such that $g(\xi)\simeq^{\Xi_{\Au}} \tilde g(\xi)$ for all $\xi\in \hat\Gamma_{\Au}(\mathcal O_2)$, and this shows that $\hat\Gamma_{\Au}(\mathcal O_2)$ and $\XiAuex$ are equivalent good parameterizations.

Since $\Gammau$ and $\Xi_{\Au}$ are equivalent (Corollary \ref{c.paramunital}), any witness to this is also a witness to that $\Gammaexu$ and $\XiAuex$ are equivalent, in particular $\Gammaexu$ is also Borel. Finally, by fixing a faithful representation of $\mathcal O_2$ on the Hilbert space $H$ we obtain a Borel injection of $\Gammau(\mathcal O_2)$ into $\Gammaexu(H)$, while on the other hand there clearly is a natural Borel injection from $\hat\Gamma_{\Au}(\mathcal O_2)$ into $\Gammau(\mathcal O_2)$. This finishes the proof of (1).

(2) Arguing exactly as in the proof of Proposition \ref{pr.selector}, the action of $U(\mathcal O_2)$ on the sections of $P^{\mathcal O_2}_{\text{u}}$ satisfy Lemma \ref{l.select}, since any two injective unital embeddings of $C^*(\gamma)$ are approximately unitarily equivalent for $\gamma\in\Gammaexu$.
\end{proof}

Since by Lemma \ref{L.unitization} the map that assigns to $\gamma\in\Gamma$ its unitization is Borel, we obtain:

\begin{theorem}[Borel Kirchberg exact embedding Theorem, non-unital case]

\ 

\begin{enumerate}[\indent\rm (1)]

\item The sets $\Gammaex$, $\Gamma(\mathcal O_2)$ and $\hat\Gamma_{\mathfrak A}(\mathcal O_2)$ are Borel and provide equivalent parameterizations of the separable exact $C^*$-algebras.

\item There is a Borel function $f:\Gammaex\to\mathcal O_2^\N$ such that $f(\gamma)$ is a code for a embedding of $C^*(\gamma)$ into $\mathcal O_2$ for all $\gamma\in\Gammaex$. In other words, the relation $P^{\mathcal O_2}$ admits a Borel uniformization.
\end{enumerate}
\end{theorem}


\section{Below a group action}\label{s.bga}

Conjugacy of unitary operators on a separable Hilbert space cannot be reduced to isomorphism of countable structures by \cite{KecSof:Strong}.  However, a complete classification  of this relation is provided by the spectral measures. We may therefore consider a more general notion of classifiability: being reducible to an orbit equivalence relation of a Polish group action. 
We don't know whether  isomorphism of separable (simple) C*-algebras is implemented 
by a Polish group action.
 (see Question~\ref{Q.action} and Problem~\ref{P.action}). 
In this section we will prove that isomorphism of \emph{unital} nuclear simple separable C$^*$-algebras is indeed Borel reducible to an orbit equivalence relation induced by a Polish group action, 
 Theorem \ref{t.belowgrp} below. We also note a simpler fact that this also applies to isomorphism 
of abelian separable C*-algebras (Proposition~\ref{P.abelian}). 
Before turning to the proof of this we briefly discuss Question \ref{Q.action} in general.


\begin{theorem} Assume $A$ and $B$ are separable weakly dense
subalgebras of $\cB(H)$. Then $A\cong B$ if and only if there is a unitary $u$ such that
$uAu^*=B$. 
\end{theorem} 

\begin{proof} Only the direct implication 
requires a proof. Let $\alpha\colon A\to B$ be an isomorphism. 
Fix a unit vector $\xi\in H$ and let $\omega_\xi$ denote the 
vector state corresponding to $\xi$, $\omega_\xi(a)=(a\xi|\xi)$. 
Since $A$ is weakly dense in $\cB(H)$, the restriction of $\omega_\xi$
to $A$ is a pure state of $A$, and similarly the restriction of $\omega_\xi$ to $B$ 
is a pure state of $B$. 
By \cite{KiOzSa} there is an automorphism $\beta$ of $B$ such that 
$\omega_\xi\rs B=\beta\circ \alpha\circ \omega_\xi$. The isomorphism 
$\beta\circ \alpha$ extends to the isomorphism between the 
GNS representations of $A$ and $B$ corresponding to the pure states 
$\omega_\xi\rs A$ 
and $\omega_\xi\rs B$. This automorphism of $\cB(H)$ is implemented by a
unitary $u$ as required. 
\end{proof} 

By the previous theorem, in order to give a positive answer to Question~\ref{Q.action}
it would suffice to have a natural Borel space whose points  are separable C*-subalgebras of $\cB(H)$. The space $\Gamma$ defined in \ref{ss.gamma} appears to be similar to such a space, but
the following proposition, suggested to the first author by Alekos Kechris, is an obstacle to the direct approach. 

\begin{prop} \label{P.nonstandard}
On the space $\Gamma$ consider  the relation $\gamma\, E\,\gamma'$ iff 
$C^*(\gamma')=C^*(\gamma)$. Then the quotient Borel structure is nonstandard, even when restricted to $\{\gamma\in \Gamma: C^*(\gamma)$ is simple, unital,  and nuclear$\}$. 
\end{prop} 

\begin{proof} Note that $E$ is Borel by Lemma~\ref{L.equality.borel}. 
It will suffice to  construct a Borel map $\Phi\colon 2^{\bbN}\to \Gamma$ such that 
$x\, E_0 \, y$ if and only if $C^*(\Phi(x))=C^*(\Phi(y))$. (Here $E_0$ denotes eventual equality in the space $2^\N$.) We will assure that every parameter
in the  range of 
$\Phi$ corresponds to a simple nuclear algebra.  
Let $H_j$ be the two-dimensional complex Hilbert space and 
let $\zeta_j$ denote the vector $\begin{pmatrix} 1 & 0 \end{pmatrix}$ in $H_j$. 
 Identify $H$ with $\bigotimes_{j\in \bbN} (H_j,\zeta_j)$. 
For $x\subseteq \bbN$ let 
\[
u_x=\bigotimes_{j\in x} 
\begin{pmatrix} 
1 & 0 \\
0 & -1.
\end{pmatrix}
\]
This is a unitary  operator on $H$. 
Fix a set $\gamma_j$, for $j\geq 1$, that generates the CAR algebra
$A=\bigotimes_j M_2(\bbC)$, represented on $H$ so that the $j$'th copy of $M_2(\bbC)$
maps to $\cB(H_j)$. 
Let $\Phi(x)=\gamma$ be such that $\gamma_0=u_x$ and $\gamma_j$, for $j\geq 1$, 
are as above.  

Then $C^*(\Phi(x))=C^*(\Phi(y))$ if and only if $u_x u_y^*\in A$, 
if and only if $x\Delta y$ is finite. 
\end{proof}

\subsection{A reduction to an action of $\Aut(\mathcal O_2)$.}

Let $\Aut(\mathcal O_2)$ denote the automorphism group of $\mathcal O_2$, and equip $\Aut(\mathcal O_2)$ with the strong topology, which makes it a Polish group. We now aim to prove:


\begin{theorem}\label{t.belowgrp}
The isomorphism relation for nuclear simple unital separable C$^*$-algebras is Borel reducible to an orbit equivalence relation induced by a Borel action of $\Aut(\mathcal O_2)$ on a standard Borel space. 
\end{theorem}

The proof of this requires some preparation, the most substantial part being a version of Kirchberg's ``$A\otimes\mathcal O_2\simeq\mathcal O_2\otimes\mathcal O_2$ Theorem'' for nuclear simple unital and separable $A$. However, we start by noting the following:


\begin{prop}\label{p.simple}
The set
$$
\{\gamma\in\Gamma: C^*(\gamma) \text{ is simple}\}
$$ 
is Borel.
\end{prop}


\begin{proof} We use the facts that  a C*-algebra $A$ is simple if and only if for 
every state $\phi$ the GNS representation 
$\pi_\phi$ is an isometry and that the operator norm in the GNS representation reads as 
$\|\pi_\phi(a)\|=\sup_{\phi(b^*b)\leq 1} \phi(b^*a^*ab)$ for all $a\in A$. 

Recall from \ref{ss.codingstates} the coding of states. We define $R\subseteq\Gamma\times\C^\N$ by
$$
R(\gamma,\hat\phi)\iff \hat\phi\text{ codes a state on } C^*(\gamma).
$$
Then $R$ is easily Borel, and as noted in \ref{ss.codingstates}, the sections $R_\gamma=\{\hat\phi\in\C^\N: R(\gamma,\hat\phi)\}$ are compact. For $n\in\N$ and $\varepsilon>0$, define $Q_{n,\varepsilon}\subseteq\Gamma\times\C^\N$ by
\begin{align*}
Q_{n,\varepsilon}(\gamma,\hat\phi)\iff  
 (\forall k)(\forall l)(\forall m) &(\p_k(\gamma)=\p_m(\gamma)^*\p_n(\gamma)^*\p_n(\gamma)\p_m(\gamma)\\
& \land \p_l(\gamma)=\p_m(\gamma)^*\p_m(\gamma)
 \land \hat\phi(l)\leq 1 \implies \hat\phi(k)+\varepsilon\leq\|\p_n(\gamma)\|).
\end{align*}
Then $Q_{n,\varepsilon}$ is Borel, and the sections $(Q_{n,\varepsilon})_\gamma$ are closed, and therefore compact. Thus the sets
$$
S_{n,\varepsilon}=\{\gamma\in\Gamma:(\exists\hat\phi) Q_{n,\varepsilon}(\gamma,\hat\phi)\}
$$
are Borel, by \cite[Theorem 28.8]{Ke:Classical}. We claim that
$$
\{\gamma\in\Gamma: C^*(\gamma)\text{ is simple}\}=\Gamma\setminus\bigcup_{n\in\N, \varepsilon>0} S_{n,\varepsilon}.
$$
To see this, first note that if $C^*(\gamma)$ is simple then the GNS representation of any state $\phi$ on $C^*(\gamma)$ is faithful, and so for any $n\in\N$ we have
\begin{equation}\label{eq.faithful}
\sup\{\phi(\p_m(\gamma)^*\p_n(\gamma)^*\p_n(\gamma)\p_m(\gamma)): 
\phi(\p_m(\gamma)\p_m(\gamma)^*\leq 1 \}
=\|\p_n(\gamma)\|
\end{equation}
Hence $\gamma\notin S_{n,\varepsilon}$ for all $n\in\N$ and $\varepsilon>0$. On the other hand, if $\gamma\notin S_{n,\varepsilon}$ for all $n\in\N$ and $\varepsilon>0$, then \eqref{eq.faithful} holds, and so all states are faithful. Hence $C^*(\gamma)$ is simple.
\end{proof}
A strengthening of Proposition~\ref{p.simple} will be given in \cite{FaToTo}. 

\medskip

Since Effros has shown that the class of nuclear separable C$^*$-algebras is Borel (see \cite[\S 5]{Kec:C*} for a proof), we now have:

\begin{corollary}\label{c.unsBorel}
The set
$$
\{\gamma\in\Gamma: C^*(\gamma)\text{ is simple, nuclear and unital}\}
$$
is Borel.
\end{corollary}


\subsection{A Borel version of Kirchberg's $A\otimes \mathcal O_2$ Theorem}

For $A$ and $B$ fixed separable C$^*$-algebras, let 
$$
\Hom(A,B)=\{f:A\to B: f \text{ is a $*$-homomorphism}\}.
$$
Then $\Hom(A,B)\subseteq L_1(A,B)$, the set of bouned linear maps from $A$ to $B$ with operator norm at most 1, and is closed in the strong operator topology, hence is a Polish space. We let $\End(A)=\Hom(A,A)$.

Kirchberg's $A\otimes\mathcal O_2$ Theorem states that $A$ is nuclear simple separable and unital if and only if $A\otimes\mathcal O_2$ is isomorphic to $\mathcal O_2\otimes\mathcal O_2$. The latter is itself isomorphic to $\mathcal O_2$ by a result of Elliott, see e.g. \cite[7.1.2 and 5.2.1]{Ror:Classification}. Our next theorem is an effective version of this theorem. 

Let $\SAu(\mathcal O_2)$ denote the standard Borel space of closed unital $*$-subalgebras of $\mathcal O_2$. Since the parameterizations $\Gammaexu$ and $\SAu(\mathcal O_2)$ are weakly equivalent (see \ref{ss.exact}), it follows from Corollary~\ref{c.unsBorel} that the set
$$
\SAuns(\mathcal O_2)=\{A\in \SAu(\mathcal O_2): A\text{ is nuclear and simple}\}
$$
is Borel. We will work with this parameterization of unital nuclear simple separable C$^*$-algebras below.

\begin{theorem}\label{t.AtensorO2}
There is a Borel map $F:\SAuns(\mathcal O_2)\to \mathrm{End}(\mathcal O_2\otimes\mathcal O_2)$ such that $F(A)$ is a monomorphism of $\mathcal O_2\otimes\mathcal O_2$ onto $A\otimes\mathcal O_2$.
\end{theorem}

The proof uses an approximate intertwining argument\footnote{We refer the reader to \cite[2.3]{Ror:Classification} for a general discussion of approximate intertwining.  The argument is also known as Elliott's Intertwining Argument.} that we now describe. Let $A$ be a simple unital separable nuclear C$^*$-algebra, viewed as a unital subalgebra of $\mathcal{O}_2$.  Recall that for such $A$, the algebra $A \otimes \mathcal{O}_2$ (and so in particular $\mathcal{O}_2 \otimes \mathcal{O}_2 \cong \mathcal{O}_2$) has the property that every unital $*$-endomorphism is approximately inner (see \cite[6.3.8]{Ror:Classification}.)  Fix a $*$-isomorphism $\gamma: \mathcal{O}_2 \otimes \mathcal{O}_2 \to \mathcal{O}_2$ and a summable sequence $(\epsilon_n)$ of strictly positive tolerances.  We will apply Elliott's Intertwining Argument to the {\it a priori} non-commuting diagram
\[
\xymatrix{
{A \otimes \mathcal{O}_2}\ar[r]^{\mathbf{id}}\ar[d]^{\iota} & 
{A \otimes \mathcal{O}_2}\ar[r]^{\mathbf{id}}\ar[d]^{\iota} &
{A \otimes \mathcal{O}_2}\ar[r]^{\mathbf{id}}\ar[d]^{\iota} & \cdots \\
{\mathcal{O}_2 \otimes \mathcal{O}_2}\ar[r]^{\mathbf{id}}\ar[ur]^{\eta} &
{\mathcal{O}_2 \otimes \mathcal{O}_2}\ar[r]^{\mathbf{id}}\ar[ur]^{\eta} &
{\mathcal{O}_2 \otimes \mathcal{O}_2}\ar[r]^{\mathbf{id}}\ar[ur]^{\eta} & \cdots
}
\]
where $\iota$ is the tensor product of the inclusion $A \hookrightarrow \mathcal{O}_2$ with the identity map on $\mathcal{O}_2$ and $\eta$ is given by $a \mapsto 1_A \otimes \gamma(a)$.  Let us describe the procedure step-by-step, so that we may refer back to this description when arguing that the intertwining can be carried out effectively.  \vspace{2mm}

Fix a dense sequence $(x_n^A)$ in $A \otimes \mathcal{O}_2$ and a dense  sequence $(y_n)$ in $\mathcal{O}_2 \otimes \mathcal{O}_2$. Fix also a dense sequence of unitaries $u_n^A\in A\otimes\mathcal O_2$ and a dense sequence of unitaries $v_n\in\mathcal O_2\otimes\mathcal O_2$. We assume that $u_1^A=1_{A\otimes\mathcal O_2}$ and $v_1=1_{\mathcal O_2\otimes\mathcal O_2}$. We define by recursion a sequence of finite sets $F_k^A\subseteq A$ and $G_k^A\subseteq\mathcal O_2\otimes\mathcal O_2$, sequences $(n_k^A)_{k\in\N}$ and $(m_k^A)_{k\in\N}$ of natural numbers, and homomorphisms $\iota_k:A\otimes \mathcal O_2\to\mathcal O_2\otimes\mathcal O_2$, $\eta_k:\mathcal O_2\otimes\mathcal O_2\to A\otimes\mathcal O_2$ subject to the following conditions: $n_1^A=1$, $m_1^A=1$, $F_1^A=G_1^A=\emptyset$, $\iota_1^A=\iota$, $\eta_1^A=\eta$, and for $k>1$ we require that

\begin{enumerate}
\item $F_k^A=\{x_{k-1}^A\}\cup F_{k-1}^A\cup \eta_{k-1}(G_{k-1}^A)$.
\item $G_k^A=\{y_{k-1}\}\cup G_{k-1}^A\cup\iota_{k-1}(F_k^A)$.
\item $n_k^A> n_{k-1}^A$ is least such that if we let $\eta_k^A=\Ad(u^A_{n_k^A})\circ\eta$ then the diagram
$$
\xymatrix{
{A \otimes \mathcal{O}_2}\ar[r]^{\mathbf{id}}\ar[d]^{\iota_{k-1}^A} & 
{A \otimes \mathcal{O}_2}\\
{\mathcal{O}_2 \otimes \mathcal{O}_2}\ar[ur]^{\eta_k^A} &\\
}
$$
commutes up to $\epsilon_k$ on $F_k^A$.  This is possible because any two endomorphisms of $A \otimes \mathcal{O}_2$ are approximately unitarily equivalent and the sequence $(u^A_n)$ is dense in the unitaries of $A \otimes \mathcal{O}_2$.
\item $m_k^A>m_{k-1}^A$ is least such that if we let $\iota_k^A=\Ad(v_{m_k})\circ\iota$ then the diagram
$$
\xymatrix{
 & 
{A \otimes \mathcal{O}_2}\ar[d]^{\iota_k^A}\\
{\mathcal{O}_2 \otimes \mathcal{O}_2}\ar[r]^{\mathbf{id}}\ar[ur]^{\eta_k^A} &
{\mathcal{O}_2 \otimes \mathcal{O}_2}
}
$$
commutes up to $\epsilon_k$ on $G_k^A$.  This is possible because any two endomorphisms of $\mathcal{O}_2 \otimes \mathcal{O}_2$ are approximately unitarily equivalent and the sequence $(v_n)$ is dense in the unitaries of $\mathcal{O}_2 \otimes \mathcal{O}_2$.

\end{enumerate}

With these definitions the diagram
$$
\xymatrix{
{A \otimes \mathcal{O}_2}\ar[r]^{\mathbf{id}}\ar[d]^{\iota_1^A} & 
{A \otimes \mathcal{O}_2}\ar[r]^{\mathbf{id}}\ar[d]^{\iota_2^A} &
{A \otimes \mathcal{O}_2}\ar[r]^{\mathbf{id}}\ar[d]^{\iota_3^A} & \cdots \\
{\mathcal{O}_2 \otimes \mathcal{O}_2}\ar[r]^{\mathbf{id}}\ar[ur]^{\eta_2^A} &
{\mathcal{O}_2 \otimes \mathcal{O}_2}\ar[r]^{\mathbf{id}}\ar[ur]^{\eta_3^A} &
{\mathcal{O}_2 \otimes \mathcal{O}_2}\ar[r]^{\mathbf{id}}\ar[ur]^{\eta_4^A} & \cdots
}
$$
is an approximate intertwining (in the sense of \cite[2.3.1]{Ror:Classification}), and so
$$
\eta_\infty^A:\mathcal O_2\otimes\mathcal O_2\to A\otimes\mathcal O_2:\eta^A_\infty(b)=\lim_{k\to\infty}\eta_k(b)
$$
defines an isomorphism.

\begin{proof}[Proof of Theorem \ref{t.AtensorO2}]
Fix a dense sequence $z_n\in\mathcal O_2$. Also, let $y_n, v_n\in\mathcal O_2\otimes\mathcal O_2$ be as above. By the Kuratowski-Ryll-Nardzewski Selection Theorem we can find Borel maps $f_n:\SAu(\mathcal O_2)\to\mathcal O_2$ such that $(f_n(A))_{n\in\N}$ is a dense sequence in $A$. Let $\pi:\mathbb{N} \to \mathbb{N}^2$ be a bijection with $\pi(n) = (\pi_1(n),\pi_2(n))$.  Associating the $\mathbb{Q} + i \mathbb{Q}$ span of $\{ f_{\pi_1(n)}(A) \otimes z_{\pi_2(n)} \ | \ n \in \mathbb{N} \}$ to $A$ is clearly Borel, and this span is dense;  let us denote it by $x_n^A$.   From (the proof of) Lemma \ref{L.E.1}.(\ref{L.E.1.4}), we obtain a sequence of Borel maps $\Un_k:\SAu(\mathcal O_2)\to \mathcal O_2\otimes\mathcal O_2$ such that $\Un_k(A)$ is dense in the set of unitaries in $A\otimes \mathcal O_2$. We let $u_k^A=\Un_k(A)$.

With these definitions there are unique Borel maps $A\mapsto F_k^A$, $A\mapsto G_k^A$, $A\mapsto n_k^A$ and $A\mapsto m_k^A$ satisfying (1)--(4) above; in particular, if these maps have been defined for $k=l-1$, and $F_l^A$ is defined, then $n_l^A$ is defined as the least natural number $n$ greater than $n_k^A$ such that 
$$
(\forall a\in F_k^A) \|\Ad(u^A_{n} u^A_{n_{l-1}} \cdots u^A_{n_1})\circ\eta\circ \Ad(v_{m_{l-1}}\cdots v_{m_1})\circ\iota(a)-a\|_{\mathcal O_2}<\varepsilon_k.
$$
Thus the graph of $A\mapsto n_l^A$ is Borel. Similarly, $A\mapsto m_k^A$ is seen to be Borel for all $k\in\N$. But now we also have that the map $F:\SAu(\mathcal O_2)\to\mathrm{End}(\mathcal O_2\otimes\mathcal O_2):A\mapsto\eta_\infty^A$ is Borel, since
$$
F(A)=\eta_\infty^A\iff (\forall l) F(A)(y_l)=\lim_{k\to\infty} \Ad(u_{n_k}u_{n_{k-1}}\cdots u_1)\circ\eta(y_l)
$$
provides a Borel definition of the graph of $F$.
\end{proof}

\begin{proof}[Proof of Theorem \ref{t.belowgrp}]
The Polish group $\Aut(\mathcal O_2)$ acts naturally in a Borel way on $\SAu(\mathcal O_2)$ by $\sigma\cdot A=\sigma(A)$. Let $E$ be the corresponding orbit equivalence relation. We claim that isomorphism in $\SAuns(\mathcal O_2)$ is Borel reducible to $E$.

By the previous Theorem there is a Borel map $g:\SAuns(\mathcal O_2)\to\SAuns(\mathcal O_2)$ with the following properties:
\begin{itemize}
\item $g(A)\cong A$.
\item For all $A\in\SAuns(\mathcal O_2)$ there is an isomorphism $G_A:A\otimes\mathcal O_2\to\mathcal O_2$ under which\\ $G_A(A\otimes 1_{\mathcal O_2})=g(A)$. 
\end{itemize}
In other words, there is an effective unital embedding of nuclear unital simple separable C$^*$-algebras into $\mathcal{O}_2$ with the property that the relative commutant of the image of any such algebra in $\mathcal{O}_2$ is in fact isomorphic to $\mathcal{O}_2$.  We claim that $g$ is a Borel reduction of $\cong^{\SAuns(\mathcal O_2)}$ to $E$.

Fix $A,B\in\SAuns(\mathcal O_2)$. Clearly if $g(A) E g(B)$ then $A\cong B$. On the other hand, if $A\cong B$ then there is an isomorphism $\varphi:A\otimes \mathcal O_2\to B\otimes\mathcal O_2$ which maps $A\otimes 1_{\mathcal O_2}$ to $B\otimes 1_{\mathcal O_2}$. Thus
$$
\sigma=G_B\circ\varphi\circ G_A^{-1}\in\Aut(\mathcal O_2)
$$
satisfies $\sigma\cdot A=B$.
\end{proof}

Since by Corollary \ref{c.reduction} it holds that the homeomorphism relation for compact subsets of $[0,1]^\N$ is Borel reducible to isomorphism of nuclear simple unital AI algebras, we recover the following unpublished result of Kechris and Solecki:

\begin{theorem}[Kechris-Solecki]\label{t.kechrissolecki}
The homeomorphism relation for compact subsets of $[0,1]^\N$ is below a group action.
\end{theorem}

Note that the set $A=\{\gamma\in \Gamma: C^*(\gamma)$ is abelian$\}$ is
Borel since  $A$ is clearly closed in the weak operator topology in $\Gamma$. 
As a subspace of $\Gamma$ it therefore provides a good standard Borel parameterization
for abelian C*-algebras, used in the following. 

\begin{prop}\label{P.abelian} 
The isomorphism relation for unital 
abelian separable C*-algebras is Borel reducible to isomorphism of AI algebras, and therefore to an orbit equivalence relation induced by a Polish group action. 
\end{prop} 

\begin{proof} 
For $\gamma\in A$ we have that 
 $C^*(\gamma)\cong C(X)$ where $X$ is the pure state space of $C^*(\gamma)$. 
The result now follows by Lemma~\ref{L.SPT} and Corollary \ref{c.reduction}. An alternative proof of the last claim appeals to Theorem~\ref{t.kechrissolecki} instead. 
\end{proof}

\section{Bi-embeddability of AF algebras}\label{S.biembed}

In this section we will show that the bi-embeddability relation of separable unital AF algebras is not Borel-reducible to a Polish group action (Corollary~\ref{C.Biembed}). 
More precisely, we prove that every $K_\sigma$-equivalence relation is Borel reducible to this analytic equivalence relation. (Recall that a subset of  a Polish space is $K_\sigma$ if it is a countable union of compact sets.) 
This Borel reduction is curious since bi-embeddability of separable unital UHF algebras is bi-reducible
with the isomorphism of separable unital UHF algebras, and therefore smooth.   
For $f$ and $g$ in the Baire space $\bbN^\bbN$ we define
\begin{align*}
f\leq^\infty g&  \text{ if and only if } (\exists m)(\forall i) f(i)\leq g(i)+m\\
f=^\infty g & \text{ if and only if } f\leq^\infty g \text{ and } g\leq^\infty f
\end{align*}
This equivalence relation,  also denoted $E_{K_\sigma}$, was introduced by Rosendal in  \cite{Ro:Cofinal}. Rosendal proved that $E_{K_\sigma}$ is \emph{complete} for $K_\sigma$ equivalence relation in the sense that (i) every $K_\sigma$ equivalence relation is Borel reducible to it and (ii) $E_{K_\sigma}$ is itself $K_\sigma$. By  a result of Kechris and Louveau (\cite{KeLou:Structure}), $E_{K_\sigma}$ is not Borel reducible to  any orbit equivalence relation of a Polish group action. In particular, $E_{K_\sigma}$, or any analytic equivalence relation that Borel-reduces $E_{K_\sigma}$, is not effectively classifiable by 
the Elliott invariant.

If $A$ and $B$ are C*-algebras then we denote by $A\hookrightarrow B$ the existence of a $*$-monomorphism of $A$ into $B$;   $A$ is therefore bi-embeddable with $B$ if and only if $A\hookrightarrow B$ and $B\hookrightarrow A$.


\begin{prop} There is a Borel-measurable map $ \bbN^\bbN\to\Gamma:f\mapsto C_f$
such that 
\begin{enumerate}
\item each $C_f$ is a unital AF algebra;
\item $C_f$ isomorphically embeds into $C_g$ if and only if $f\leq^\infty g$; 
\item $C_f$ is bi-embeddable with $C_g$ if and only if $f=^\infty g$.
\end{enumerate}
\end{prop}

  
\begin{proof} 
We first describe the construction and then verify it is Borel. Let $p_i$, $i\in \bbN$, be the increasing enumeration of all primes. For $f\in \bbN^\bbN$ and $n\in \bbN$ define UHF algebras 
\begin{align*} 
  A_f& = \bigotimes_{i=1}^\infty M_{p_i^{f(i)}}(\bbC)\qquad\text{and}\qquad
  B_n  =\bigotimes_{i=1}^\infty M_{p_i^n}(\bbC). 
\end{align*}
Hence $B_n$ is isomorphic to $A_f$ if $f(i)=n$ for all $i$. For $f$ and $g$ we have that $A_f\hookrightarrow A_g$ if and only if $f(i)\leq g(i)$ for all $i$. Also, $A_f\hookrightarrow A_g\otimes B_n$ if and only if $f(i)\leq g(i)+n$ for all $i$. Therefore, $f\leq^\infty g$ if and only if $A_f\hookrightarrow A_g\otimes B_n$ for a large enough $n$. Let $C_f$ be the unitization of 
$
  A_f\otimes \bigoplus_{n=1}^\infty B_n. 
$
We claim that $f\leq^\infty g$ if and only if $C_f\hookrightarrow C_g$.

First assume   $f\leq^\infty g$ and let $n$ be such that $f(i)\leq g(i)+n$ for all $i$. Then by the above $A_f\otimes B_m \hookrightarrow A_g\otimes B_{m+n}$ for all $m$, and therefore $C_f\hookrightarrow C_g$. 

Now assume $C_f\hookrightarrow C_g$. Then in particular $A_f\hookrightarrow \bigoplus_{n=1}^\infty A_g\otimes B_n$. Since $A_f$ is simple, we have $A_f\hookrightarrow A_g\otimes B_n$ for some $n$ and therefore $f\leq^\infty g$. We have therefore proved that the map $f\mapsto C_f$ satisfies (2). Clause (3) follows immediately.

It remains to find a Borel measurable map $\Phi\colon \bbN^\bbN\to \Gamma$ such that $C^*(\Phi(f))$ is isomorphic to $C_f$ for all $f$. Since $\otimes$ (for nuclear C*-algebras) and $\bigoplus$ are Borel (by Lemma~\ref{L.B.3.0} for the former;  the latter is trivial), it suffices to show that there is a Borel map $\Psi\colon \bbN^\bbN\to \Gamma$ such that $C^*(\Psi(f))$ is isomorphic to $A_f$ for all $f$. 

Let $D$ denote the maximal separable UHF algebra, $\bigotimes_{i=1}^\infty \bigotimes_{n=1}^\infty M_{p_i}(\bbC)$. Let $\phi$ be its unique trace and let $\pi_\phi\colon D\to \cB(H_\phi)$
be the GNS representation corresponding to $\phi$. Then $H_\phi$ is a
tensor product of finite-dimensional Hilbert spaces $H_{n,i}$ 
such that $\dim(H_{n,i})=i^2$. Also, for each pair $n,i$ there is an isomorphic copy $D_{n,i}$ of $M_i(\bbC)$ acting on $H_{n,i}$ and a unit vector $\xi_{n,i}$ such that $\omega_{\xi_{n,i}}$ agrees with the normalized trace on $D_{n,i}$. 
The algebra generated by $D_{n,i}$, for $n,i\in\N$, is isomorphic to $D$. 

Now identify $H_\phi$ with $H$ as used to define $\Gamma$. Each $D_{n,i}$ is singly generated, so we can fix a generator $\gamma_{n,i}$. Fix a bijection $\chi$ between $\bbN$ and $\bbN^2$, and write $\chi(n)=(\chi_0(n), \chi_1(n))$. For $f\in \bbN^\bbN$ let $\Psi(f)=\gamma$ be defined by $\gamma_n=0$ if $f(\chi_1(n))<\chi_0(n)$ and $\gamma_n=\gamma_{\chi_0(n),\chi_1(n)}$ if $f(\chi_1(n))\geq \chi_0(n)$. Then $C^*(\gamma)$ is isomorphic to the tensor product of the $D_{n,i}$ for $n\leq f(i)$, which is in turn isomorphic to $A_f$. Moreover, the map $f\mapsto \Psi(f)$ is continuous when $\Gamma$ is considered with the
product topology, because finite initial segments of $\Psi(f)$ are determined by finite initial segments of $f$.
\end{proof} 


\begin{corollary} \label{C.Biembed} 
$E_{K_\sigma}$ is Borel reducible to the bi-embeddability relation $E$ on separable AF C*-algebras. Therefore $E$ is not Borel reducible to a Polish group action. \qed
\end{corollary}

\section{Concluding remarks and open problems}\label{S.problems}

In this section we discuss several open problems and possible directions for further investigations related to the theme of this paper.  The first is related to Section \ref{S.biembed}.

\begin{problem}\label{nucbiembed}
Is the bi-embeddability relation for nuclear simple separable C$^*$-algebras a complete analytic equivalence relation? What about bi-embeddability of AF algebras?
\end{problem}
\noindent
We remark that the bi-embeddability, and even isomorphism, of separable Banach spaces is known to be complete for analytic equivalence relation (\cite{FeLouRo}). Moreover, bi-embeddability of countable graphs is already complete for analytic equivalence relations by \cite{LouRo}


\begin{quest} \label{Q.action} 
Is isomorphism of separable (simple) C*-algebras implemented 
by a Polish group action? 
\end{quest}

George Elliott observed that 
the isomorphism 
of nuclear simple separable C*-algebras is Borel reducible to an orbit equivalence relation 
induced by a Borel action of the automorphism group of $\cO_2\otimes \cK$. 
This is proved by an extension of the proof of Theorem~\ref{t.belowgrp} together 
with Borel version of 
Kirchberg's result that $A\otimes \cO_2\otimes\cK$ is isomorphic to $\cO_2\otimes \cK$
for every nuclear simple separable C*-algebra $A$. The  following is an extension of Question~\ref{Q.action}. 

\begin{problem}  \label{P.action} 
What is the Borel cardinality of the isomorphism relation of 
larger classes of separable C*-algebras (simple or not), such as:
\begin{enumerate}
\item[(i)] nuclear C*-algebras;
\item[(ii)] exact C*-algebras;
\item[(iii)] arbitrary C*-algebras? 
\end{enumerate}
Do these problems have strictly increasing Borel cardinality? 
Are all of them Borel reducible to the orbit equivalence relation of a Polish group action on a standard Borel space? 
\end{problem}

\noindent
We can ask still more of the classes in Problem \ref{P.action}.
On the space $\Delta^{\bbN}$ (recall that $\Delta$ is the closed unit disk in $\mathbb{C}$) define the relation $E_1$ by letting $x\, E_1\, y$ if 
$x(n)=y(n)$ for all but finitely many $n$. 
In \cite{KeLou:Structure} it was proved that $E_1$ is not Borel-reducible 
to any orbit equivalence relation of a Polish group action, and therefore $E_1\leq_B E$ 
implies $E$ is not  Borel-reducible to 
an orbit equivalence relation of a Polish group action. 
Kechris and Louveau
have  even conjectured that for Borel equivalence relations 
reducing $E_1$ is equivalent to not being induced by a Polish group action. 
While Corollary~\ref{c.reduction} implies that 
 the relations considered in Problem~\ref{P.action} are not Borel, 
it is natural to expect that  they either reduce~$E_1$ or can be reduced to
an orbit equivalence relation of a Polish group action.

We defined several Borel parameterizations (see Definition~\ref{d.parameterization}) 
of separable C*-algebras that were subsequently shown to be equivalent. 
The phenomenon that all natural  Borel parameterizations of a given classification problem
seem to wind up being equivalent has been observed by other authors. 
One may ask the following general question.


\begin{problem} Assume $\Gamma_1$ and $\Gamma_2$ are good standard 
Borel parameterizations that model isomorphism of structures in the same category $\cC$. 

Find optimal assumptions that guarantee the existence of a Borel-isomorphism $\Phi\colon \Gamma_1\to \Gamma_2$ that preserves the isomorphism in class $\cC$ in the sense that 
\[
A\cong B\text{ if and only if } \Phi(A)\cong \Phi(B). 
\]
\end{problem} 
\noindent
A theorem of this kind could be regarded as an analogue of an automatic continuity theorem.


We have addressed many basic C*-algebra constructions here and proved that they are Borel.  We have further proved that various natural subclasses of separable C*-algebras are Borel.  There is, however, much more to consider.

\begin{problem}\label{classes} Determine whether the following C*-algebra constructions and/or subclasses are Borel:
\begin{enumerate}
\item[(i)] the maximum tensor product, and tensor products more generally;
\item[(ii)] crossed products, full and reduced;
\item[(iii)] groupoid C*-algebras;
\item[(iv)] $\mathcal{Z}$-stable C*-algebras;
\item[(v)] C*-algebras of finite or locally finite nuclear dimension;
\item[(vi)] approximately subhomogeneous (ASH) algebras;
\item[(vii)] the Thomsen semigroup of a C$^*$-algebra.
\end{enumerate}
\end{problem}

\noindent
Items (iv)--(vi) above are of particular interest to us as they are connected to Elliott's classification program (see \cite{et}).  
Items (iv) and (v) are connected to the radius of comparison by the following conjecture of Winter and the second author.
\begin{conj}  Let $A$ be a simple separable unital nuclear C*-algebra.  The following are equivalent:
\begin{enumerate}
\item[(i)] $A$ has finite nuclear dimension;
\item[(ii)] $A$ is $\mathcal{Z}$-stable;
\item[(iii)] $A$ has radius of comparison zero.
\end{enumerate}
\end{conj}

\noindent
We have shown here that simple unital nuclear separable C$^*$-algebras form a Borel set. 
We will show in a forthcoming article that separable C*-algebras with radius of comparison zero also form a Borel set.  Thus, those $A$ as in the conjecture which satisfy (iii) form a Borel set.  It would be interesting to see if the same is true if one asks instead for (i) or (ii).  
Clearly, the $\cZ$-stable algebras form an analytic set. 
As for item (vi) of Problem \ref{classes}, the question of whether every unital simple separable nuclear C$^*$-algebra with a trace is ASH has been open for some time.  One might try to attack this question by asking where these formally different classes of algebras sit in the Borel hierarchy.  

In \cite{Ell:Towards}, Elliott introduced an abstract approach to  functorial classification. 
A feature of his construction is that  morphisms between classifying invariants lift
to morphisms between the objects to be classified.
This  property is shared with the classification of C*-algebras, where  
morphisms between K-theoretic invariants lift to (outer, and typically not unique)
automorphisms of the original objects. 
It would be interesting to have a set-theoretic analysis of this phenomenon
parallel to the set-theoretic analysis of abstract classification problems
used in the present paper. 
 
\begin{problem} Is there a set-theoretic model for the functorial inverse in the sense of 
Elliott? More precisely, if the categories are modelled by Borel spaces and the 
functor that assigns invariants is Borel, is the inverse functor necessarily Borel? 
\end{problem}

Finally, the following question was posed by Greg Hjorth to the third author:

\begin{problem}
Is $\simeq^{{\bf Ell}}$ Borel reducible to $\simeq^{\Lambda}$? That is, does isomorphism of Elliott invariants Borel reduce to affine isomorphism of separable metrizable Choquet simplexes?
\end{problem}

We will show in a forthcoming article that at least the Elliott invariant is below a group action. 

 It would also be natural to try to obtain an answer to the following:

\begin{problem}\label{Pr.Choquet}
Is isomorphism of separable metrizable Choquet simplexes Borel reducible to homeomorphism of compact Polish spaces? I.e., is $\simeq^\Lambda$ Borel reducible to $\simeq^{\K}_{{\rm homeo}}$?
\end{problem}

In connection to this problem we should point out a misstatement in \cite{Hj:Borel}.
In \cite[p. 326]{Hj:Borel} it was stated that Kechris and Solecki have proved that 
the homeomorphism 
of compact Polish spaces is Borel bi-reducible to $E^{X_\infty}_{G_\infty}$. 
The latter is  
an  orbit equivalence relation of a Polish group action with the property that every 
other orbit equivalence relation of a Polish group action is Borel-reducible to it. 
If true, this would give a positive solution to   Problem~\ref{Pr.Choquet}. 
However, Kechris and Solecki have proved only one direction, referred to in Theorem~\ref{T.KS}. 
 The other direction is still open.

\bibliographystyle{amsplain}
\bibliography{aabib}

\providecommand{\bysame}{\leavevmode\hbox to3em{\hrulefill}\thinspace}
\providecommand{\MR}{\relax\ifhmode\unskip\space\fi MR }
\providecommand{\MRhref}[2]{%
  \href{http://www.ams.org/mathscinet-getitem?mr=#1}{#2}
}
\providecommand{\href}[2]{#2}
\begin{thebibliography}{10}

\bibitem{alfsen71}
E.~M. Alfsen, \emph{Compact {C}onvex {S}ets and {B}oundary {I}ntegrals},
  Springer-Verlag Berlin, 1971.

\bibitem{blackadar}
B.~Blackadar, \emph{Operator {A}lgebras}, Encyclopaedia of Math. Sciences, vol.
  122, Springer-Verlag, 2006.

\bibitem{Ell:Towards}
G.~A. Elliott, \emph{Towards a theory of classification}, Advances in
  Mathematics \textbf{223} (2010), no.~1, 30--48.

\bibitem{et}
G.~A. Elliott and A.~S. Toms, \emph{Regularity properties in the classification
  program for separable amenable {C}$^*$-algebras}, Bull. Amer. Math. Soc.
  \textbf{45} (2008), 229--245.

\bibitem{Fa:Dichotomy}
I.~Farah, \emph{A dichotomy for the {M}ackey {B}orel structure}, Proceedings of
  the Asian Logic Colloquium 2009 (Yang Yue et~al., eds.), to appear.

\bibitem{FaToTo}
I.~Farah, A.~Toms, and A.~T\"ornquist, \emph{The descriptive set theory of
  {C*}-algebra invariants}, preprint.

\bibitem{frst89}
H.~Friedman and L.~Stanley, \emph{A {B}orel reducibility theory for classes of
  countable structures}, The Journal of Symbolic Logic \textbf{54} (1989),
  894--914.

\bibitem{gao09}
S.~Gao, \emph{Invariant {D}escriptive {S}et {T}heory}, Pure and Applied
  Mathematics (Boca Raton), vol. 293, CRC Press, Boca Raton, FL, 2009.

\bibitem{hjorth00}
G.~Hjorth, \emph{Classification and {O}rbit {E}quivalence {R}elations},
  Mathematical Surveys and Monographs, vol.~75, American Mathematical Society,
  2000.

\bibitem{Hj:Borel}
\bysame, \emph{Borel {E}quivalence {R}elations}, Handbook of Set Theory (2010),
  297--332.

\bibitem{JunPis}
M.~Junge and G.~Pisier, \emph{Bilinear forms on exact operator spaces and
  {$B(H)\otimes B(H)$}}, Geom. Funct. Anal. \textbf{5} (1995), no.~2, 329--363.

\bibitem{KecSof:Strong}
A.~S. Kechris and N.~E. Sofronidis, \emph{A strong generic ergodicity property
  of unitary and self-adjoint operators}, Ergodic Theory Dynam. Systems
  \textbf{21} (2001), no.~5, 1459--1479.

\bibitem{Ke:Classical}
A.S. Kechris, \emph{Classical {D}escriptive {S}et {T}heory}, Graduate texts in
  mathematics, vol. 156, Springer, 1995.

\bibitem{Kec:C*}
\bysame, \emph{The descriptive classification of some classes of {$C\sp
  *$}-algebras}, Proceedings of the Sixth Asian Logic Conference (Beijing,
  1996), World Sci. Publ., River Edge, NJ, 1998, pp.~121--149.

\bibitem{KeLou:Structure}
A.S. Kechris and A.~Louveau, \emph{The structure of hypersmooth {B}orel
  equivalence relations}, Journal of the American Mathematical Society
  \textbf{10} (1997), 215--242.

\bibitem{kelipi}
D.~Kerr, H.~Li, and M.~Pichot, \emph{Turbulence, representations, and
  trace-preserving actions}, Proc. London Math. Soc. (3) \textbf{100} (2010),
  no.~2, 459--484.

\bibitem{KiOzSa}
A.~Kishimoto, N.~Ozawa, and S.~Sakai, \emph{Homogeneity of the pure state space
  of a separable {$C\sp *$}-algebra}, Canad. Math. Bull. \textbf{46} (2003),
  no.~3, 365--372.

\bibitem{LaLi}
A.~J. Lazar and J.~Lindenstrauss, \emph{Banach spaces whose duals are {$L_{1}$}
  spaces and their representing matrices}, Acta Math. \textbf{126} (1971),
  165--193.

\bibitem{LouRo}
A.~Louveau and C.~Rosendal, \emph{Complete analytic equivalence relations},
  Trans. Amer. Math. Soc. \textbf{357} (2005), no.~12, 4839--4866.

\bibitem{moschovakis80}
Y.N. Moschovakis, \emph{Descriptive {S}et {T}heory}, Studies in Logic and the
  Foundations of Mathematics, vol. 100, North-Holland Publishing Company,
  Amsterdam, 1980.

\bibitem{Pede:Analysis}
G.K. Pedersen, \emph{Analysis {N}ow}, Graduate Texts in Mathematics, vol. 118,
  Springer-Verlag, New York, 1989.

\bibitem{Ror:Classification}
M.~R{\o}rdam, \emph{Classification of {N}uclear {$C\sp *$}-algebras},
  Encyclopaedia of Math. Sciences, vol. 126, Springer-Verlag, Berlin, 2002.

\bibitem{Ror}
\bysame, \emph{A simple {C}$^*$-algebra with a finite and an infinite
  projection}, Acta Math. \textbf{191} (2003), 109--142.

\bibitem{Ro:Cofinal}
C.~Rosendal, \emph{Cofinal families of {B}orel equivalence relations and
  quasiorders}, Journal of Symbolic Logic \textbf{70} (2005), 1325--1340.

\bibitem{sato09b}
R.~Sasyk and A.~T{\"o}rnquist, \emph{Borel reducibility and classification of
  von {N}eumann algebras}, Bulletin of Symbolic Logic \textbf{15} (2009),
  no.~2, 169--183.

\bibitem{sato09a}
\bysame, \emph{The classification problem for von {N}eumann factors}, Journal
  of Functional Analysis \textbf{256} (2009), 2710--2724.

\bibitem{sato09c}
\bysame, \emph{Turbulence and {A}raki-{W}oods factors}, J. Funct. Anal.
  \textbf{259} (2010), no.~9, 2238--2252. \MR{2674113}

\bibitem{thomsen}
K.~Thomsen, \emph{Inductive limits of interval algebras: the tracial state
  space}, Amer. J. Math. \textbf{116} (1994), no.~3, 605--620.

\bibitem{Toms}
A.~S Toms, \emph{On the classification problem for nuclear {C}$^*$-algebras},
  Ann. of Math. (2) \textbf{167} (2008), 1059--1074.

\bibitem{FeLouRo}
A.~Louveau V.~Ferenczi and C.~Rosendal, \emph{The complexity of classifying
  separable {B}anach spaces up to isomorphism}, J. Lond. Math. Soc. (2)
  \textbf{79} (2009), no.~2, 323--345.

\bibitem{Vill:Range}
J.~Villadsen, \emph{The range of the {E}lliott invariant}, J. Reine Angew.
  Math. \textbf{462} (1995), 31--55.

\end{thebibliography}

\end{document}